\documentclass[11pt,reqno]{article}
\usepackage{bm}
\usepackage{amsthm}
\usepackage{amsmath}
\usepackage{amssymb}
\usepackage[round]{natbib}
\usepackage{color}
\usepackage{authblk}
\usepackage{microtype}
\usepackage{graphicx}
\usepackage{float}
\usepackage{prodint}
\usepackage{paralist}
\usepackage{epstopdf}
\usepackage{geometry}
\DeclareMathSymbol{\shortminus}{\mathbin}{AMSa}{"39}
\newcommand{\rnc}{\renewcommand}
\newcommand{\nc}{\newcommand}
\newcommand{\mrm}{\mathrm}

\nc{\mb}{\mathbb}
\nc{\mc}{\mathcal}
\nc{\E}{\mb{E}}
\nc{\N}{\mb{N}}
\nc{\R}{\mb{R}}
\nc{\Q}{\mb{Q}}
\rnc{\P}{\mrm P}
\rnc{\d}{\mrm d}
\nc{\C}{\mc{C}}
\nc{\D}{\mc{D}}
\nc{\B}{\mc{B}}
\nc{\vbeta}{\bm \beta}
\nc{\vtheta}{\bm \theta}
\nc{\vX}{\bm X}
\nc{\vy}{\bm y}
\nc{\vU}{\bm U}
\nc{\vI}{\bm I}
\nc{\vE}{\bm E}
\nc{\ve}{\bm e}
\nc{\vV}{\bm V}
\nc{\vv}{\bm v}
\nc{\vS}{\bm S}
\nc{\vSigma}{\bm \Sigma}
\nc{\oPo}{\stackrel{\mrm p}{\rightarrow}}
\nc{\oWo}{\stackrel{w}{\rightarrow}}
\nc{\oDo}{\stackrel{d}{\longrightarrow}}
\nc{\eff}{\|F\|}

\def\E{{ E }}
\def\R{{ \mathbb{R} }}
\def\N{{ \mathbb{N} }}

\def\P{ P }

\def\E{ E }

\newtheorem{lemma}{Lemma}
\newtheorem{assump}{Assumption}
\newtheorem{theorem}{Theorem}
\newtheorem{prop}{Proposition}

\newtheorem{remark}{Remark}

\newcommand\blfootnote[1]{%
  \begingroup
  \renewcommand\thefootnote{}\footnote{#1}%
  \addtocounter{footnote}{-1}%
  \endgroup
} 

\begin{document}

\title{\Large \bf
QANOVA: Quantile-based Permutation Methods For General Factorial Designs}
\author[1,$*$]{Marc Ditzhaus}
\author[1]{Roland Fried}
\author[1]{Markus Pauly}

\affil[1]{Department of Statistics, TU Dortmund University, Germany.}

\maketitle

\begin{abstract}
\blfootnote{${}^*$ e-mail: marc.ditzhaus@tu-dortmund.de}
 Population means and standard deviations are the most common estimands to quantify effects in factorial layouts. In fact, most statistical procedures in such designs are built towards inferring means or contrasts thereof. For more robust analyses, we consider the population median, the  interquartile range (IQR) and more general quantile combinations as estimands in which we formulate null hypotheses and calculate compatible confidence regions.  Based upon simultaneous multivariate central limit theorems and corresponding resampling results, we derive asymptotically correct procedures in general, potentially heteroscedastic, factorial designs with univariate endpoints. Special cases cover robust tests for the population median or the IQR in arbitrary crossed one-, two- and higher-way layouts with potentially heteroscedastic error distributions. In extensive simulations we analyze their small sample properties and also conduct an illustrating data analysis comparing children's height and weight from different countries.

\end{abstract}



\section{Introduction}

Factorial designs are popular in various fields such as ecology, biomedicine and psychology \citep{gissi:1990,baigent:etal:1998,cassidy:etal:2008,mehta:etal:2010,kurz:etal:2015} as they us allow to study interaction effects between different factors alongside their main effects. In fact, \citet{lubsen:pocock:1994} pointed out that  ``\textit{it is desirable for reports of factorial trials to include estimates of the interaction between the treatments}''. The ANOVA-$F$-test is the most common tool for this but suffers from restrictive assumptions such as homoscedasticity and normality.  Thus, several tests have been developed that allow for non-normal errors or are valid for heteroscedastic one- and two-way or even more general factorial designs \citep{johansen:1980,brunner:dette:munk:1997,bathke:schabenberger:tobias:madden:2009,zhang:2012, paulyETAL2015, friedrich2017permuting, friedrich2017gfd, harrar2019comparison}.

All these procedures describe effects by (contrasts of) means. This is in line with a phenomenon observed in various areas: comparisons are mainly based upon means or variances but not on their robust counterparts. This can be explained in part by the simplicity and elegance gained by using linear or, under independence, additive statistics. Nevertheless it contradicts the important role of statistics based on quantiles, like the median and the interquartile range (IQR), in data exploration and modeling, e.g. in boxplots or summary statistics. The interest in analyzing quantiles has lead to the development of quantile regression, which is commonly established nowadays \citep{koenker2001quantile, koenker2012quantile}. However, as, e.g., stressed by \citet{beyerlein2014quantile} ``{\it it appears to be quite underused in medical research}''. One reason may be that, although there exist several approaches for specific designs \citep{sen1962,potthoff1963,fung1980,hettmanspergerMcKean2010, fried2011robust, chungRomano2013}, there does not exist an equal abundance of methods based on quantiles for general factorial designs.
There are procedures, at least for the median, but they often require strong distributional assumptions (as symmetry) or, at least, an extension to factorial designs is missing.
Therefore the {\it main  aims} of the present paper are to develop inference procedures (tests and compatible confidence regions)
\begin{itemize}
	\item[(i)] for the median, the interquartile range (IQR) or any linear combination of quantiles.
	\item[(ii)] within the framework of factorial designs to study robust main and interaction effects.
	\item[(iii)] for general heterogeneous or heteroscedastic models beyond normality.
	\item[(iv)] that are theoretically valid and posses a satisfactory finite sample performance.
\end{itemize}
To achieve these goals, we combine and extend the ideas of \citet{chungRomano2013} (who derive tests for equality of medians in one-way ANOVA models) and \citet{paulyETAL2015} (who establish mean-based testing procedures in general factorial designs) to (simultaneously) infer arbitrary linear contrasts of general quantiles. In view of (ii) and (iv) we thereby follow the idea of {\it permuting studentized Wald-type statistics}
to obtain methods that are finitely correct in case of exchangeable data (e.g., under the null hypothesis of equal means / medians in the classic $F$-ANOVA normal model) {\it but also} asymptotically valid for general non-exchangeable settings. This alluring technique has originally been developed for special two-sample models \citep{neuhaus:1993,janssen:1997studentized,janssenPauls2003,pauly:2011:discussion} and has recently displayed its full strength to obtain accurate methods in one-way \citep{chungRomano2013} and more general factorial designs \citep{paulyETAL2015, friedrich2017permuting, smaga2017diagonal, umlauft2017rank, harrar2019comparison}.

However, to derive the fore-mentioned theoretical evidence in our general quantile-based approaches we could not employ the methods derived in the previously mentioned papers. In fact, to overcome some technical difficulties that occur when jointly permuting sample quantiles, we had to take a detour in which we extended some results for general permutation empirical processes and uniform Hadamard differentiability \citep{vaartWellner1996} that are of own mathematical interest. Anyhow, this finally results in (i)-(iv), i.e., a flexible toolbox for inferring contrasts of different quantiles in factorial designs. In the special case of the median and its bootstrap-based variance estimator we obtain the one-way permutation test derived in \citet{chungRomano2013}.

The paper is organized as follows: We first introduce the model, estimators for population quantiles and how to formulate null hypotheses in them to test for certain main or interaction effects. In Section \ref{sec:asy_results}, we state the theory to handle the joint asymptotics for sample quantiles and their covariance matrix estimators. As the latter are crucial to obtain the correct dependency structure necessary in the aforementioned studentization we study three different approaches based upon kernel density estimators, bootstrapping or certain interval estimates. As they are mostly only known for the sample median, these considerations require certain extension to our more general situations which are explained in Sections \ref{sec:kernel}--\ref{sec:boot}. From these findings we then deduce three different asymptotically valid testing procedures. To improve their small sample performance, we consider their respective permutation versions in Section \ref{sec:perm}, prove asymptotic exactness 
and also analyze their power under local and fixed alternatives.
To compare the small sample behavior of the resulting six tests, we conducted extensive simulations presented in Section \ref{sec:sim}. Finally, we illustrate the new methodology by analyzing a recent data set on the
height and weight of children in different countries in Section \ref{sec:real_data}. All proofs are deferred to the appendix, where also some additional simulation results are presented.

\section{The set-up}
We consider a general model given by mutually independent random variables
\begin{align}\label{eq:model}
X_{ij} \sim F_i \quad (i=1,\ldots,k;j=1,\ldots,n_i)
\end{align}
with absolutely continuous distribution functions $F_i$ and corresponding densities $f_i$. This set-up allows the incorporation of divers factorial structure by adequately splitting up indices. To accept this consider for example a two-way design with factors $A$ (possessing $a$ levels) and $B$ (having $b$ levels). Setting $k=a\cdot b$ we split up the group index $i$ into $i=(i_1,i_2)$ and model observations as $X_{i_1i_2j}\sim F_{i_1i_2}$ with $i_1=1,\ldots,a$ and $i_2=1,\ldots,b$. Factorial designs of more complexity can be incorporated similarly, see, e.g., \citet{paulyETAL2015}.

Having the model fixed we now turn to the parameters of interest: Choosing $m\in\N$ different probabilities $0<p_1<\ldots <p_m<1$ we want to study inference methods for the corresponding quantiles
\begin{align}\label{eqn:def_quantiles}
q_{ir}=F_i^{-1}(p_r) = \inf\{ t\in \R : F_i(t)\geq p_r\} \quad (i=1,\ldots,k;r=1,\ldots,m).
\end{align}
Pooling them in the vector $\mathbf{q}=(\mathbf{q}_1',\dots,\mathbf{q}_k')'=(q_{11},\ldots,q_{1m},q_{21},\ldots,q_{km})' $, we are particularly interested in testing the {\it QANOVA} null hypothesis $\mathcal H_0: \mathbf{H} \mathbf{q} = \mathbf{0}_r$ for a contrast matrix $\mathbf{H}\in \R^{r\times km}$ of interest. Here, $\mathbf{H}$ is called a contrast matrix if $\mathbf H \mathbf{1}_{km} = \mathbf{0}_r$ holds, where $\mathbf{1}_d$ and $\mathbf{0}_d$ are vectors of length $d$ consisting of $1$'s and $0$'s only, and  $\mathbf{A}'$ denotes the transpose of the matrix $\mathbf{A}$.
Choosing the contrast matrices in line with the design and the question of interest allows us to test various hypotheses about main and interaction effects, see Subsection~\ref{sec:con+matr} below.
Moreover, we want to point out that respective confidence regions for corresponding contrasts of quantiles can be obtained straightforwardly by inverting the test procedures. In what follows we will therefore focus on hypothesis testing but provide some exemplary confidence intervals in the context of the illustrative data analyses given in Section~\ref{sec:real_data}.
Turning back to the null hypothesis $\mathcal H_0:  \mathbf{H} \mathbf{q} = \mathbf{0}_r$ we recall from general ANOVA that it is convenient to re-formulate it as
$\mathcal H_0:  \mathbf{T} \mathbf{q} = \mathbf{0}_{km}$ for the unique projection matrix $\mathbf{T}=\mathbf{H}'(\mathbf{HH}')^+\mathbf H$, see, e.g.,  \citet{brunner:dette:munk:1997,paulyETAL2015,smaga2017diagonal}. Here, $\mathbf{A}^+$ denotes the Moore--Penrose inverse of the matrix $\mathbf{A}$. In fact, both matrices, $\mathbf{H}$ and $\mathbf{T}$, describe the same null hypothesis while $\mathbf{T}$ has preferable properties as being symmetric and idempotent.
To infer $\mathcal H_0$ we propose sensitive test statistics in the vector of corresponding sample quantiles. To introduce them, let
\begin{align*}
\widehat F_{i}(t)= n_i^{-1}\sum_{j=1}^{n_i}{1}\{X_{ij}\leq t\} \quad \text{and}\quad \widehat F(t)= n^{-1}\sum_{i=1}^k\sum_{j=1}^{n_i}{1}\{X_{ij}\leq t\},
\end{align*}
denote the group-specific and pooled empirical distribution function, respectively,
where $n=\sum_{i=1}^kn_i$ is the pooled sample size. 
Then the natural estimator of the quantile $q_{ir}$ is
\begin{align}\label{eqn:def_qhat}
\widehat q_{ir} = \widehat F_i^{-1}(p_r) = \inf\{ t\in \R : \widehat F_i(t)\geq p_r\} = X_{\lceil n_ip_r\rceil:n_i}^{(i)} 
\end{align}
for $i=1,\ldots,k;r=1,\ldots,m$, where $X_{1:n_i}^{(i)}\leq \ldots\leq X_{n_i:n_i}^{(i)}$ are the order statistics of group $i$.
\subsection{Examples of specific hypotheses} \label{sec:con+matr}
To give some examples of hypotheses covered within this framework we first consider a one-way design. For $m=1$, we obtain the $k$-sample null hypothesis of equal $p_1$-quantiles
\begin{itemize}
	\item \textit{No group effect:} $\mathcal H_0=\{\mathbf{P}_{k}\mathbf{q} = \mathbf{0}_{k} \} =\{ q_{1}=\ldots=q_{k}\}$ with $ \mathbf{P}_k =  \mathbf I_k - \mathbf{J}_k/k$.
\end{itemize}
Here, $\mathbf{I}_k\in\R^{k\times k}$ denotes the unit matrix, $\mathbf{J}_k=\mathbf{1}_k\mathbf{1}_k'$ and we suppressed the second index of the quantiles ($m=1$). Choosing $p_1=1/2$ gives the null hypothesis of equal medians which reduces to the null hypothesis of equal means in case of symmetric error distributions. Setting $k=ab$, we consider a two-way design with factors A (having levels $i_1=1,\dots,a$) and B (with levels $i_2=1,\dots,b$) and suppose that we like to formulate main and interaction effects in terms of quantiles, e.g. medians. The corresponding three null hypotheses are 
\begin{itemize}
	\item \textit{No main effect of A:} $\mathcal H_0 = \{\mathbf{H}_{A}\mathbf{q} = \mathbf{0}_{ab} \} = \{ \bar{q}_{1\cdot} = \ldots = \bar{q}_{a\cdot}\}$ with $\mathbf{H}_{A}= \mathbf P_a \otimes (\mathbf J_b/b)$,
	
	\item \textit{No main effect of B}: $\mathcal H_0 = \{\mathbf{H}_{B}\mathbf{q} = \mathbf{0}_{ab} \} = \{ \bar{q}_{\cdot1} = \ldots = \bar{q}_{\cdot b}\}$ with $\mathbf{H}_B = (\mathbf J_a/a) \otimes \mathbf P_b$,
	
	\item \textit{No interaction effect}: $\mathcal H_0 = \{\mathbf{H}_{AB}\mathbf{q} = \mathbf{0}_{ab} \} = \{
	\bar{q}_{\cdot \cdot}- \bar{q}_{\cdot i_2} - \bar{q}_{i_1\cdot} + {q}_{i_1i_2}   \equiv 0\}  $ with $\mathbf{H}_{AB} = \mathbf P_a \otimes \mathbf P_b$.
\end{itemize}
Here, $\otimes$ is the Kronecker product and $\bar{q}_{i_1 \cdot}$, $\bar{q}_{\cdot i_2}$ and $\bar{q}_{\cdot \cdot}$ are the means over the dotted indices. 
The latter hypotheses can also be described more lucid by utilizing an additive effects notation. To this end, we decompose the quantile $q_{i_1i_2} = q^\mu + q^\alpha_{i_1} + q^\beta_{i_2} + q^{\alpha\beta}_{i_1i_2}$ from group $(i_1,i_2)$
into a general effect $q^\mu$, main effects
$q^\alpha_{i_1}$ and $q^\beta_{i_2}$ 
as well as an interaction effect $q^{\alpha\beta}_{i_1i_2}$ assuming the usual side conditions $\sum_{i_1} q^\alpha_{i_1} = \sum_{i_2} q^\beta_{i_2} =  \sum_{i_1} q^{\alpha\beta}_{i_1i_2} =
\sum_{i_2} q^{\alpha\beta}_{ij} = 0$. Then the null hypotheses can be written as
$\{\mathbf{H}_{A}\mathbf{q} = \mathbf{0}_{ab} \} = \{q^\alpha_1=\dots q^\alpha_a=0\}$ or
$\{\mathbf{H}_{AB}\mathbf{q} = \mathbf{0}_{ab} \} = \{q^{\alpha\beta}_{i_1i_2} \equiv 0 \text{ for all } i_1,i_2 \}$.
Beyond working with specific quantiles it is also possible to infer hypotheses about
linear combinations $\mathbf{c}' \mathbf{q}_i = \sum_{r=1}^mc_rq_{ir}$ of quantiles. Here, $\mathbf{c}\in\mathbf{R}^k$ is an arbitrary vector, e.g.   choosing $c_1=-c_2=-1$ for $m=2$ and setting $p_1=0.25$ and $p_2=0.75$ leads to the group-specific interquartile ranges $\mathbf{c}' \mathbf{q}_i = IQR_i$. To obtain similar hypothesis in these parameters as above the
contrast matrix has to be specified to $\mathbf{\widetilde H}= \mathbf{ H} \otimes (c_1,\ldots,c_r)$, where $\mathbf{H}$ is one of the contrast matrices introduced above. For example, $\mathbf{H} = \mathbf{P}_{k}$ together with the previous choices for $\mathbf{c}$ and $p_1,p_2$ gives the null hypothesis $\{IQR_1=\dots=IQR_k\}$ of equal IQRs among all $k$ groups. However, the framework is much more flexible and even allows to infer hypotheses about IQRs and medians simultaneously by choosing $p_1=0.5$, $p_2=0.25$ and $p_3=0.75$  together with adequate contrast matrices.

\section{Asymptotic results}\label{sec:asy_results}
To establish the joined asymptotic theory for the sample quantiles and their covariance matrix estimators we assume non-vanishing groups throughout, i.e., as $\min(n_i:i=1,\ldots,k)\to\infty$
\begin{align}
&\frac{n_i}{n}\to\kappa_i >0\label{eqn:ni_n_kappai}.
\end{align}
Recall that the sample median will be asymptotically normal if the underlying density is positive and continuous in a neighbourhood of the true median. This statement can be extended to the multivariate case \citep{serfling2009}, e.g., under the following assumption, which we consider throughout.
\begin{assump}\label{ass:densities}
	Let $F_i$ be continuously differentiable at $q_{ir}$ with positive derivative $f_i(q_{ir})>0$ for every $r=1,\ldots,m$ and $i=1,\ldots,k$.
\end{assump}
\begin{prop}[Theorem B in Sec. 2.3.3 of \citet{serfling2009}]\label{prop:uncond}
	For $i\in\{1,\ldots,k\}$
	\begin{align}
	\sqrt{n}\Bigl( \widehat q_{ir} -  q_{ir}\Bigr)_{r=1,\ldots,m} \overset{\mathrm d}{\longrightarrow} \mathbf{Z}_i,
	\end{align}
	where $\mathbf{Z}_i$ is a zero-mean, multivariate normal distributed random variable with nonsingular covariance matrix $\mathbf{\Sigma}^{(i)}$ given by its entries
	\begin{align}\label{eqn:def_sigmai}
	\mathbf{\Sigma}^{(i)}_{ab} = \kappa_i^{-1}\frac{1}{f_i(q_{ia})f_i(q_{ib})}( p_a\wedge p_b - p_ap_b )\qquad (a,b\in\{1,\ldots,m\}).
	\end{align}
\end{prop}
In general, the covariance matrix is unknown and, thus, needs to be estimated.
However, let us suppose, for a moment, that we have chosen a consistent estimator $\mathbf{\widehat \Sigma}^{(i)}$ for $\mathbf{\Sigma}^{(i)}$. Then we could already define  a {\it Wald-type statistic} for testing $\mathcal{H}_0: \mathbf{T} \mathbf{q} = \mathbf{0}_r$
\begin{align}\label{eq:WTS}
S_n(\mathbf{T}) = n( \mathbf{T} \mathbf{ \widehat q})' ( \mathbf{T} \mathbf{ \widehat \Sigma} \mathbf{T}' )^{+} \mathbf{T} \mathbf{ \widehat q}\quad \text{with }\mathbf{ \widehat\Sigma} = \bigoplus_{i=1}^k\mathbf{ \widehat\Sigma }^{(i)},
\end{align}
where $\oplus$ denotes the direct sum. By Proposition \ref{prop:uncond} the limiting covariance matrix $\mathbf{\Sigma} = \oplus_{i=1}^k\mathbf{ \Sigma}^{(i)}$ is positive definite which implies that the Moore--Penrose inverse $( \mathbf{T} \mathbf{ \widehat \Sigma} \mathbf{T}' )^+$ converges in probability to $( \mathbf{T} \mathbf{\Sigma} \mathbf{T}' )^+$. Thus, $S_n(\mathbf{T})$ converges to $Z=\mathbf{Y}' ( \mathbf{T} \mathbf{\Sigma} \mathbf{T}' )^+ \mathbf{Y}$ in distribution under $\mathcal{H}_0$, where $\mathbf{Y}\sim N(\mathbf{0},  \mathbf{T} \mathbf{\Sigma} \mathbf{T}')$. Moreover, the limit $Z=\mathbf{Y}' ( \mathbf{T} \mathbf{ \Sigma} \mathbf{T}' )^+ \mathbf{Y}$ is chi-square distributed with $\text{rank}( \mathbf{T} \mathbf{ \Sigma} \mathbf{T}') = \text{rank}(\mathbf{T} \mathbf{ \Sigma}^{1/2}) = \text{rank}(\mathbf{T})$ degrees of freedom \citep[Theorem 9.2.2]{rao:mitra:1971}. We summarize this as
\begin{theorem}\label{theo:teststat_uncon}
	Under $\mathcal H_0: \mathbf{T}\mathbf{q} = \mathbf{0}_r$, $S_n(\mathbf{T})$ converges in distribution to $Z\sim\chi^2_{\text{rank}(\mathbf{T})}$.
\end{theorem}
Thus, comparing $S_n(\mathbf{T})$ with the $(1-\alpha)$-quantile  of the limiting null distribution defines an asymptotic exact level $\alpha$ test
$\varphi_n = 1\{S_n(\mathbf{T})> \chi^2_{\text{rank}(\mathbf{T}), 1-\alpha}\}$. 
As Proposition \ref{prop:uncond} is not restricted to the null hypothesis, we can even deduce that $n^{-1}S_n(\mathbf{T})$ always converges in probability to $( \mathbf{T} \mathbf{ q})' ( \mathbf{T} \mathbf{  \Sigma} \mathbf{T} )^+ \mathbf{T} \mathbf{  q}$. Since $\mathbf{T}\mathbf{q}\neq \mathbf{0}_{km}$ implies $( \mathbf{T} \mathbf{ q})' ( \mathbf{T} \mathbf{  \Sigma} \mathbf{T}' )^+ \mathbf{T} \mathbf{  q}>0$ (see the Supplement for a verification) consistency follows.
\begin{theorem}\label{theo:teststat_consist}
	Under $\mathcal H_1: \mathbf{T}\mathbf{q} \neq \mathbf{0}_r$, $S_n(\mathbf{T})$ converges in probability to $\infty$.
\end{theorem}

It remains to find appropriate estimators $\mathbf{\widehat \Sigma}^{(i)}$ for the unkown covariance matrices. For that purpose, we examine different strategies: 'Brude force' via plug-in of a kernel density estimator into \eqref{eqn:def_sigmai} or using a different approach that first estimates the diagonal elements $\Sigma_{aa}^{(i)}$ and then employs
their following relationship with the remaining matrix elements:
\begin{align}\label{eqn:def_sigmai_rewrite}
\mathbf{\Sigma}^{(i)}_{ab} = \sqrt{	\mathbf{\Sigma}^{(i)}_{aa} 	\mathbf{\Sigma}^{(i)}_{bb}} \frac{p_a\wedge p_b - p_ap_b }{\sqrt{(p_a-p_a^2)(p_b-p_b^2)}}\qquad (a,b\in\{1,\ldots,m\}).
\end{align}
In the latter case, we consider two ways for estimating the variances $\Sigma_{aa}^{(i)}$: Via bootstrapping \citep{efron1979} or with the interval estimator proposed in \citet{priceBonett2001}.
In the following subsections we explain all three possibilities in detail.

\subsection{Kernel estimator}\label{sec:kernel}
A popular way to estimate densities are so-called kernel density estimators, which are based on a Lebesgue density $K:\R \to [0,\infty)$ with $\int K(x) \,\mathrm{ d } x = 1$ and a bandwidth $h_{n}\to 0$. For more flexibility, we allow for different choices within the groups and thus add the corresponding group index, i.e., we work with $K_i$ and $h_{ni}$. Then, the kernel density estimator for $f_i$ is given by
\begin{align}\label{eqn:def_kernel_estimator}
\widehat f_{K,i}(x) = (n_ih_{ni})^{-1}\sum_{i=1}^{n_i} K_i\Bigl( \frac{x - X_{ij}}{h_{ni}} \Bigr)\quad (i=1,\dots,k).
\end{align}
\citet{nadaraya:1965} proved strong uniform consistency of \eqref{eqn:def_kernel_estimator}:
\begin{align}\label{eqn:kernel_uniform}
\sup_{x\in\R} \Bigl | \widehat f_{K,i}(x) - f_i(x) \Bigr| \to 0 \text{ with probability one}\quad (i=1,\dots,k)
\end{align}
\noindent under the following assumption:
\begin{assump}\label{ass:kernel}
	Let $K_i$ be of bounded variation and $f_i$ be uniformly continuous. Furthermore, suppose that $\sum_{n=1}^\infty \exp(-\gamma n h_{ni}^2)$ converges for any choice of $\gamma$.
\end{assump}
\noindent Here, the convergence of the series $\sum_{n=1}^\infty \exp(-\gamma n h_{ni})$ is, e.g., implied by choosing $h_{n,i} = n_i^{-\theta}$ for some $\theta\in(0,1/2)$. We further note that \citet{schuster:1969} discussed necessary and sufficient conditions for the stated uniform consistency. In particular, all $f_i$ need to be uniformly continuous. Moreover, the conditions on the bandwidths can be weakened when the kernel fulfils additional regularity conditions \citep{silverman:1978}. Anyhow, combining Proposition \ref{prop:uncond} and \eqref{eqn:kernel_uniform} yields consistency of the plug-in covariance matrix estimators.
\begin{lemma}\label{lem:kern_density_estimator}
	Under Assumption \ref{ass:kernel} we have for all $i=1,\dots,k$ and $a,b=1,\dots,m$:
	\begin{align}\label{eqn:def_hat_sigma_kernel}
	\mathbf{\widehat \Sigma}_{ab}^{(i),K}\equiv \frac{n}{n_i}\frac{p_a\wedge p_b - p_ap_b}{\widehat f_{K,i}(\widehat q_{ia})\widehat f_{K,i}(\widehat q_{ib})} \to \mathbf{\Sigma}_{ab}^{(i)}\quad \text{in probability}.
	\end{align}
\end{lemma}

\subsection{Bootstrap estimator}\label{sec:boot}
In their one-way tests for equality of medians, \citet{chungRomano2013} used the bootstrap approach of \citet{efron1979} to estimate the asymptotic variance of the sample median. We adopt this idea for general quantiles. Therefore, for every group $i$, let $X^*_{i1},\ldots,X^*_{in_i}$ denote a bootstrap sample (drawn with replacement) from the observations $\mathbf{X}_i = (X_{ij})_{j=1,\ldots,n_i}$. From this we can calculate bootstrap versions of all previous estimators which we indicate by a superscript $^*$, e.g., $\widehat q_{ir}^{\:*}$ and $\widehat F_{i}^*$. Then, the mean squared error of the bootstrapped sample quantile given the data can be explicitly calculated using a simple reordering trick and \eqref{eqn:def_qhat}
\begin{align*}
\Bigl( \widehat \sigma_i^*(p_r) \Bigr)^2 &\equiv \E( n_i (\widehat q_{ir}^{\:*} - \widehat q_{ir})^2 \mid \mathbf{X}_i) = n_i \sum_{j=1}^{n_i} \Bigl( X_{ij} - \widehat q_{ir} \Bigr)^2 \P \Bigl(  \widehat q_{ir}^{\:*} = X_{ij}\mid \mathbf{X}_i\Bigr)\\
& = n_i \sum_{j=1}^{n_i} \Bigl( X_{j:n_i}^{(i)} - \widehat q_{ir} \Bigr)^2 P_{ij};\quad P_{ij}=\P \Bigl(  X_{\lceil mp_r \rceil :n_i}^{(i),*} = X_{j:n_i}^{(i)} \mid \mathbf{X}_i \Bigr).
\end{align*}
Following \citet{efron1979}, the probabilities $P_{ij}$ can be rewritten to
\begin{align*}
P_{ij} = \P( B_{n_i,(j-1)/n_i} \leq \lceil n_ip_r \rceil-1) - \P( B_{n_i,j/n_i} \leq \lceil n_ip_r \rceil-1) 
\end{align*}
where $B_{n,p}$ denotes a Binomial distributed random variable with size parameter $n$ and success probability $p$. In contrast to the standard jackknife method, the bootstrap  median variance estimator $( \widehat \sigma_i^*(1/2) )^2$ converges  to $1/(4f_i^2(F^{-1}(1/2)))$ as desired \citep{efron1979}. Moreover, a detailed proof for strong consistency of this estimator was given by \citet{ghoshETAL:1984} under 
\begin{assump}\label{ass:Boot_estim}
	Let  $\max_{i=1,\ldots,k}\E(|X_{i1}|^\delta)<\infty$ for some $\delta>0$.
\end{assump}
Later, \citet{babu:1986} weakened their assumptions and \citet{hall:martin:1988} studied the exact convergence rate of the estimator. Nevertheless, Assumption~\ref{ass:Boot_estim} is no big restriction for practical purposes and we therefore  prove consistency under this presumption.
\begin{lemma}\label{lem:Boot_estimator}
	Under Assumption \ref{ass:Boot_estim} we have for all $i=1,\dots,k$ and $a,b=1,\dots,m$:
	\begin{align*}
	\mathbf{\widehat \Sigma}_{ab}^{(i),B} \equiv \frac{n}{n_i}\widehat \sigma_i^*(p_a) \widehat \sigma_i^*(p_b)  \frac{p_a\wedge p_b - p_ap_b }{\sqrt{(p_a-p_a^2)(p_b-p_b^2)}} \overset{p}{\rightarrow} \mathbf{\Sigma}_{ab}^{(i)}.
	\end{align*}
\end{lemma}

\subsection{Interval-based estimator}\label{sec:PB}
\citet{mckeanSchrader1984} introduced an estimator for the sample median standard deviation based on a standardized confidence interval. Later, \citet{priceBonett2001} suggested to modify this estimator to improve its performance in small sample size settings. Both estimators are consistent \citep{priceBonett2001} and can compete with the aforementioned bootstrap approach in simulations \citep{mckeanSchrader1984,priceBonett2001} with a slightly better performance of the Price--Bonnet modification. While both papers only treat the median, extensions to general quantiles follow intuitively and have already been used, e.g., for the $25\%$- and $75\%$-quantile in \citet{bonett2006}. For a thorough definition of this extension let $p\in(0,1)$ be a given level. Then we define the (extended) McKean--Schrader estimator for the standard deviation of the $p$-th sample quantile as
\begin{align*}
\widehat \sigma_i^{MS}(p) = n_i^{1/2}\frac{(X^{(i)}_{u_i(p):n_i} - X^{(i)}_{l_i(p):n_i})}{ 2z_{\alpha/2}},
\end{align*}
where $\alpha\in(0,1)$ and $l_i (p) = 1 \vee \lfloor n_ip - z_{\alpha/2} \sqrt{n_i}\sqrt{p(1-p)}\rfloor$ as well as $u_i (p) = n_i \wedge \lfloor n_ip + z_{\alpha/2} \sqrt{n_i}\sqrt{p(1-p)}\rfloor$ are the lower and upper limits of binomial intervals. Here, $z_{\alpha/2}$ denotes the $(1-\alpha/2)$-quantile of the standard normal distribution. 
Typically, $\alpha=0.05$ is chosen leading to $z_{\alpha/2}\approx 1.96$. A brief discussion on the effect of the choice $\alpha$ on the estimator can be found in \citet{priceBonett2001}. In fact, the Price--Bonnet modification concerns the choice of $\alpha$: They propose to replace it in the denominator by the following finite sample correction (where we suppressed the dependency on $i$ for ease of notation)
\begin{align*}
\alpha_n^*(p) &= \P\Bigl( F_i^{-1}(p) \notin (X^{(i)}_{l_i(p):n_i}, X^{(i)}_{u_i(p):n_i}) \Bigr) = 1- \P \Bigl( \sum_{j=1}^{n_i} \mathbf{1}\{ X_{i,j} \leq  F_i^{-1}(p)\} \in (l_i(p), u_i(p) ) \Bigr)\\
& = 1 - \sum_{i= l_i(p)+1}^{ u_i(p) - 1} \binom{n_i}{j} p^j(1-p)^{n_i - j}.
\end{align*}
Clearly, $\alpha_n^*(p) \to \alpha$ by the central limit theorem. For large sample sizes the benefit of the correction is negligible and may even lead to computational problems due to $\binom{n_i}{j} \gg 1$, especially for $j\approx n_i/2$. Thus, we only use the modifications for sample sizes smaller than $100$ and recommend to set $\alpha_n^*(p)=\alpha$ for larger values ($n_i>100$). 
Moreover, the simulations of \citet{priceBonett2001} reveal that additionally adding $2n_i^{-1/2}$ to the denominator results in a slight reduction of bias and mean squared error. Altogether, we thus define their extended estimator for the respective standard deviation as
\begin{align*}
\widehat \sigma_i^{\text{PB}}(p) = n_i^{1/2}\frac{(X^{(i)}_{u_i(p):n_i} - X^{(i)}_{l_i(p):n_i})}{ 2z_{\alpha_n^*(p)/2}+ 2n_i^{-1/2}}.
\end{align*}
As explained above, this estimator is consistent for the variance and we thus obtain:
\begin{lemma}\label{lem:PB_estimator}
	We have for all $i=1,\dots,k$ and $a,b=1,\dots,m$:
	\begin{align}\label{eqn:def_hat_sigma_PB}
	\mathbf{\widehat \Sigma}_{ab}^{(i),\text{PB}} = \frac{n}{n_i}\widehat \sigma_i^{\text{PB}}(p_a) \widehat \sigma_i^{\text{PB}}(p_b)  \frac{p_a\wedge p_b - p_ap_b }{\sqrt{(p_a-p_a^2)(p_b-p_b^2)}} \overset{p}{\rightarrow} \mathbf{\Sigma}_{ab}^{(i)}.
	\end{align}
\end{lemma}
Utilizing the three different choices of covariance estimators results in three different versions of the asymptotic test $\varphi_n$. However, simulation results (Section~\ref{sec:sim}) exhibit serious issues for small to moderate sample sizes which may be due
to a rather poor $\chi^2$-approximation to the test statistic. To tackle this problem, we propose the initially mentioned technique of permuting studentized statistics.

\section{Permutation test}\label{sec:perm}

For a better finite sample performance, it is often advisable to replace the asymptotic critical value of the test, here the $(1-\alpha)$-quantile of the $\chi^2_{\text{rank}(\mathbf{T})}$-distribution, by a resampling-based critical value. For the current problem, we promote the permutation approach, which leads to a finitely exact test under exchangeability, i.e., under $\widetilde{ \mathcal H}_0: F_1=\ldots=F_k$. Moreover, the proper studentization within the Wald-type statistic makes it possible to transfer the consistency and asymptotic exactness (under $\mathcal{H}_0:\mathbf{T}\mathbf{q}=0$) of the tests $\varphi_n$ to their permutation versions. To explain this, let $\mathbf{X}^\pi=(X_{ij}^\pi)_{i=1,\ldots,k;j=1,\ldots,n_i}$ be a random permutation of the pooled data $\mathbf{X}=(X_{ij})_{i=1,\ldots,k;j=1,\ldots,n_i}$. As for Efron's bootstrap, we draw new samples from the pooled data, but now without replacement. In other words, we randomly permute the group memberships of the observations $X_{ij}$. Pooling the data affects our Assumptions \ref{ass:densities} and \ref{ass:kernel} such that we need to replace the original distribution functions $F_i$ and their densities $f_i$ by their pooled versions $F=\sum_{i=1}^k\kappa_i F_i$ and $f=\sum_{i=1}^k\kappa_if_i$ describing the (unconditional) distribution of $X_{ij}^\pi$. To be concrete, we postulate

\begin{assump}\label{ass:densities_perm}
	Let $F$ be differentiable with uniformly continuous derivative $f$ such that $f(q_{r})>0$ for all $r$, where $q_r=F^{-1}(p_r)$, and $K_i$ be a kernel fulling Assumption \ref{ass:kernel}.
\end{assump}
As in \citet{chungRomano2013} it turned out that the asymptotic correctness of the permutation approach needs a certain convergence rate in the sample size condition \eqref{eqn:ni_n_kappai}:
\begin{align}\label{eqn:ni_n_perm}
\frac{n_i}{n} - \kappa_i = O(n^{-1/2}).
\end{align}

\begin{theorem}\label{theo:stat_perm}
	Under $\mathcal H_0:\mathbf{T}\mathbf{q}=\mathbf{0}_{km}$ as well as under $\mathcal H_1:\mathbf{T}\mathbf{q}\neq\mathbf{0}_{km}$, the permutation version $S_n^\pi(\mathbf{T})$ of $S_n(\mathbf{T})$ with any of the covariance estimators \eqref{eqn:def_hat_sigma_kernel} -- \eqref{eqn:def_hat_sigma_PB} always mimics its null distribution asymptotically, i.e.,
	\begin{align}\label{eqn:theo:stat:perm}
	\sup_{x\in\R} \Bigl |  \P \Bigl( S_n^\pi(\mathbf{T}) \leq x \vert \mathbf{X} \Bigr) - \chi^2_{\text{rank}(\mathbf{T})}((-\infty,x])  \Bigr| \overset{p}{\rightarrow} 0.
	\end{align}
\end{theorem}
Replacing the critical value $\chi^2_{\text{rank}(\mathbf{T}),1-\alpha}$ of the asymptotical tests with $c_n^\pi(\alpha)$, the $(1-\alpha)$-quantile of the conditional distribution function $ x\mapsto \P (S_n^\pi(\mathbf{T}) \leq x \vert \mathbf{X})$, leads to three different permutation tests $\varphi_n^\pi = 1\{S_n(\mathbf{T})>c_n^\pi(\alpha)\}$. Under the assumptions given in Theorem~\ref{theo:stat_perm} it follows that $c_n^\pi(\alpha)$ converges in probability to $\chi^2_{\text{rank}(\mathbf{T}),1-\alpha}$ irrespective whether the null hypothesis is true or not. Thus, we can deduce the asymptotic exactness of the permutation test  and its consistency for general fixed alternatives \citep[Lemma 1 and Theorem 7]{janssenPauls2003}.  In addition, we prove in the next section that the permutation test has an asymptotic relative efficiency of $1$ compared to the asymptotic test $\varphi_n$, i.e., the tests' asymptotic power values coincide for local alternatives.

\subsection{Local alternatives}\label{sec:local_alt}
To study local alternatives we need to replace Model \eqref{eq:model} with its local counterpart given by a triangular array of row-wise independent random variables
\begin{align*}
X_{nij} \sim F_{ni} \quad (i=1,\ldots,k;j=1,\ldots,n_i)
\end{align*}
with absolutely continuous distribution functions $F_{ni}$, corresponding densities $f_{ni}$, quantiles $q_{nir}$ and quantile vector $\mathbf{q}_n= (q_{n11},\ldots,q_{n1m},q_{n21},\ldots,q_{nkm})'$. Within this framework we discuss local alternatives of the form $\mathbf{T}\mathbf{q}_n = O(n^{-1/2})$, i.e., small perturbations of the null hypotheses, under the following additional regularity conditions:
\begin{assump}\label{ass:local_alt}
	For every $i=1,\ldots,k$ let $F_i$ be an absolutely continuous distribution function with corresponding density $f_i$. Moreover, set $F=\sum_{i=1}^k\kappa_iF_i$. 
	\begin{enumerate}[(i)]
		\item \label{enu:ass:al_alt_F} For some $M>0$ let $\sqrt{n} |F_{ni}(x)-F_i(x)|\leq M$ for all $n\in\N$ and all $x\in\R$.
		
		\item \label{enu:ass:local_alt_fi} Suppose that $f_i$ is continuous and positive at $q_{ir}$ and that $f_{ni}$ converges uniformly to $f_i$ in a compact neighborhood around $q_{ir}=F_i^{-1}(p_r)$ for all $r$.
		
		\item \label{enu:ass:local_alt_fi_perm} For the permutation approach, suppose additionally \eqref{eqn:ni_n_perm}, Assumption \ref{ass:densities_perm} and  uniform convergence of $f_{ni}$ to $f_i$ in a compact neighborhood around $q_{r}=F^{-1}(p_r)$ for every $r$.
	\end{enumerate}
\end{assump}
While (ii) and (iii) are local versions of the regularity conditions assumed for Model \eqref{eq:model}, condition (i) ensures the usual $\sqrt{n}$-convergence of $F_{ni}$ to $F_i$. Anyhow, the asymptotic power function of both tests can be described by means of a non-central  $\chi^2$ distribution:
\begin{theorem}\label{theo:local_alt}
	Under $\sqrt{n}\mathbf{T}\mathbf{q}_n \to \boldsymbol{\theta}\neq \mathbf{0}_{km}$ the asymptotic test $\varphi_n$ and its permutation variant $\varphi_n^\pi$ with any of the covariance estimators \eqref{eqn:def_hat_sigma_kernel} -- \eqref{eqn:def_hat_sigma_PB} have the same asymptotic power $P(Z > \chi^2_{\text{rank}(\mathbf{T}),1-\alpha})>\alpha$, where $Z$ is $\chi^2_{\text{rank}(\mathbf{T})}(\delta)$-distributed with non-centrality parameter $\delta= \boldsymbol{\theta}'( \mathbf{T} \mathbf{  \Sigma} \mathbf{T}' )^+ \boldsymbol{\theta} > 0$.
\end{theorem}

\section{Simulations}\label{sec:sim}

To asses the tests' small sample performance we complement our theoretical findings with numerical comparisons. For ease of presentation, we restrict to
\begin{enumerate}
	\item {\bf A one-way layout} in which we like to infer the null hypothesis $\mathcal{H}_0: \{IQR_1=\dots=IQR_4\}$ of equal IQRs, i.e., as described in Section \ref{sec:con+matr} we choose probabilities $p_1=0.25$ and $p_2=0.75$  and specify the contrast matrix as $H= P_4 \otimes (-1,1)$.
	\item {\bf A $2\times 2$ layout} in which we test for the presence of main or interaction effects measured in terms of medians, i.e.
	setting $k=a\cdot b=2\cdot 2$ we infer the hypotheses
	$\mathcal{H}_0:\{ {\bf H}_{A} {\bf q} = {\bf 0}_{ab}\}$ (no main median effect of factor $A$) and $\mathcal{H}_0:\{ {\bf H}_{AB} {\bf q} = {\bf 0}_{ab}\}$ (no median $A\times B$ interaction effect), see Section \ref{sec:con+matr}.

\end{enumerate}
Data was simulated within Model \eqref{eq:model} via $X_{ij}= \mu_i + \sigma_i(\epsilon_{ij}-m_i)\sim F_i, i=1\dots,4, j=1,\dots,n_i$, where we consider (a) balanced as well as unbalanced settings given by sample size vectors $\mathbf{n_1}=(15,15,15,15)$ and $\mathbf{n_2}=(10,10,20,20)$, respectively. (b) five different distributions for $\epsilon_{ij}$ including the standard normal distribution ($N_{0,1}$), Student's t-distribution with $df=2$ and $df=3$ degrees of freedom ($t_2$ and $t_3$), the chi-square distribution with $df=3$ degrees of freedom ($\chi^2_3$) and the standard log-normal distribution ($LN_{0,1}$). All distributions were centered by substracting the respective median $m_i$ from $\epsilon_{ij}$. (c) a homoscedastic setting $\boldsymbol{\sigma_1}=(\sigma_1,\dots,\sigma_4) =(1,1,1,1)$ and heteroscedastic designs given by standard deviation  vectors $\boldsymbol{\sigma_2}=(1,1.25,1.5,1.75)$ and  $\boldsymbol{\sigma_3}=(1.75,1.5,1.25,1)$. In combination with ${\bf n}_2$ the latter represent a positive, respectively, negative pairing.

The simulations were conducted by means of the computing environment R \citep{R}, version 3.5.0, generating $N_{\text{sim}}=5000$ simulation runs and $N_{\text{perm}}=1999$ permutation iterations for each setting. The nominal level was set to $\alpha = 5\%$. We compare the type-1 error rate as well as the power values of our tests in Sections \ref{sec:sim_typ1} and \ref{sec:sim_power}, respectively. In both cases, we include all three variance estimation strategies introduced in Sections \ref{sec:kernel}--\ref{sec:boot}. For the kernel density estimation, we choose the classical Gaussian kernel with a bandwidth according to Silverman's rule-of-thumb \citep[Eq. (3.31)]{silverman:1986},  where we applied the function \textit{bw.nrd0} from the R package \textit{stats} to determine the latter.

In case of the $2\times 2$-median design these methods are additionally compared with the current state-of-the art tests for regression parameters in quantile regression: from the R package \textit{quantreg} \citep{koenker2012quantile} we 
choose the rank inversion method by \citet{koenkerMACHADO1999} for non-iid errors, the default choice in \textit{quantreg}, and the wild bootstrap approach of \citep{fengETAL2011wild}. For a fair comparison, we include the main factors $A$ and $B$ and their interaction in the respective regression model. Hence, regression parameters $\beta_A$, $\beta_B$ and $\beta_{AB}$ are estimated, and corresponding $p$-values for testing $\mathcal H_0:\beta_A = 0$ (no main effect A) and $\mathcal H_0:\beta_{AB}=0$ (no interaction effect), are derived by both \textit{quantreg} approaches.

\subsection{Type-1 error}\label{sec:sim_typ1}

In this subsection, we discuss the type-I-error control of all procedures.
To simulate under the corresponding null hypotheses, we set $\mu_i=\mu_{i_1i_2}=0$ in the $2\times2$-median-based cases and restrict to the homoscedastic setting $\boldsymbol{\sigma} = \boldsymbol{\sigma}_1$ for the $4$-sample IQR testing question. The standard error of the estimated sizes in case of $N=5000$ simulation runs is $0.3\%$ if the true type-I-error probability is $5\%$, i.e., estimated sizes outside the interval [$4.4\%$,$5.6\%$] deviate significantly from the nominal $5\%$ significance level.

The observed type-1 error rates for the $2\times 2$-median design are displayed in Table~\ref{tab:median_null} for testing the hypothesis of no main effect. It is readily seen that all asymptotic tests are rather conservative with type-$I$-errors reaching down to
$1.7\%$ for the bootstrap-based and $0.7\%$ for the interval-based approaches, respectively.
This conservativeness is less pronounced for the test based upon the kernel density variances estimator that exhibits values between $2.7\%$ and $5.7\%$ and a reasonable good error control in case of the standard normal and $\chi_3^2$ distribution except for the settings with positive variance pairing. In contrast, all permutation methods control the type-$I$-error level reasonably well except for the situations with a skewed distribution and negative pairing. Here, we find error rates up to $7.2\%$ for the tests based upon the interval- and kernel-based variance estimators. For the two  quantile regression methods from the R package \textit{quantreg} \citep{koenker2012quantile} the observations are diverse: The rank-based approach tends to conservative test decisions in case of unbalanced sample sizes with observed error rates in the range $2.5\% - 3.5\%$. However, in case a balanced homscedastic design with symmetric errors, a slight liberality ($6.4\%-6.9\%$) is detected.  For all other settings the decisions are accurate. In comparison, the wild bootstrap strategy is liberal for almost all balanced settings (with observed error rates up to $7.7\%$) and conservative for all positive pairings ($2.8\% - 3.7\%$). Overall, the permutation procedure that uses a bootstrap variance estimator exhibits the most robust type-$I$-error control with values ranging from $4.7\%-6.4\%$. 

Summarizing the results for the interaction tests presented in the appendix, we get a  similar impression for the wild bootstrap quantile regression strategy and the
six Wald-type procedures. For them, the only major difference is that the permutation methods also exhibit a fairly well error control for the settings with skewed distributions and negative pairing. However, the results for the rank-based quantile regression method are partially different: While the type-1 error rate is still accurate for balanced sample sizes, the decisions become very liberal  in the unbalanced scenarios with 
estimated type-I-error rates between $6.1\%$ and $10.1\%$.


The type-1 error rates in the situation of the 4-sample testing problem of equal IQRs are presented in Table~\ref{tab:IQR_null}. Here, the finite sample behavior of the asymptotic tests becomes even more extreme: For the symmetric distributions, the type-1 error rates are between $0.4\%$ and $1.3\%$ for the interval-based estimator and between $0.3\%$ and $1.2\%$ for the bootstrap approach, i.e., very conservative. In contrast, the decisions for the kernel-based method are quite accurate with values between $3.7\%$ and $5.0\%$. Switching to skewed distribution, however, the type-$I$ error rates increase, leading to very liberal decisions  in the log-normal case with values up to $10.2\%$ for the kernel-based and $7.5\%$ for the interval-based tests. Here, only the bootstrap-based method remained very conservative.
In comparison, all permutation counterparts lead to satisfactory type-1 error control close to the $5\%$-level. 

Due to the extreme behavior of the asymptotic tests in this setting, we conducted additional simulation results in the appendix. Therein, all asymptotic tests for equality of IQRs more or less approach the $5\%$ level for larger group-specific sample sizes $n_i\geq 150$.

\begin{table}
	\centering
	\caption{Type-1 error rate in $\%$ (nominal level $\alpha = 5\%$) for
		testing the median null hypothesis of no main effect in the $2\times 2$ design for the rank-based (Rank) and wild bootstrap (Wild) quantile regression approach as well as all asymptotic and permutation tests using the interval-based (Int), kernel density (Kern) and bootstrap (Boot) approach for estimating the covariance matrix. Values inside the $95\%$ binomial interval $[4.4,5.6]$ are printed bold.}\label{tab:median_null}
	\begin{tabular}{lll|rrr|rrr|rr|l}
		\multicolumn{3}{c}{}&\multicolumn{3}{c}{Asymptotic} & \multicolumn{3}{c}{Permutation} & \multicolumn{2}{c}{Quantile reg.} & \\
		Distr & $\mathbf{n}$ & $\boldsymbol{\sigma}$   & Int & Kern & Boot & Int & Kern & Boot & Rank & Wild & Setting\\
		\hline		
		$N_{0,1}$ & $\mathbf{n_1}$ & $\boldsymbol{\sigma_1}$ & 2.6 & 4.2 & 3.3 & \textbf{4.9} & \textbf{5.1} & \textbf{5.2} & 6.9 & 6.6 & balanced homosc.  \\
		&  & $\boldsymbol{\sigma_2}$ & 3.0 & \textbf{4.4} & 3.3 & \textbf{5.5} & {5.8} & \textbf{5.5} & 5.7 & 7.3
		& balanced heterosc. \\
		& $\mathbf{n_2}$ & $\boldsymbol{\sigma_1}$ &  2.2 & \textbf{4.8} & 3.6 & \textbf{5.0} & \textbf{5.6} & \textbf{5.6} & 2.5 & 4.1 & unbalanced homosc. \\
		
		&  & $\boldsymbol{\sigma_2}$ &   2.0 & 4.0 & 3.0 & {5.7} & \textbf{4.7} & \textbf{4.7}& 3.5 & 3.6 & positive pairing \\
		&  & $\boldsymbol{\sigma_3}$ &   2.5 & \textbf{5.4} & 4.0 & 6.2 & 6.4 & 6.3 & 3.0 & \textbf{5.2} & negative pairing\\
		\hline
		$t_3$ & $\mathbf{n_1}$ & $\boldsymbol{\sigma_1}$ &  1.7 & 2.7 & 2.3 & \textbf{5.1} & \textbf{5.2} & \textbf{5.2} & 6.4 & \textbf{5.5} & balanced homosc.  \\
		&  & $\boldsymbol{\sigma_2}$ & 2.0 & 2.9 & 2.6 & \textbf{5.5} & \textbf{5.2} &\textbf{4.9} & \textbf{5.6} & 6.0 & balanced heterosc. \\
		& $\mathbf{n_2}$ & $\boldsymbol{\sigma_1}$ & 0.8 & 3.0 & 2.1 & \textbf{4.5} & \textbf{4.5} & 3.5 & 3.2 & \textbf{4.4}  & unbalanced homosc. \\
		&  & $\boldsymbol{\sigma_2}$ & 1.1 & 3.1 & 2.4 & 6.4 & \textbf{4.7} & \textbf{5.2} & 3.1 & 2.8  & positive pairing \\
		&  & $\boldsymbol{\sigma_3}$ & 0.7 & 3.7 & 2.7 & \textbf{5.8} & 6.5 & 6.4 & 3.1 & 3.9 & negative pairing \\	  		
		\hline 		  	
		$LN_{0,1}$ & $\mathbf{n_1}$ & $\boldsymbol{\sigma_1}$ & \textbf{4.9} & 4.0 & 2.0 & \textbf{5.4} & 5.8 & 5.7 & \textbf{4.6} & 6.5 & balanced homosc.  \\
		&  & $\boldsymbol{\sigma_2}$ &   \textbf{5.2} & 3.8 & 1.8 & 5.8 & 5.8 & 6.0 & \textbf{4.6} & 7.6 & balanced heterosc. \\
		
		& $\mathbf{n_2}$  & $\boldsymbol{\sigma_1}$ & 3.1 & 3.0 & 1.7 & \textbf{4.8} & \textbf{4.7} & \textbf{4.8}  & 2.8 & 3.6 & unbalanced homosc. \\
		&  & $\boldsymbol{\sigma_2}$ &  3.4 & 3.4 & 2.1 & {5.9} & \textbf{5.3} & \textbf{5.4} & 3.0 & 3.7 & positive pairing \\
		&  & $\boldsymbol{\sigma_3}$ &  3.8 & \textbf{4.2} & 2.5 & 6.6 & 6.8 & 6.3 & 3.0 & \textbf{5.4} & negative pairing \\  		
		\hline
		$\chi^2_3$ & $\mathbf{n_1}$ & $\boldsymbol{\sigma_1}$ & \textbf{5.1} & \textbf{5.0} & 3.2 & \textbf{5.5} & \textbf{5.6} & 5.8 & 5.7 & 7.7 & balanced homosc.  \\
		&  & $\boldsymbol{\sigma_2}$ &  \textbf{4.4} & \textbf{4.7} & 2.7 & \textbf{5.1} & \textbf{5.5} & \textbf{5.2} & \textbf{5.4} & 7.7 & balanced heterosc. \\
		& $\mathbf{n_2}$ & $\boldsymbol{\sigma_1}$ &  3.6 & \textbf{4.5} & 2.8 & \textbf{5.0} & \textbf{5.1} & \textbf{5.2} & \textbf{3.0} & \textbf{4.8} & unbalanced homosc. \\
		&  & $\boldsymbol{\sigma_2}$ &  3.3 & 3.7 & 2.6 & \textbf{5.1} & \textbf{4.6} & \textbf{4.8} & 3.3 & 3.2  & positive pairing \\
		&  & $\boldsymbol{\sigma_3}$ &  \textbf{4.6} & 5.7 & 3.5 & 7.2 & 7.2 & 6.4 & 3.2 & \textbf{5.6} & negative pairing \\
		\hline
	\end{tabular}
\end{table}


%

\begin{table}
	\centering
	\caption{Type-1 error rate in $\%$ (nominal level $\alpha = 5\%$) for the 4-sample IQR testing problem of our asymptotic and permutation tests using the interval-based (Int), kernel density (Ker) and bootstrap (Boo) approach for estimating the covariance matrix. Values inside the $95\%$ binomial interval $[4.4,5.6]$ are printed bold.}\label{tab:IQR_null}
	
	\begin{tabular}{l|rrr|rrr||rrr|rrr}
		\multicolumn{1}{c}{} & \multicolumn{6}{c}{$\mathbf{n}_1$ (balanced)} & \multicolumn{6}{c}{$\mathbf{n}_2$ (unbalanced)} \\ 
		\multicolumn{1}{c}{}&\multicolumn{3}{c}{Asymptotic} & \multicolumn{3}{c}{Permutation} & \multicolumn{3}{c}{Asymptotic} & \multicolumn{3}{c}{Permutation} \\
		Distr  & Int & Ker & Boo & Int & Ker & Boo & Int & Ker & Boo & Int & Ker & Boo \\
		\hline	
		$N_{0,1}$ & 1.3 & \textbf{4.6} & 1.2 & {\bf 5.2} & \textbf{5.1} & {\bf 5.2} &  1.0 & \textbf{5.0} & 1.1 &{\bf 4.7} &{\bf 4.8} & {\bf 4.9} \\
		$t_2$ &  0.6 & {3.7} & 0.3 & \textbf{5.1} & \textbf{5.1} & \textbf{5.1} & 0.4 & {\bf 4.8} & 0.7 & \textbf{5.0} & \textbf{5.0} & {\bf 5.2}\\
		$t_3$ & 0.9 & {3.7} & 0.6 & {\bf 4.9} & \textbf{5.0} & {\bf 5.3} & 0.5 & \textbf{4.4} & 0.9 & 4.5 &{ \bf 4.6} & \textbf{5.1} \\
		$LN_{0,1}$ & {7.5} & 8.5 & 1.6 & {\bf 5.2} & \textbf{4.9} & {\bf 4.6} & {4.3} & 10.2 & 1.6 &{\bf 4.8} & \textbf{5.2} & {\bf 4.7}\\
		$\chi_3^2$ & \textbf{5.1} & 6.2 & 1.8 & {\bf 4.5} & {\bf 4.8} & {\bf 4.7} & {3.8} & 8.1 & 1.7 &{\bf 5.1} & \textbf{5.0} & {\bf 4.8}\\
		\hline
	\end{tabular}
\end{table}

\subsection{Power behavior under shift and scale alternatives}\label{sec:sim_power}
Due to the diverse behavior of the asymptotic tests and the rank-based quantile regression method under the null hypotheses and for ease of presentation, we solely focus on permutation tests and the wild bootstrap quantile regression strategy here. The results for the asymptotic tests are presented in the appendix and apart from their different level under $ \mathcal{H}_0$, their power curves run almost parallel to the respective curve of the permutation version.

\begin{figure}
	\centering
	\underline{\text{Median}}\\
	\includegraphics[ width =0.43\textwidth]{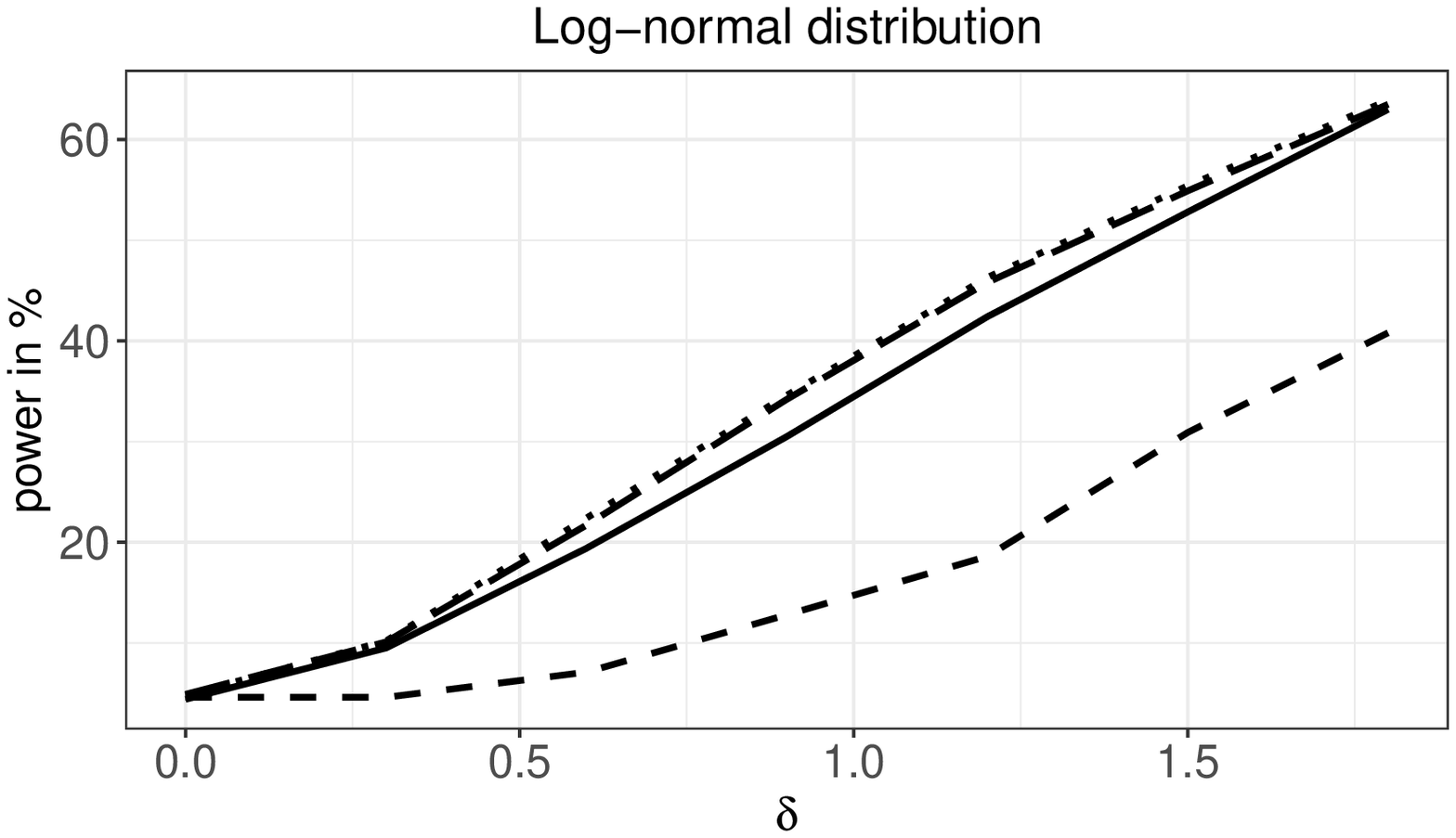}\qquad
	\includegraphics[ width = 0.43\textwidth]{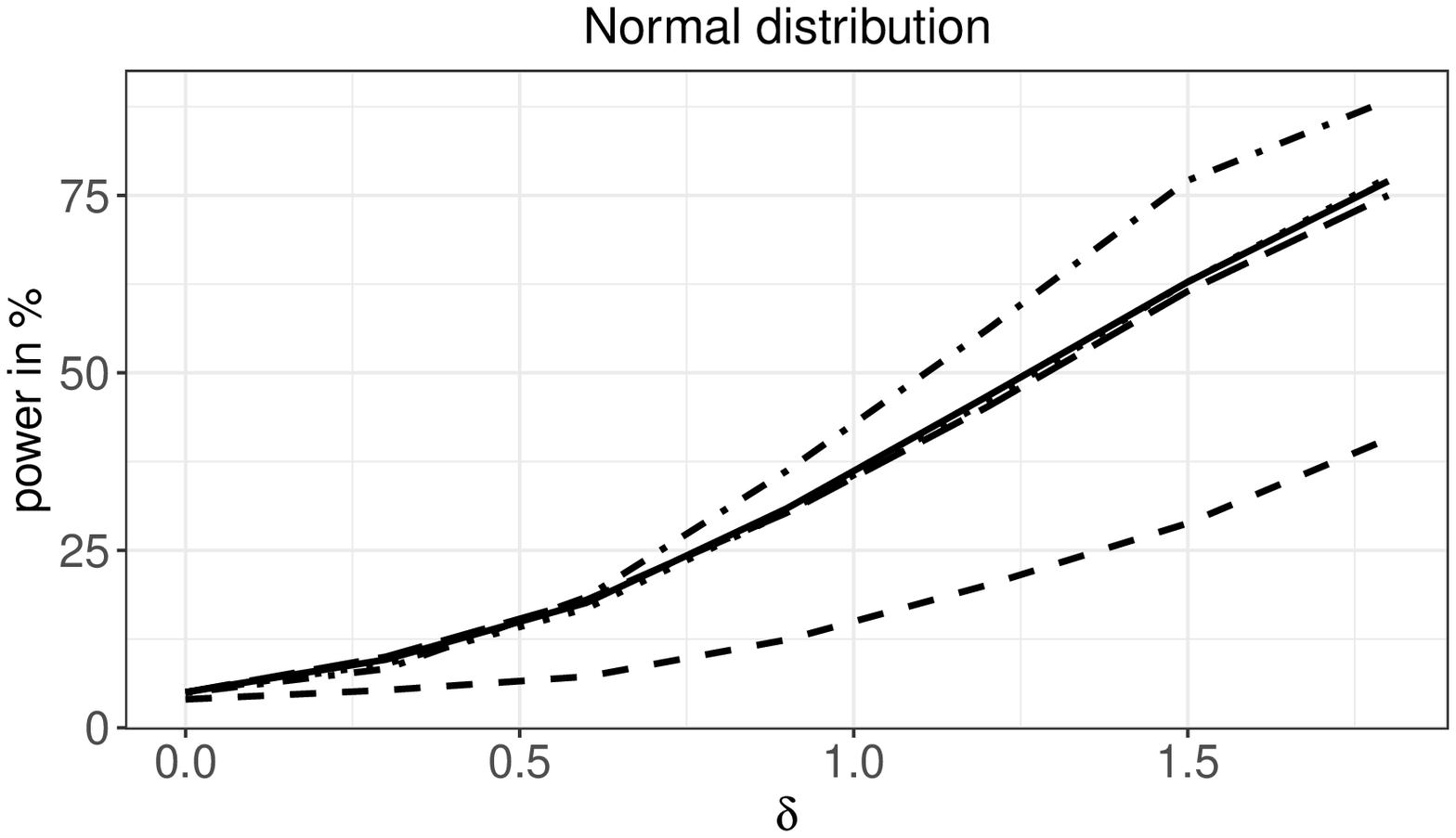}\\
	\includegraphics[ width =0.43\textwidth]{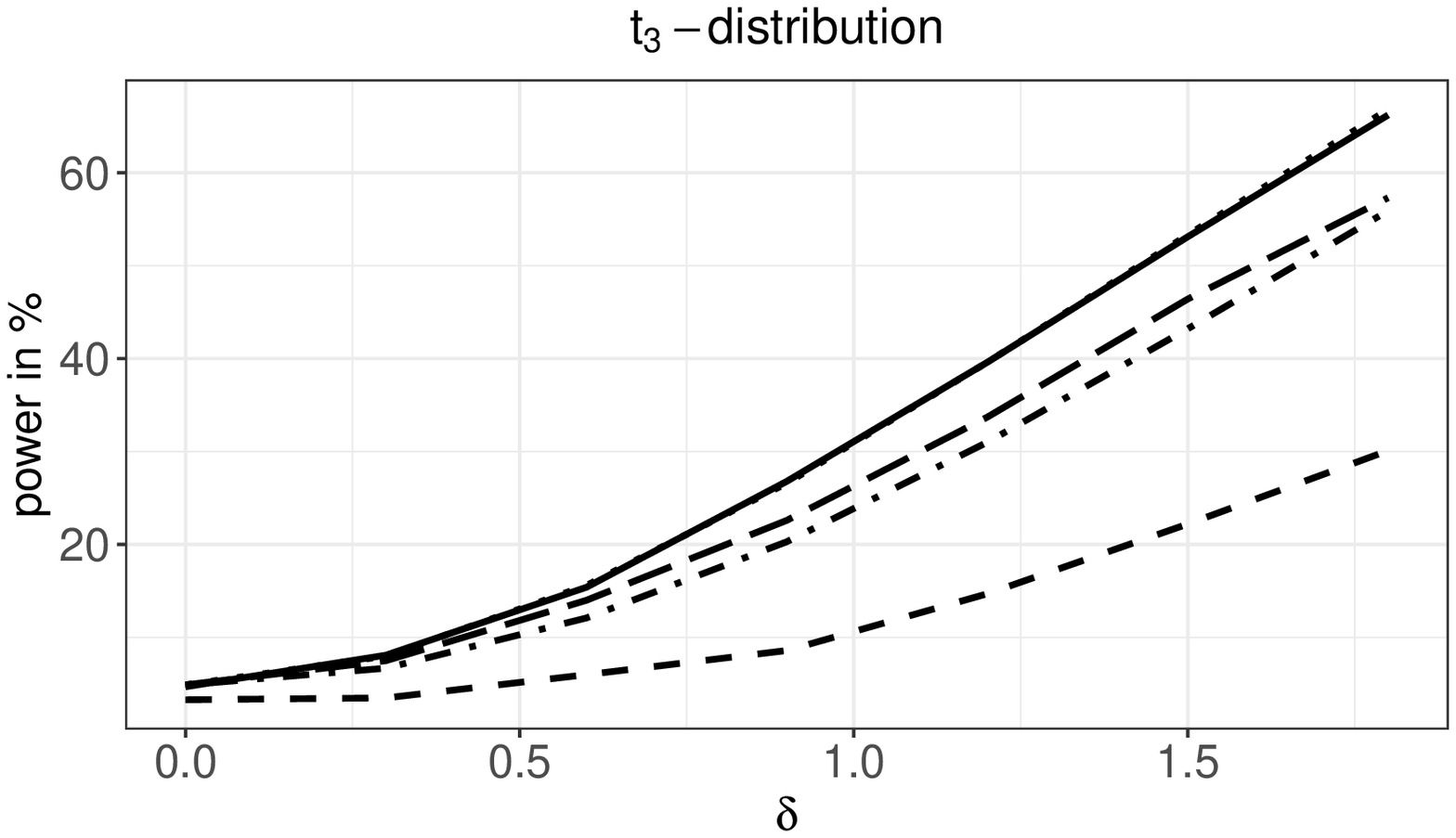}\qquad
	\includegraphics[ width = 0.43\textwidth]{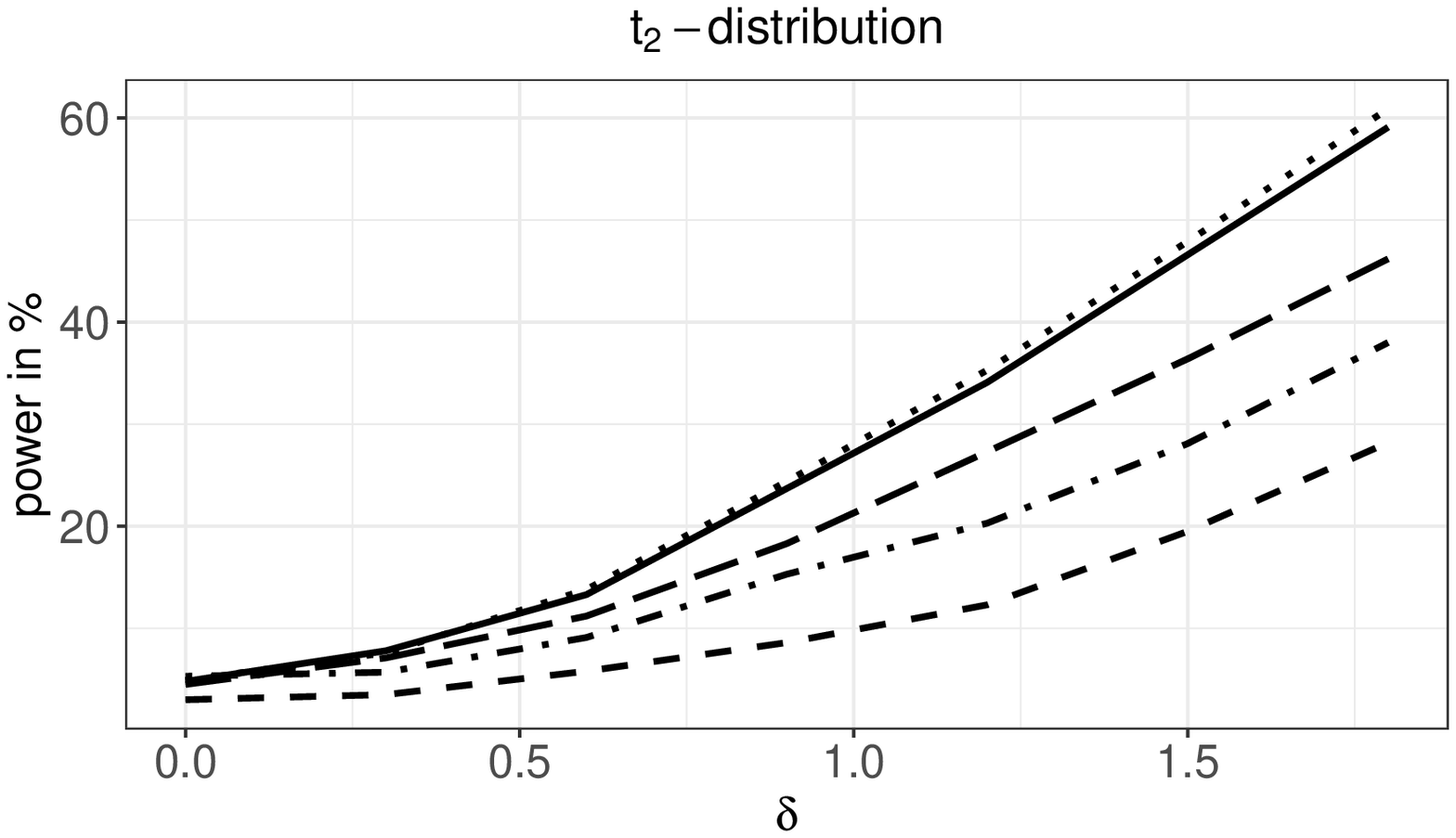} \\
	\underline{\text{IQR}}\\
	\includegraphics[ width =0.43\textwidth]{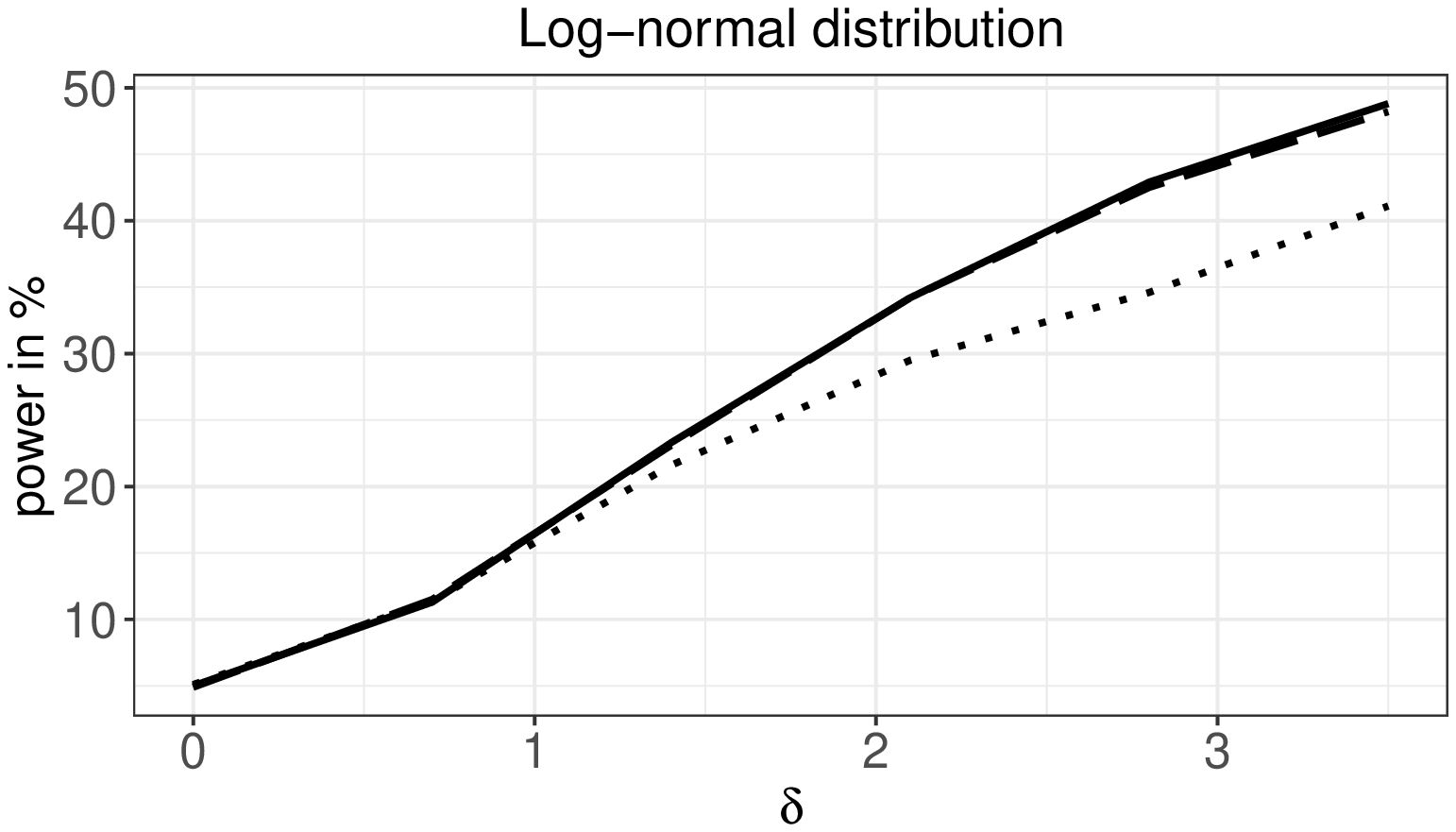}\qquad
	\includegraphics[ width = 0.43\textwidth]{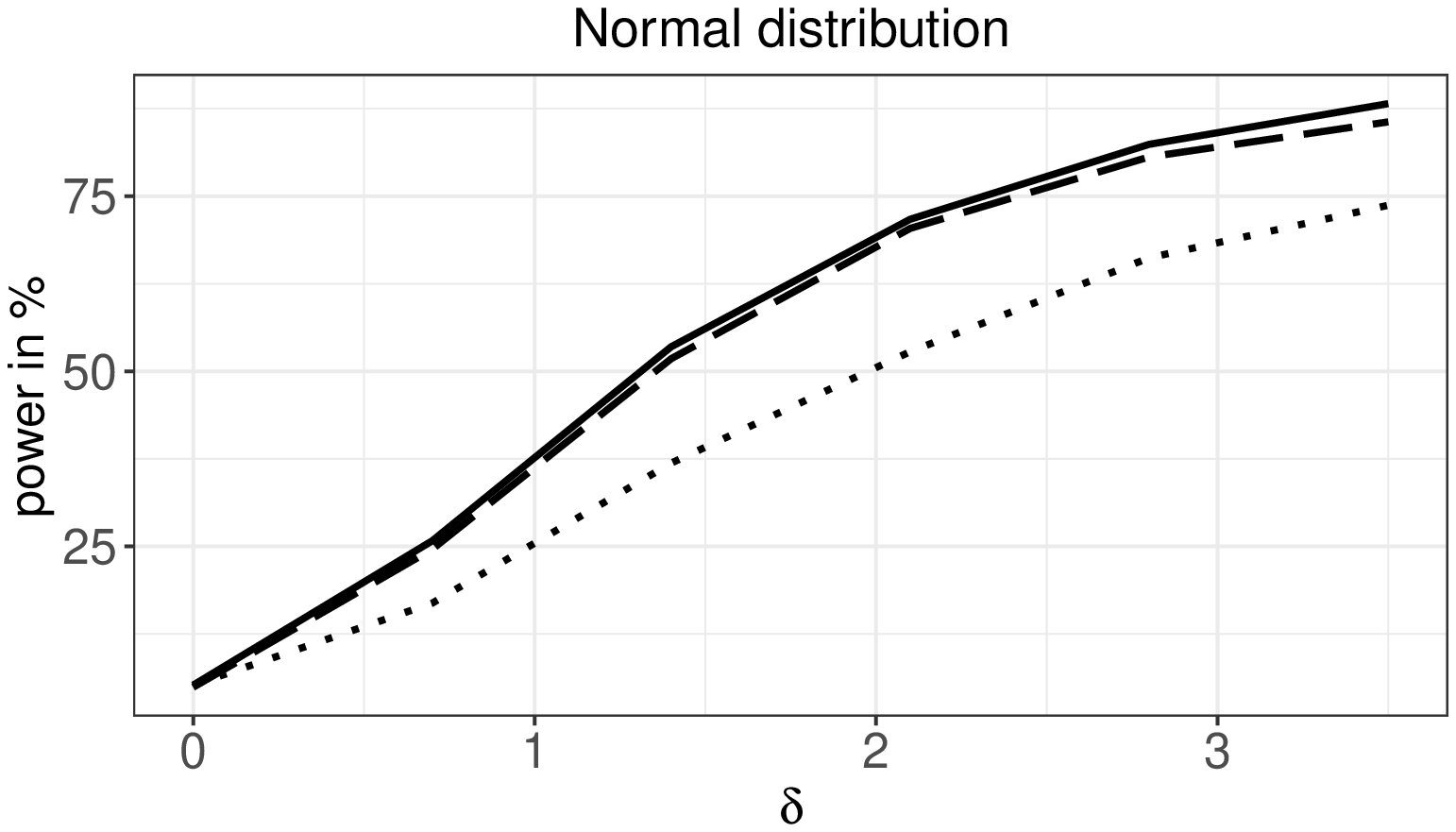}\
	\caption{Power curves for the $2\times 2$-median testing problem (first two rows) and for the 4-sample IQR testing problem (last row) of the permutation PBK test (dash-dotted), the wild bootstrap quantile regression test (dashed), and the three permutation tests based on interval-based (long-dashed), kernel density (dotted) and bootstrap (solid) covariance matrix estimation, resp., for $\mathbf{n}=\mathbf{n_2}$, $\boldsymbol{\sigma}=\boldsymbol{\sigma_1}$ and shift alternatives $\boldsymbol{\mu}=(0,0,0,\delta)$ (median) or scale alternatives $\boldsymbol{\sigma}=(1,1,1,1+\delta)$ (IQR) }\label{fig:median_power} \label{fig:IQR_power}
\end{figure}

To achieve a scenario under the alternative in the $2\times 2$-median test setting, we disturbed the respective null set-up by adding a shift parameter $\delta=\mu_{2,2}$ to the last group. In addition to the three proposed permutation tests and the wild bootstrap approach in the context of quantile regression, we considered the permutation Wald-type test (PBK) of \citet{paulyETAL2015} which was developed for testing means in general factorial designs. Their procedure is implemented in the R package \textit{GFD} \citep{friedrich2017gfd}.
For a fair comparison, we included their PBK test just for the cases where
mean and median coincide, i.e., for the symmetric distributions. The results for the procedures inferring a main effect are presented in Figure~\ref{fig:median_power}, while the corresponding power curves of the interaction tests are shown in the appendix. Studying Figure~\ref{fig:median_power}, we observe that the PBK test leads to higher power values compared to our tests for the normal distribution settings but is less powerful under the $t_2$- and $t_3$-distributions. An explanation may be given by the (asymptotic) efficiencies of the location estimators: While the sample mean is more efficient than the sample median under normal distributions  the situation is reversed for the two more heavy-tailed $t$-distributions. 
A comparison among the three median-based permutation tests shows that the interval-based approach leads to lower power values than the other two methods for both $t$-distributions, while the bootstrap approach is slightly less powerful than the other two tests in case of the skewed log-normal distribution. Under normality, however, the tests' power functions are almost identical. 
In comparison, the wild bootstrap quantile regression method has considerably less power than all other methods for testing main median effects. The power curves for the interaction effects presented in the appendix show a similar pattern for almost all tests. The only exception is the 
wild bootstrap approach which exhibits a similar power behavior as the permutation tests. Moreover, it is is even slightly advantageous for shift alternatives with $\delta>1$.

To obtain alternatives in case of the 4-sample IQR testing problem, we consider scale alternatives of the kind $\boldsymbol{\sigma}=(1,1,1,1+\delta)$. For ease of presentation we only show the results for normal as well as lognormal distributions here. The resulting  power curves are plotted in Figure~\ref{fig:IQR_power}. We can observe that the kernel density approach leads to lower power values compared to the other two methods; especially in case of small sample sizes.

\begin{figure}
	\centering
	\includegraphics[ width =0.48\textwidth]{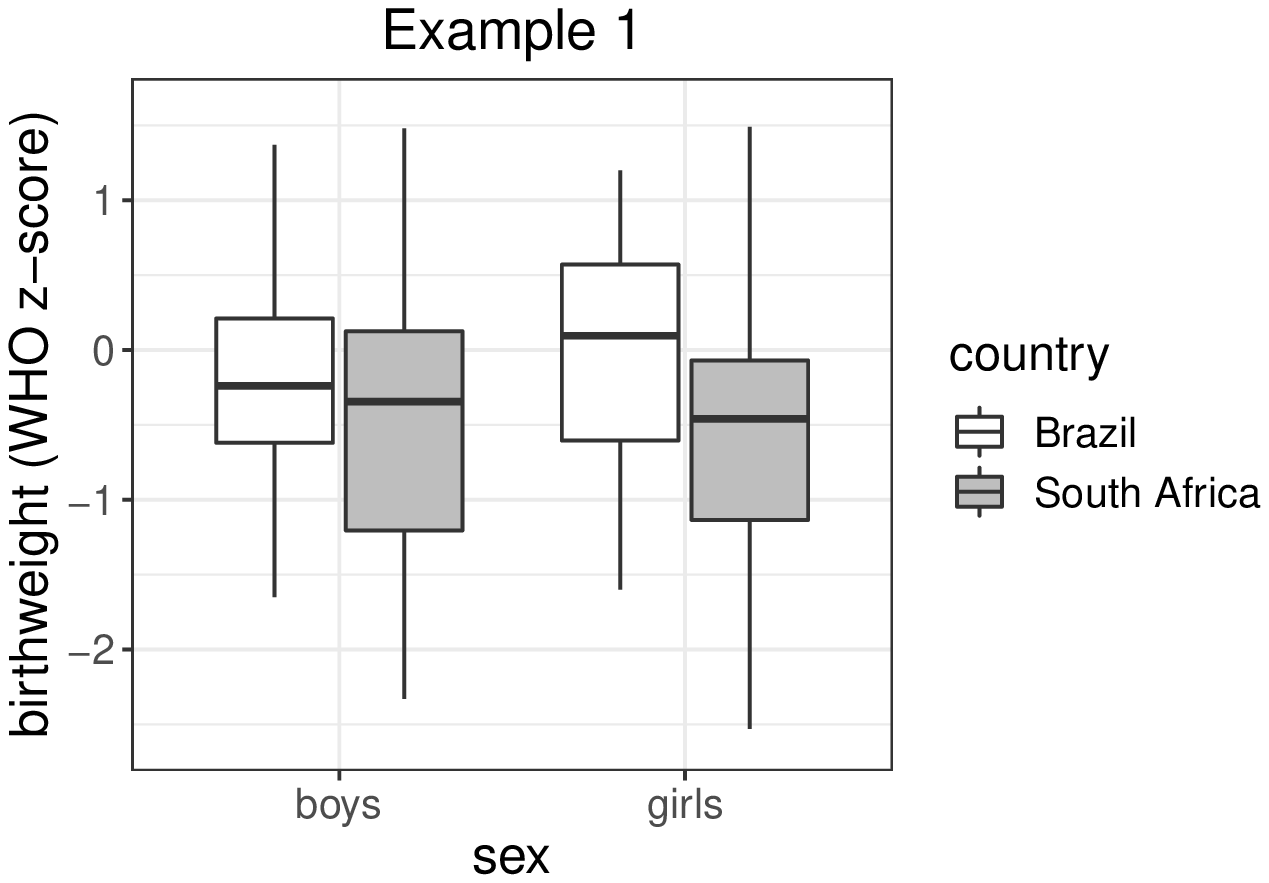}\quad
	\includegraphics[ width = 0.48\textwidth]{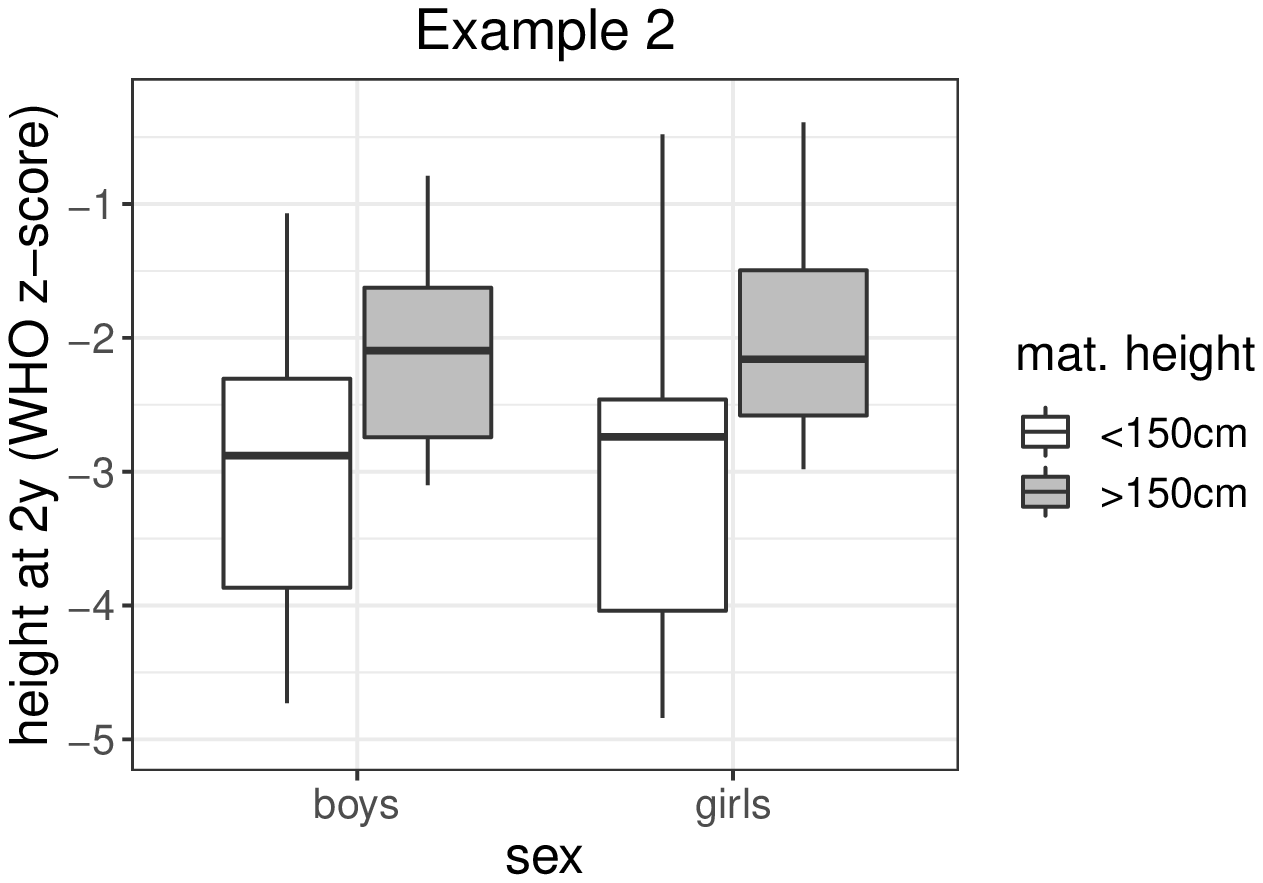}
	\caption{Group-wise boxplots (outliers are not displayed) for the birthweight data from Example 1 (left) and the height data from Example 2 (right).} \label{fig:boxplot}
\end{figure}

\noindent {\bf Recommendation.} Summarizing the findings we recommend the use of the permutation methods over their asymptotic counterparts as they show a much better type-$I$-error control in case of small and moderate sample sizes ($n_i\leq 200$). However, there is no general recommendation for choosing between the three permutation versions as their power behavior (slightly) differed with respect to underlying settings, e.g. for comparing IQRs the interval- and bootstrap-based approaches performed better while the kernel method exhibit the largest power for testing medians in a $2\times 2$ design with heavy tails.

\section{Illustrative data analysis}\label{sec:real_data}

A typical everyday situation in which we are confronted with quantiles are {\it percentile curves} for child heights and weights.
We re-analyzed growth and weight data of children from 5 different sites (Brazil, India, Guatemala, the Philippines, and South Africa), which was provided to us by the COHORTS group \citep{richter:ETAL:2011}. Both, height and weight, were converted to z-scores regarding the WHO child standards \citep{WHO,deOnis:ETAL:2007_WHO}. Having a comparison of percentile curves in mind, we test for effects in three quantiles simultaneously, considering the $25\%$-, $50\%$- and $75\%$-quantile. In addition, this also demonstrates the flexibility of the proposed methodology.
For illustrative purposes we focus on specific subgroups in the COHORT data set in which interesting effects are present:

\textit{Example 1:} We compare the birth weight of firstborns from the countries (factor A) Brazil and South Africa including both genders (factor B). To avoid confounding effects regarding age, education or marital status, we restrict our analysis to 30-year-old or younger married mothers with a comparable education level of 9 completed school years. The resulting $n=173$ children is divided into $n_1=65$ boys and $n_2=46$ girls from Brazil, and $n_3=36$ boys and $n_4=26$ girls from South Africa. We would like to infer whether there are differences between the countries regarding the boys' and girls' birth weight, respectively.

\textit{Example 2:} We investigate the effect of the mother's height (factor A) on the children's height at the age of $2$ years. Both sexes (factor B) are included. We restrict to firstborns of unmarried mothers from the Philippines. For this analysis, we divide the women into the groups "small"  and "tall" consisting of the women respectively being smaller and taller than the median height of $150$cm. The group "small" consists of data for $n_1=8$ boys and $n_2=13$ girls, and in the group "tall" there are data for $n_3=12$ boys and $n_4=11$ girls.

To get a first graphical impression, the group-specific box plots are presented in Figure~\ref{fig:boxplot}.
In both cases it appears that factor A (country and maternal height, respectively) leads to a shift of all three empirical quantiles of the children's height and weight. To infer this conjecture we like to check for a main effect of factor $A$ regarding the three quantiles $\mathbf{q}_{i} = (q_{i1}, q_{i2},q_{i3})^\top, i=1,\dots,4$ corresponding to the probabilities $(p_1,p_2,p_3)=(0.25,0.5,0.75)$ simultaneously. That is, we test $\mathcal{H}_0: \{ \mathbf{q}_{1} + \mathbf{q}_{2} = \mathbf{q}_{3} + \mathbf{q}_{4}\}$. The $p$-values of all three asymptotic and permutation tests (ignoring multiplicity) are summarized in Table~\ref{tab:data_ex}.
\begin{table}[ht]
	\centering
	\caption{For the effect of the country on the birth weight (Example 1) and the maternal height on the height at 2 years, the p-values (in $\%$) are shown for our asymptotic and permutation approach using the interval-based (Int), kernel density (Ker) and bootstrap
		(Boo) strategies for covariance matrix estimation} \label{tab:data_ex}
	\begin{tabular}{r|rrr|rrr}
		\multicolumn{1}{c}{}&\multicolumn{3}{c}{Asymptotic} &  \multicolumn{3}{c}{Permutation}  \\
		& Int & Ker & Boo  & Int & Ker & Boo \\
		\hline
		Example 1 & 10.54 & 8.63 & 9.83 & 3.80 & 4.60 & 3.50  \\
		Example 2 & 10.95 & 9.19 & 8.43 & 3.30 & 6.75 & 3.60  \\
	\end{tabular}
\end{table}

It is apparent that the asymptotic and permutation test lead to different decisions at nominal level $\alpha=5\%$.
In fact, the seemingly present effect from Figure~\ref{fig:boxplot} is not detected by any asymptotic tests as their
$p$-values are around $8$--$10\%$. In contrast, the p-values of the permutation approaches are, except for the kernel density method in Example 2, less than $5\%$. To investigate the reasons why these decisions are so different, we conducted an additional simulation study for the three-quantile testing problem under the sample size settings of Example 2. The results are presented in the appendix and may  explain the above decisions to some extent. They can be summarized as follows: As in Section~\ref{sec:sim}, the asymptotic tests are quite conservative with type-1 error rates ranging between $0.8\%$ and $4.2\%$. Moreover, the permutation kernel density approach is less powerful than the other two permutation methods under shift alternatives for skewed distributions.

Beyond hypothesis testing, the theoretical results can also be used to formulate asymptotically valid confidence regions for contrasts of quantiles by inverting the corresponding tests. We exemplify this for the difference between two quantiles as effect parameter of interest. 
To this end, consider Example 1 and encode factor A (country) and factor B (gender) as follows: $i_2=1$ for the boys, $i_2=2$ for the girls, $i_1=1$ for Brazil and $i_1=2$ for South Africa. Then, for a fixed gender $i_2$, the asymptotic correct $z$- and permutation-$(1-\alpha)$-confidence intervals for the difference $\theta_{i_2} = q_{2i_2} - q_{1i_2}$ of the countries' quantiles (e.g. the medians) are 
\begin{align*}
&I_n = \Bigl[ (\widehat q_{2i_2} -\widehat q_{1i_2}) \pm  \frac{z_{\alpha/2} }{\sqrt{n} }\sqrt{\widehat \sigma_{1i_2}^2 +  \widehat\sigma_{2i_2}^2}   \Bigr], \quad
I_n^\pi = \Bigl[ (\widehat q_{2i_2} -\widehat q_{1i_2}) \pm  \frac{c_{ni_2}^\pi(\alpha/2) }{\sqrt{n} }\sqrt{\widehat \sigma_{1i_2}^2 +  \widehat\sigma_{2i_2}^2}   \Bigr],
\end{align*}
where $\widehat \sigma_{i_1i_2}^2 =  \boldsymbol{\widehat\Sigma}^{(i_1i_2)}_{11}$ is an estimator for the asymptotic variance of $\sqrt{n}(\widehat q_{i_1i_2} - q_{i_1i_2})$ using one of our strategies from Section \ref{sec:kernel}--\ref{sec:PB} and $c_{ni_2}^\pi(\alpha/2)$ is the $(1-\alpha/2)$-quantile of the permutation distribution of $\sqrt{n}(\widehat q_{2i_2}^\pi - \widehat q_{1i_2}^\pi)(\widehat \sigma_{1i_2}^{2,\pi} +  \widehat\sigma_{i_2}^{2,\pi})^{-1/2}$. To illustrate the application we calculated the $95\%$ permutation-based confidence-intervals for the median difference separately for gender in Table~\ref{tab:data_conf_int}. Ignoring multiplicity, we see that all three permutation procedures agree on a significant difference in the girl's median birthweight (at level $\alpha=5\%$) but do not find a corresponding effect for the boys.

\begin{table}[ht]
	\centering
	\caption{Point estimates $\widehat\theta$ for the difference $\theta_{i_2} = q_{2i_2} - q_{1i_2}$ of the countries' median with respect to sex for Example~1 together with permutation-based $95\%$ confidence intervals. Here, Int (interval-based), Ker (kernel density) and Boo (bootstrap) indicate the applied covariance matrix estimation technique.
	} \label{tab:data_conf_int}
	\begin{tabular}{c|cccc}
		Gender & $\widehat \theta$ & Int & Ker & Boo \\
		\hline
		Boys& $\shortminus$0.15 & [$\shortminus$0.44, 0.14] & [$\shortminus$0.48, 0.18] & [$\shortminus$0.43, 0.13]  \\[0.5em]
		Girls & $\shortminus$0.56 &  [$\shortminus$1.04, $\shortminus$0.08] & [$\shortminus$1.10, $\shortminus$0.02]& [$\shortminus$1.02, $\shortminus$0.10] \\[0.5em]
		\hline
	\end{tabular}
\end{table}

\section{Discussion}\label{sec:dis}
While an abundance of methods exists for inferring means and mean vectors in general heterogeneous factorial designs \citep{johansen:1980, brunner:dette:munk:1997, bathke:schabenberger:tobias:madden:2009, zhang:2012, konietschke2015parametric, paulyETAL2015, harrar2019comparison},
there are not so many methods for the analysis of medians or quantiles. To this end, we combined the idea of studentized permutations from heteroscedastic mean-based \citep{paulyETAL2015} and one-way median-based ANOVA \citep{chungRomano2013} to establish flexible methods for inferring quantiles in general factorial designs which we coin QANOVA. In fact, we proposed three different permutation methods in Wald-type statistics that only differ in the way the covariance matrix is estimated. All of them are applicable to construct confidence regions and to test null hypotheses about arbitrary contrasts of different quantiles.

The resulting procedures are finitely exact under exchangeability of the data and shown to be asymptotically valid. In doing so, we had to extend some results about general permutation empirical processes and uniform Hadamard differentiability \citep{vaartWellner1996} that are of own mathematical interest. From these results we could not only deduce the asymptotic exactness under the null hypotheses but also prove results about the procedures' asymptotics under fixed and local alternatives. In the special case of the median and a bootstrap-based covariance estimator, these results even reveal new insights into the \citet{chungRomano2013} one-way permutation test.

In addition to these theoretical analyses, we also analyzed the procedures in extensive simulations presented in the paper and the appendix. Our results indicate an accurate type-$I$-error control for the permutation methods in almost all simulation settings.  Only in case of skewed distributions and small unbalanced samples with a heteroscedastic negative pairing, a slight liberality was found when testing for main effects in a $2\times 2$ design. Beyond this, we can recommend all three permutation methods with clear conscience and we are currently working on implementing them within an \textsc{R}-package. 
Moreover, we are also confident that the current results can be transferred to questions about related quantile-based estimands as, e.g., coefficients of quartile variation \citep{bonett2006}.

\appendix

\section{Additional simulation results}
To compare the asymptotic and permutation tests in terms of power, we display in Table~\ref{tab:median_power} the results for the shift alternatives described in Section~\ref{sec:sim_power}. For a fair comparison of the power values, we include also the sample size corrected versions of the asymptotic tests, i.e., we estimated the finite sample size quantile of the test statistic by $5000$ Monte--Carlo iterations under the respective null hypothesis ($\delta=0$). Of course, the sample size corrected tests serve just as (fairer) competitors but can not be applied in practice because they require the prior knowledge of the underlying null hypotheses. Except under log-normal distributions, the conservative type-error rates of the asymptotic tests cause lower power values compared to the permutation approaches, where the power values of the latter are, in almost all settings, very close to the one of the sample size corrected tests.

\begin{table}
	\centering
	\caption{Power values in $\%$ (nominal level $\alpha = 5\%$) for the $2\times 2$-median testing problem of the permutation Wald-type test (PKB) of \cite{paulyETAL2015} as well as our asymptotic and permutation tests using the interval-based (Int), kernel density (Ker) and bootstrap (Boo) approach for estimating the covariance matrix under $\mathbf{n}=\mathbf{n_2}$ and $\boldsymbol{\sigma} = \boldsymbol{\sigma_1}$ for shift alternatives $\boldsymbol{\mu}=(0,0,0,\delta)$}\label{tab:median_power}
	\begin{tabular}{ll|rrrr|rrrr|rrrr}
		\multicolumn{2}{c}{}&\multicolumn{4}{c}{Asymptotic} & \multicolumn{4}{c}{Permutation} & \multicolumn{4}{c}{Size corrected}  \\
		Distr & $\delta$  & Int & Ker & Boo & PBK & Int & Ker & Boo & PBK & Int & Ker & Boo & PBK  \\
		\hline
		$N_{0,1}$ &  0 &  2.2 & 3.8 & 2.9 & {5.2} & {4.6} & 4.1 & 4.2 & 4.2 & {5.4} & 4.6 & 4.5 & 5.3 \\
		&  0.3 &  3.9 & 7.9 & 5.8 & {9.1} & 8.3 & {8.6} & {8.6} & 8.2 & 9.7 & 9.1 & {9.8} & 8.2 \\
		&  0.6 &  8.4 & 14.6 & 12.3 & {20.6} & 17.9 & 16.8 & 17.8 & {19.0} & 17 & 15.7 & {18.3} & 17.3 \\
		&  0.9 &  17.9 & 28.5 & 25.3 & {37.4} & 31.4 & 31.8 & 32.3 & {35.4} & 31.2 & 29.8 & 31.5 & {32.9} \\
		&  1.2 &  29.5 & 43.4 & 39.4 & {59.8} & 44.0 & 45.1 & 46.3 & {57.3} & 43.4 & 44.4 & 45.4 & {55.6} \\
		&  1.5 &  47.5 & 60.1 & 55.5 & {78.2} & 61.0 & 63.0 & 62.1 & {76.5} & 61 & 61.7 & 61.8 & {77.0} \\
		&  1.8 &  61.5 & 75.4 & 70.2 & {90.1} & 74.6 & 76.6 & 76.3 & {89.2} & 77.4 & 76 & 76.2 & {89.5} \\
		\hline
		$t_2$ &  0 &  0.6 & 2.6 & 1.4 & {3.7} & 4.6 & {4.7} & {4.7} & 4.3 & {4.8} & 4.6 & 4.3 & 4.2 \\
		&  0.3 &  0.7 & 3.4 & 2.5 & {4.0} & 7 & {7.4} & {7.4} & 4.6 & 6.1 & {7.7} & 7.6 & 4.3 \\
		&  0.6 &  2.3 & {9.4} & 7.1 & 9.0 & 11.6 & 13.7 & {14.2} & 9.3 & 12. & {14.4} & 13.8 & 9.6 \\
		&  0.9 &  3.7 & 13.6 & 10.6 & {13.9} & 17.6 & {20.3} & 19.9 & 14.6 & 17.9 & {19.8} & {19.8} & 14.1 \\
		&  1.2 &  7.0 & {25.8} & 20.3 & 21.8 & 25.1 & {34.2} & 32.0 & 22.9 & 27.7 & {33.7} & 32.8 & 22.6 \\
		&  1.5 &  14.2 & {39.4} & 31.8 & 28.5 & 37.5 & {48.5} & 47.1 & 29.2 & 41.0 & {50.3} & 49.0 & 30.9 \\
		&  1.8 &  20.5 & {54.6} & 46.6 & 38.1 & 47.7 & {60.7} & 59.4 & 39.4 & 54.6 & {65.4} & 63.5 & 38.1 \\
		\hline
		$t_3$ &  0 &  1.1 & 3.8 & 2.6 & {5.9} & {6.0} & {6.0} & 5.6 & {6.0} & {6.5} & 5.8 & 5.5 & 5.0 \\
		&  0.3 &  2.2 & 5.4 & 5.0 & {7.3} & 9.0 & 8.2 & {9.4} & 7.0 & 9.1 & 8.5 & {9.5} & 7.0 \\
		&  0.6 &  3.7 & 10.6 & 9.4 & {13.6} & 14.6 & {16.5} & 16.2 & 13.1 & 13.6 & 14.6 & {15.7} & 13.5 \\
		&  0.9 &  8.0 & {20.8} & 16.9 & 20.5 & 23.2 & {26.6} & {26.6} & 19.9 & 23.8 & {25.3} & 24.8 & 17.9 \\
		&  1.2 &  13.8 & 30.5 & 26.3 & {30.6} & 33.6 & 38.9 & {39.0} & 29.4 & 35.6 & {39.4} & 39.0 & 30.3 \\
		&  1.5 &  23.6 & {47.0} & 42.8 & 43.2 & 45.0 & {53.3} & 52.8 & 42.7 & 47.1 & {53.7} & 53.0 & 41.6 \\
		&  1.8 &  33.7 & {60.6} & 56.2 & 57.8 & 56.6 & 66 & {66.1} & 56.8 & 59.4 & {67.3} & 66.6 & 56.2 \\
		\hline
		$LN_{0,1}$ &  0 &  4.3 & {4.6} & 2.9 & --- & {6.4} & 6.0 & 5.8 & --- & 5.6 & {6.1} & 5.8 & --- \\
		&  0.3 &  6.2 & {7.4} & 3.8 & --- & 10.7 & {11.0} & 9.7 & --- & 8.9 & {10.4} & 9.3 & --- \\
		&  0.6 &  15.4 & {15.9} & 10.2 & --- & 22.6 & {23.4} & 20.2 & --- & {20.2} & 20.0 & 19.6 & --- \\
		&  0.9 &  24.9 & {27.8} & 18.8 & --- & 33.4 & {33.6} & 29.4 & --- & 32.2 & {33.4} & 31.0 & --- \\
		&  1.2 &  43.8 & {45.4} & 35.4 & --- & {49.4} & 48.4 & 44.1 & --- & 49.5 & {51.7} & 48.1 & --- \\
		&  1.5 &  55.2 & {58.3} & 47.4 & --- & 54.8 & {55.0} & 52.7 & --- & 60.8 & {63.4} & 60.2 & --- \\
		&  1.8 &  68.5 & {70.1} & 60.9 & --- & {64.2} & 63.5 & 63.3 & --- & 74.2 & {74.8} & 73.8 & --- \\
	\end{tabular}
\end{table}

\begin{table}
	\centering
	\caption{Type-1 error rate in $\%$ (nominal level $\alpha = 5\%$) for
		testing the median null hypothesis of no interaction effect in the $2\times 2$ design for the rank-based (Rank) and wild bootstrap (Wild) quantile regression approach as well as all asymptotic and permutation tests using the interval-based (Int), kernel density (Kern) and bootstrap (Boot) approach for estimating the covariance matrix. Values inside the $95\%$ binomial interval $[4.4,5.6]$ are printed bold.}\label{tab:median_int_null}
	\begin{tabular}{lll|rrr|rrr|rr|l}
		\multicolumn{3}{c}{}&\multicolumn{3}{c}{Asymptotic} & \multicolumn{3}{c}{Permutation} & \multicolumn{2}{c}{Quantile reg.} & \\
		Distr & $\mathbf{n}$ & $\boldsymbol{\sigma}$   & Int & Kern & Boot & Int & Kern & Boot & Rank & Wild & Setting\\
		\hline		
		$N_{0,1}$ & $\mathbf{n_1}$ & $\boldsymbol{\sigma_1}$ & 2.5 & 4.0 & 3 & \textbf{4.7} & \textbf{5.0} & \textbf{4.8} & \textbf{5.0} & 6.5 & balanced homosc.  \\
		&  & $\boldsymbol{\sigma_2}$ &  2.5 & 3.9 & 3.2 & \textbf{5.3} & \textbf{5.4} & \textbf{5.2} & \textbf{5.2} & 7.6
		& balanced heterosc. \\
		& $\mathbf{n_2}$ & $\boldsymbol{\sigma_1}$ &  1.6 & \textbf{5.0} & 3.3 & \textbf{5.3} & \textbf{5.4} & \textbf{5.4} & 9.5 & 4.1 & unbalanced homosc. \\
		&  & $\boldsymbol{\sigma_2}$ &  2 & 4.3 & 3.0 & 6.4 & \textbf{5.4} & \textbf{5.4} & 9.2 & 3.8 & positive pairing \\
		&  & $\boldsymbol{\sigma_3}$ &  1.7 & \textbf{4.6} & 3.1 & \textbf{5.4} & 5.8 & \textbf{5.5} & 10.0 & \textbf{4.6} & negative pairing\\		
		\hline
		$t_3$ & $\mathbf{n_1}$ & $\boldsymbol{\sigma_1}$ & 1.6 & 2.6 & 2.2 & \textbf{5.3} & \textbf{5.0} & \textbf{4.9} & \textbf{4.7} & 6.3 & balanced homosc.  \\
		&  & $\boldsymbol{\sigma_2}$ & 2.1 & 3.1 & 3.0 & 5.9 & 5.9 & 5.7 & \textbf{5.6} & 6.1 & balanced heterosc. \\
		& $\mathbf{n_2}$ & $\boldsymbol{\sigma_1}$ & 0.7 & 2.7 & 1.9 & \textbf{5.0} & \textbf{4.7} & \textbf{4.7}  & 9.5 & 2.9 & unbalanced homosc. \\
		&  & $\boldsymbol{\sigma_2}$ & 1.0 & 2.8 & 2.2 & 6.4 & \textbf{5.3} & \textbf{5.6} & 10.1 & 2.9 & positive pairing \\
		&  & $\boldsymbol{\sigma_3}$ & 0.8 & 3.6 & 2.3 & \textbf{4.9} & 5.8 & \textbf{5.5} & 8.5 & 3.9 & negative pairing \\	
		\hline 		  	
		$LN_{0,1}$ & $\mathbf{n_1}$ & $\boldsymbol{\sigma_1}$ & \textbf{4.5} & 3.4 & 1.7 & \textbf{5.2} & \textbf{5.1} & \textbf{5.0} & \textbf{5.0} & 7.1 & balanced homosc.  \\
		&  & $\boldsymbol{\sigma_2}$ & \textbf{4.7} & 3.8 & 2.1 & 5.8 & 5.8 & \textbf{5.3} & \textbf{5.0} & 7.5 & balanced heterosc. \\
		& $\mathbf{n_2}$  & $\boldsymbol{\sigma_1}$ & 3.3 & 3.1 & 1.8 & \textbf{5.4} & \textbf{5.1} & \textbf{5.3} & 7.0 & \textbf{4.5}  & unbalanced homosc. \\
		&  & $\boldsymbol{\sigma_2}$ & 3.1 & 3.0 & 1.7 & 6.1 & \textbf{5.5} & \textbf{5.4} & 6.9 & 4.3 & positive pairing \\
		&  & $\boldsymbol{\sigma_3}$ &  2.8 & 2.9 & 1.6 & \textbf{5.6} & \textbf{5.4} & \textbf{5.3} & 6.1 & \textbf{4.6} & negative pairing \\  	
		\hline
		$\chi^2_3$ & $\mathbf{n_1}$ & $\boldsymbol{\sigma_1}$ & \textbf{5.0} & \textbf{4.9} & 2.8 & \textbf{5.5} & \textbf{5.4} & \textbf{5.4} & \textbf{5.3} & 6.6 & balanced homosc.  \\
		&  & $\boldsymbol{\sigma_2}$ & \textbf{5.0} & \textbf{4.8} & 3.0 & 5.9 & 5.9 & \textbf{5.6} & \textbf{5.6} & 7.7 & balanced heterosc. \\
		& $\mathbf{n_2}$ & $\boldsymbol{\sigma_1}$ & 3.3 & 4.0 & 2.7 & \textbf{4.8} & \textbf{4.7} & \textbf{4.7} & 8.6 & 4.1 & unbalanced homosc. \\
		&  & $\boldsymbol{\sigma_2}$ &  4.2 & 4.3 & 3.0 & 6.3 & 5.9 & 5.8 & 8.2 & 4.0 & positive pairing \\
		&  & $\boldsymbol{\sigma_3}$ & 3.4 & \textbf{4.4} & 3.0 & 6.2 & 6.1 & \textbf{5.6} & 8.2 & \textbf{4.5} & negative pairing \\
		\hline
	\end{tabular}
\end{table}

The power plots of the four permutation procedures and the wild bootstrap quantile regression method in case of the $2\times 2$-median interaction testing problem are presented in Figure~\ref{fig:median_int_power}. As mentioned in the paper the conclusions are similar to the one drawn for the respective test versions for main median effects with two exceptions: 1. the type-1 errors of the rank inversion quantile regression test changes from being rather conservative to being quite liberal in the unbalanced cases. 2. the wild bootstrap quantile regression test can now compete with our permutation procedures and is even slightly favorable for larger shirt alternatives $\delta>1$.

\begin{figure}
	\centering
	\includegraphics[ width =0.45\textwidth]{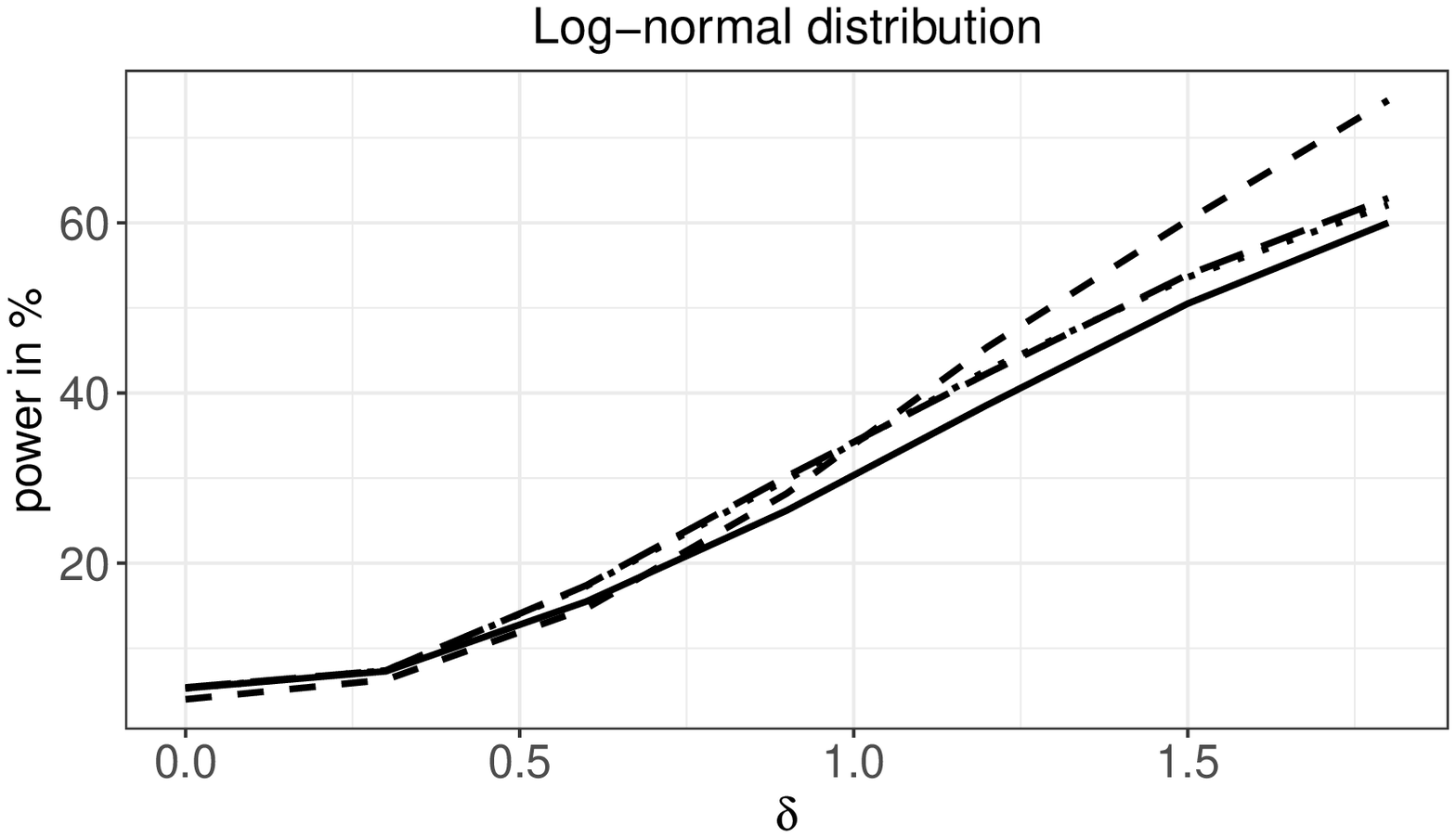}\qquad
	\includegraphics[ width = 0.45\textwidth]{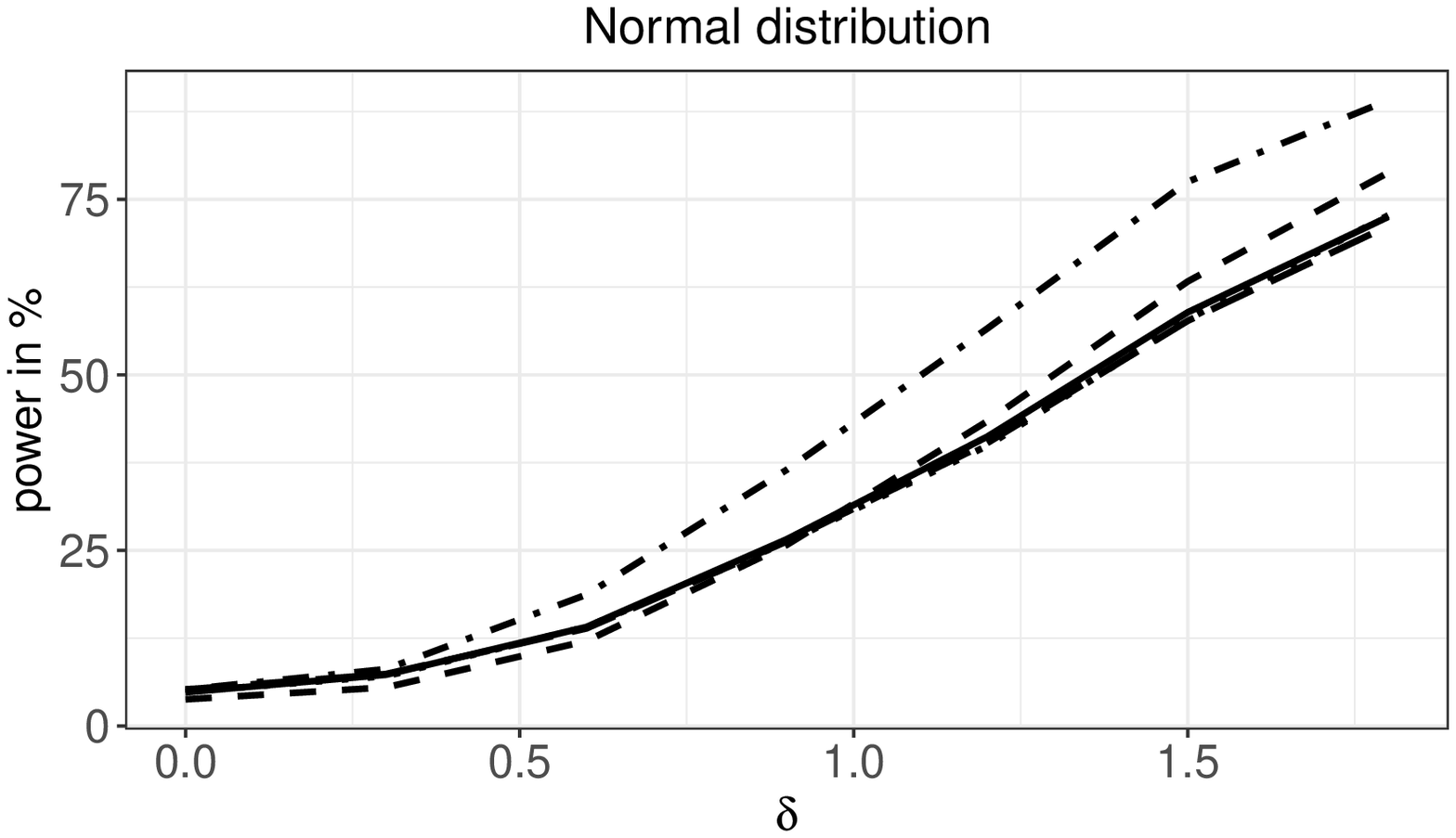}\\
	\includegraphics[ width =0.45\textwidth]{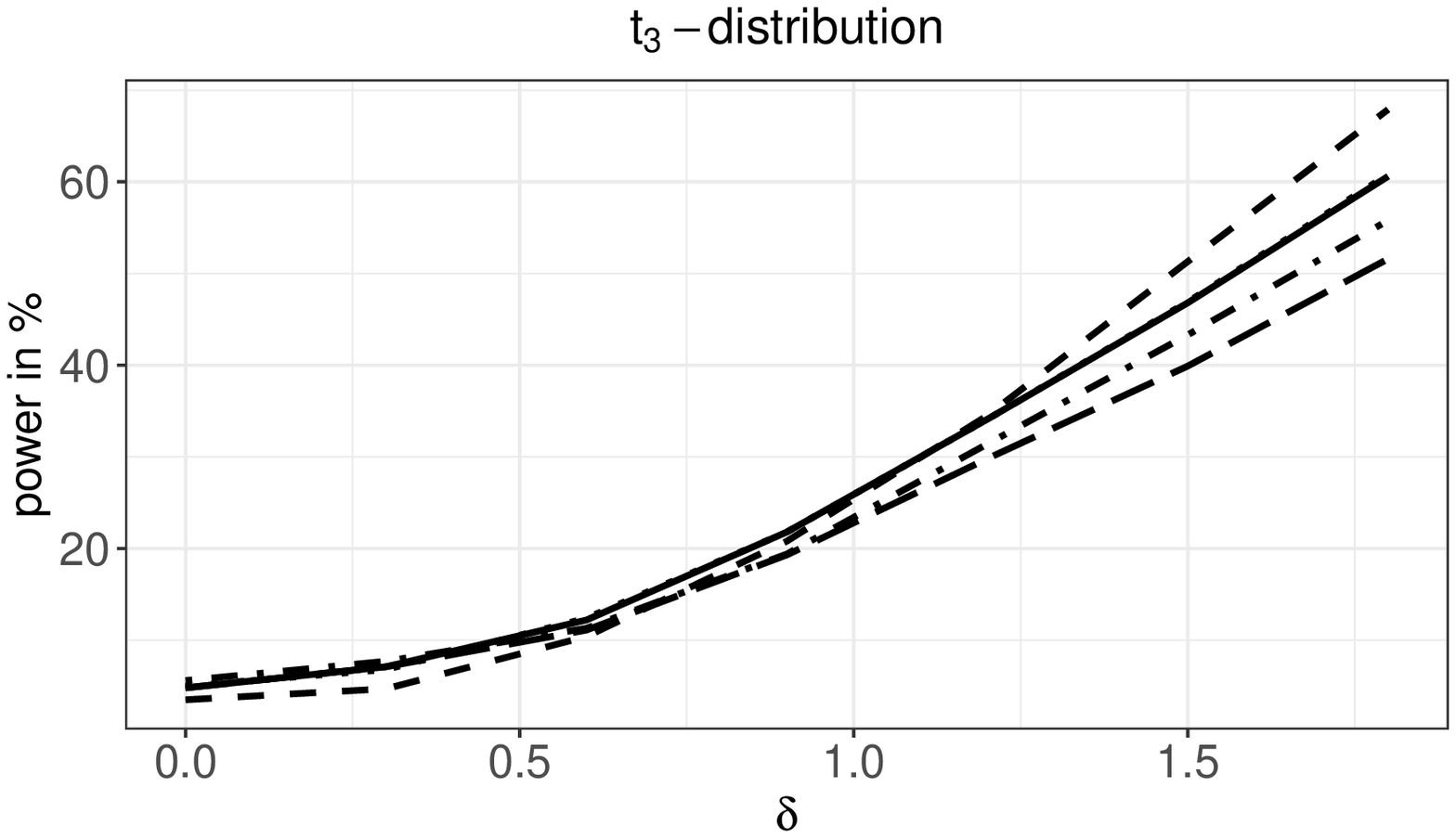}\qquad
	\includegraphics[ width = 0.45\textwidth]{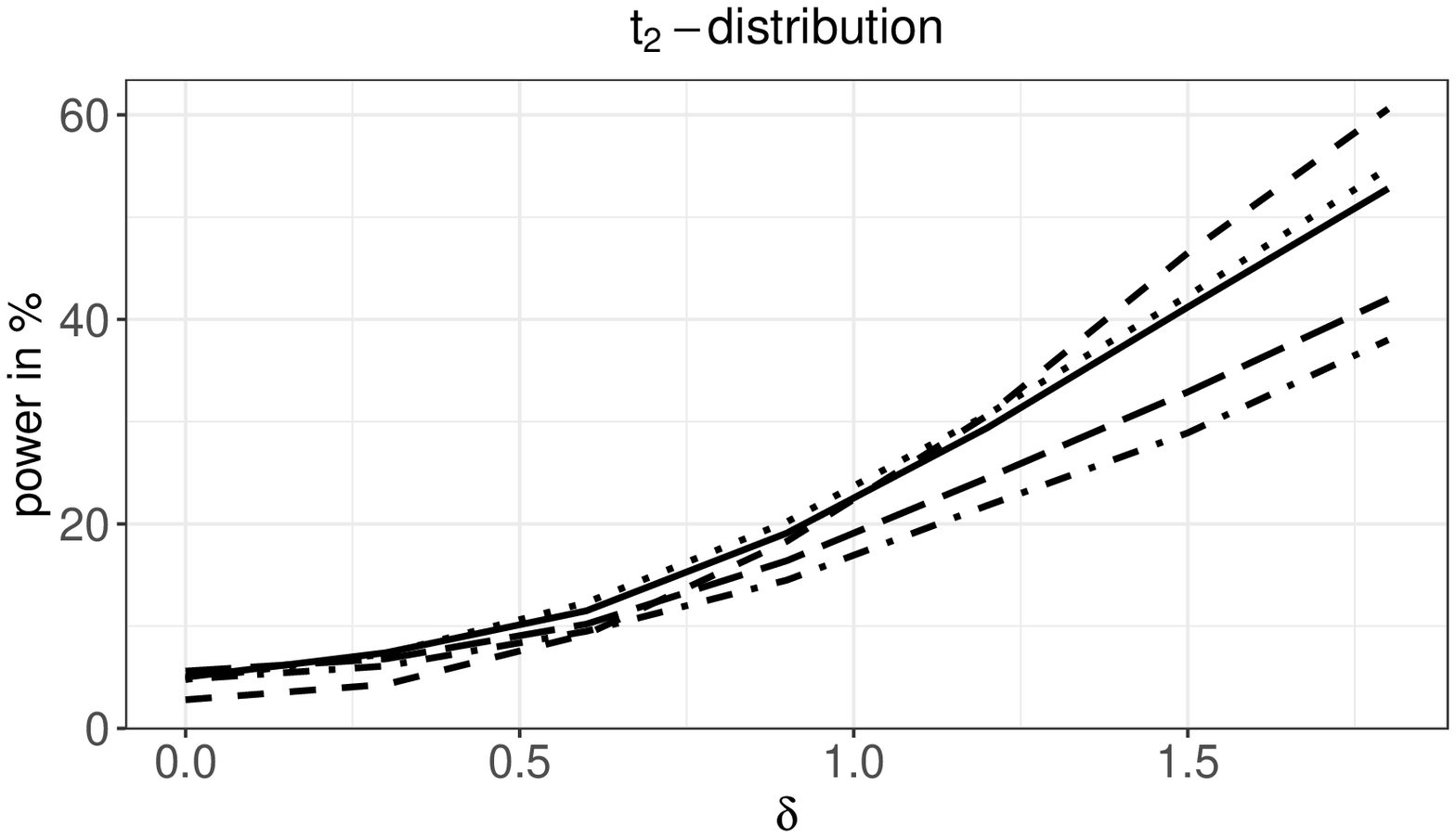}
	\caption{Power curves for the $2\times 2$-median interaction testing problem of the permutation PBK test (dash-dotted), the wild bootstrap quantile regression test (dashed) as well as the three current permutation tests based on interval-based (long-dashed), kernel density (dotted) and bootstrap (solid) covariance matrix estimation, resp., for $\mathbf{n}=\mathbf{n_2}$, $\boldsymbol{\sigma}=\boldsymbol{\sigma_1}$ and shift alternatives $\boldsymbol{\mu}=(0,0,0,\delta)$ }\label{fig:median_int_power}
\end{figure}

To illustrate the effect of increasing sample sizes on the type-1 error rate of the asymptotic tests, we conducted some additional simulations for the 4-sample IQR testing problem under equal sample size scenarios $n=(n_1,n_1,n_1,n_1)$ for growing $n_1$. The results for normal and log-normal distributions are presented in Figure~\ref{fig:IQR_incr_n}. It can be seen that the type-1 error rates get closer and closer to the $5\%$ benchmark line, where in comparison to the others this process is quite slow for the bootstrap approach under normal and for the kernel density method under log-normal distributions. In all, these plots strengthen our preference from Section~\ref{sec:sim_typ1} for the permutation approaches in case of small to moderate sample sizes.

\begin{figure}
	\centering
	\includegraphics[ width =0.45\textwidth]{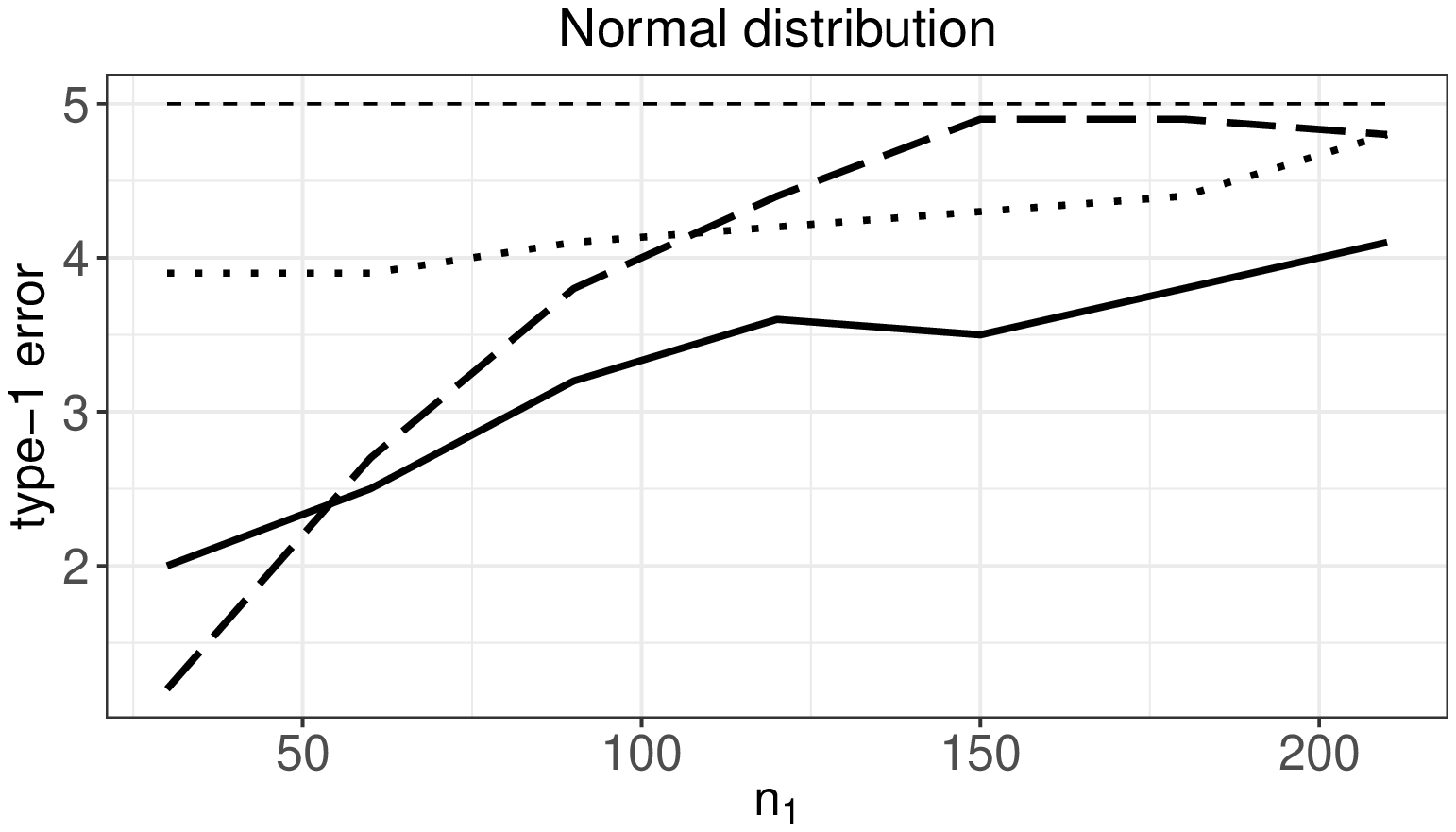}\qquad
	\includegraphics[ width = 0.45\textwidth]{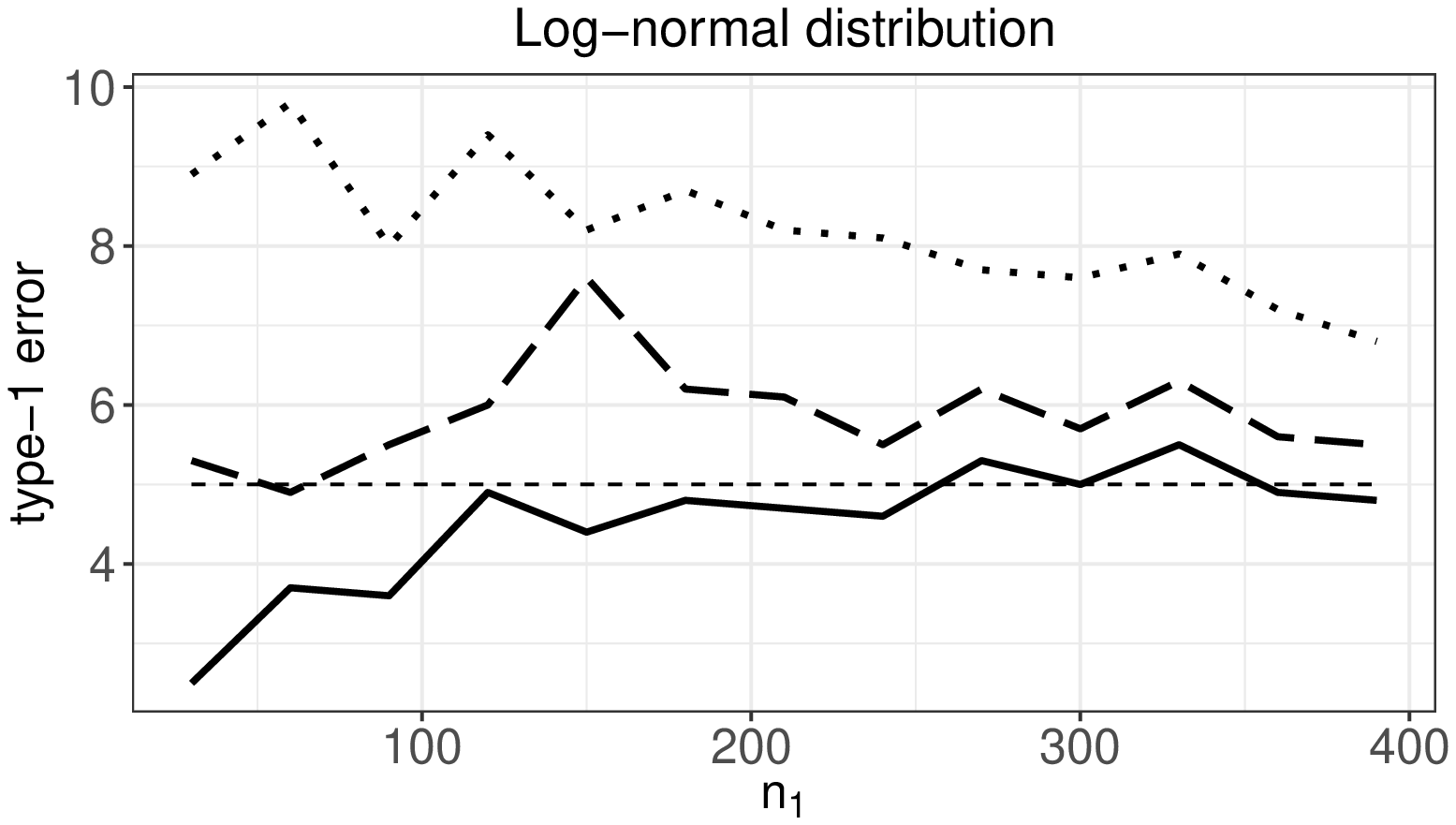}\\
	\caption{Type-1 error for the 4-sample IQR testing problem of our asymptotic tests based on interval-based (dashed), kernel density (dotted) and bootstrap (solid) covariance matrix estimation, resp., under normal (left) and log-normal distribution (right) for increasing balanced sample sizes  $\mathbf{n}=(n_1,\ldots,n_1)$   }\label{fig:IQR_incr_n}
\end{figure}

For a deeper understanding of the different test decision in our data analysis, we run additional simulations for the $2\times 2$ three-quantile test under the sample size situation $\mathbf{n}=(8,13,12,11)$ of Example 2. We compared the type-1 error rate of different distributions under shift alternatives $\boldsymbol{\mu}=(0,0,0,\delta)$. The resulting power curves are plotted in Figure \ref{fig:3quan_power} and the respective type-1 error rates ($\delta=0$) are separately displayed in Table~\ref{tab:3quan_null}. The asymptotic tests lead to quite conservative type-1 error rates reaching down to $0.6\%$, while the permutation tests keep the nominal level accurately with type-1 error rates between $4.2$--$5.3\%$. For the two skewed distributions, log-normal and $\chi_3^2$-distribution, the permutation kernel density approach exhibit a substantially lower power than the other two permutation approaches. In case of normal distributions, the power curves are very close to each other with a small benefit of the bootstrap method. Under the $t_3$-distribution it turns out that the interval-based approach leads to smaller power values compared to the other two tests.
\begin{table}
	\centering
	\caption{Type-1 error rate in $\%$ (nominal level $\alpha = 5\%$) for the three-quantile testing problem of our asymptotic and permutation tests using the interval-based (Int), kernel density (Ker) and bootstrap (Boo) approach for estimating the covariance matrix under the sample size setting $\mathbf{n}=(8,13,12,11)$}\label{tab:3quan_null}
	
	\begin{tabular}{l|rrr|rrr|}
		
		\multicolumn{1}{c}{}&\multicolumn{3}{c}{Asymptotic} & \multicolumn{3}{c}{Permutation}  \\
		Distr  & Int & Ker & Boo & Int & Ker & Boo \\
		\hline		
		$N_{0,1}$ & 1.5   &   3.9  &    1.7 &   5.1   &   5.2  &  5.1\\
		$LN_{0,1}$  & 2.0   &   2.6  &    0.6   & 5.0   &   4.3   &   4.3 \\
		$t_3$ & 0.9    &  3.2   &   1.2  &  5.0  & 5.3 & 5.3\\
		$\chi_3^2$ & 2.5  & 3.2 & 1.4 & 4.8 & 4.2 & 4.8
	\end{tabular}
\end{table}

\begin{figure}
	\centering
	\includegraphics[ width =0.45\textwidth]{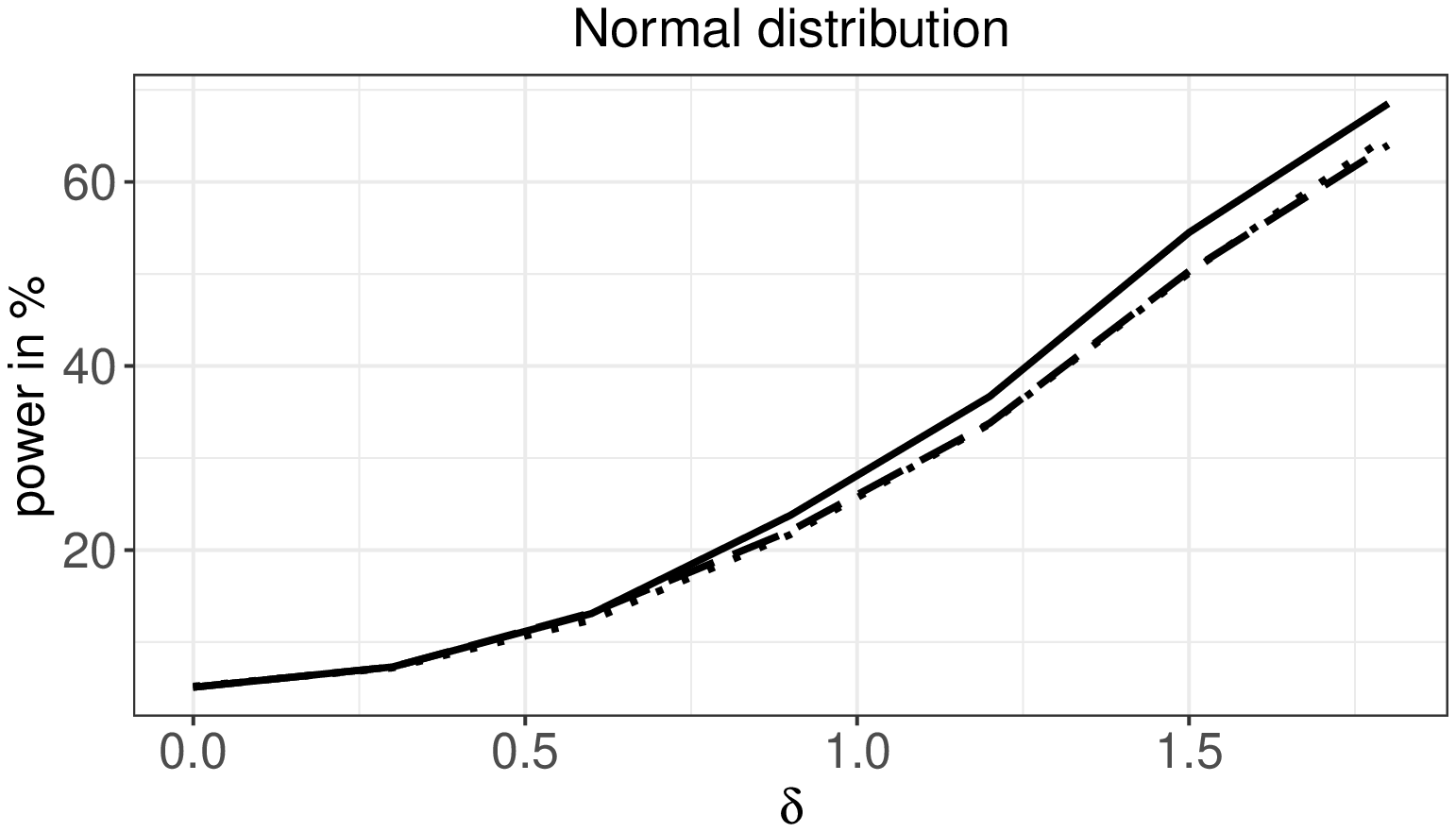}
	\qquad
	\includegraphics[ width =0.45\textwidth]{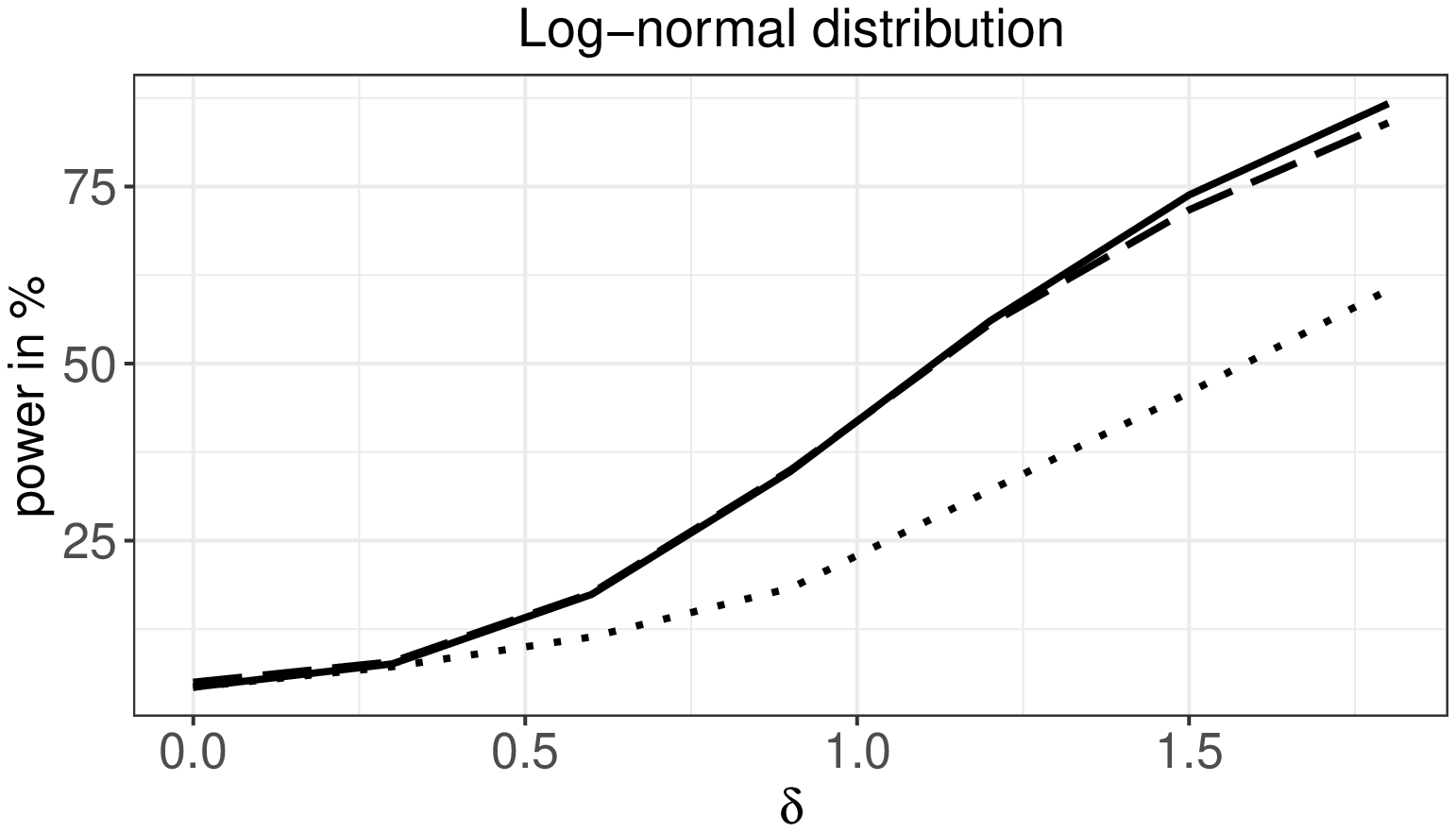}\\
	\includegraphics[ width = 0.45\textwidth]{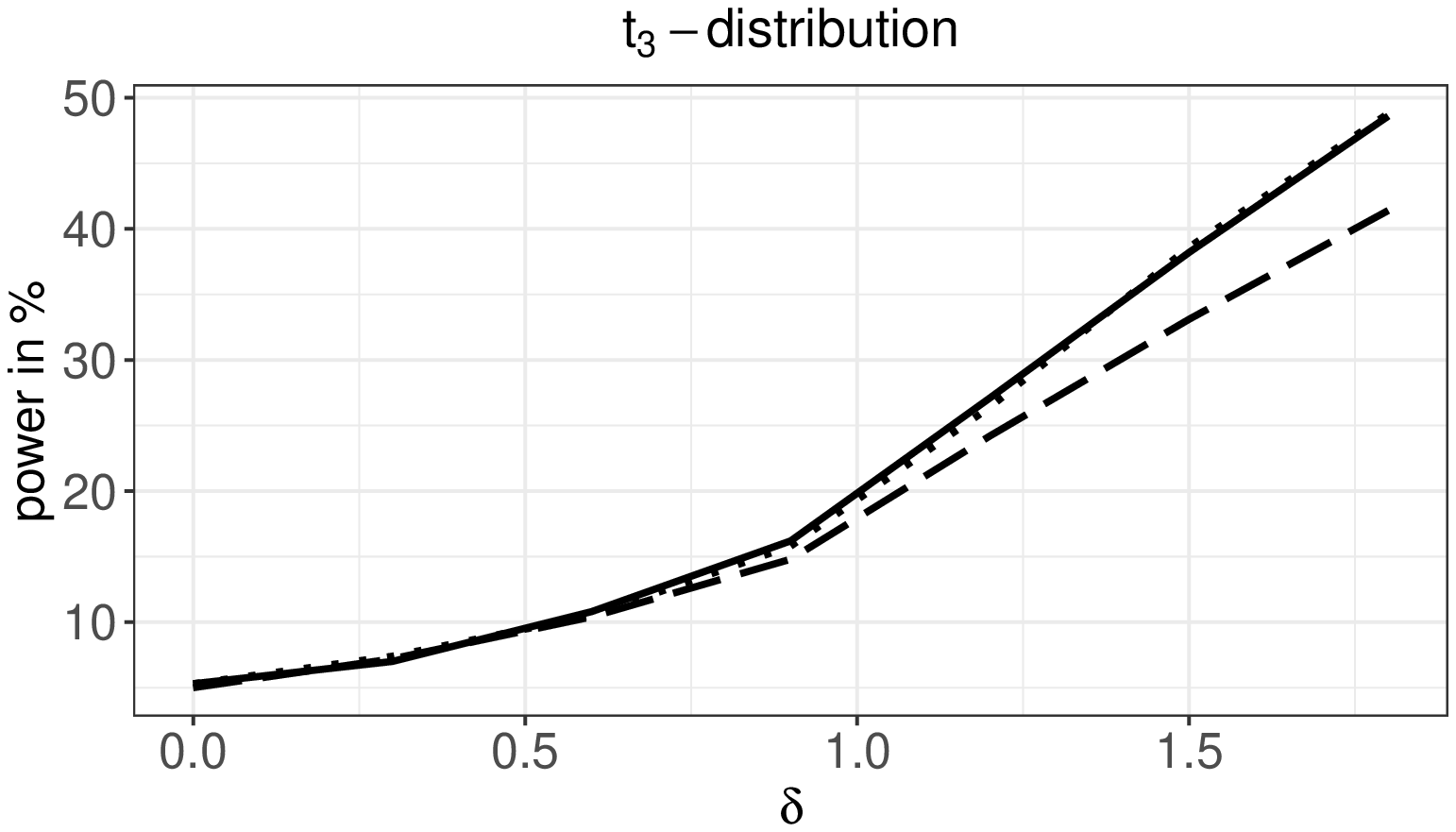} \qquad
	\includegraphics[ width = 0.45\textwidth]{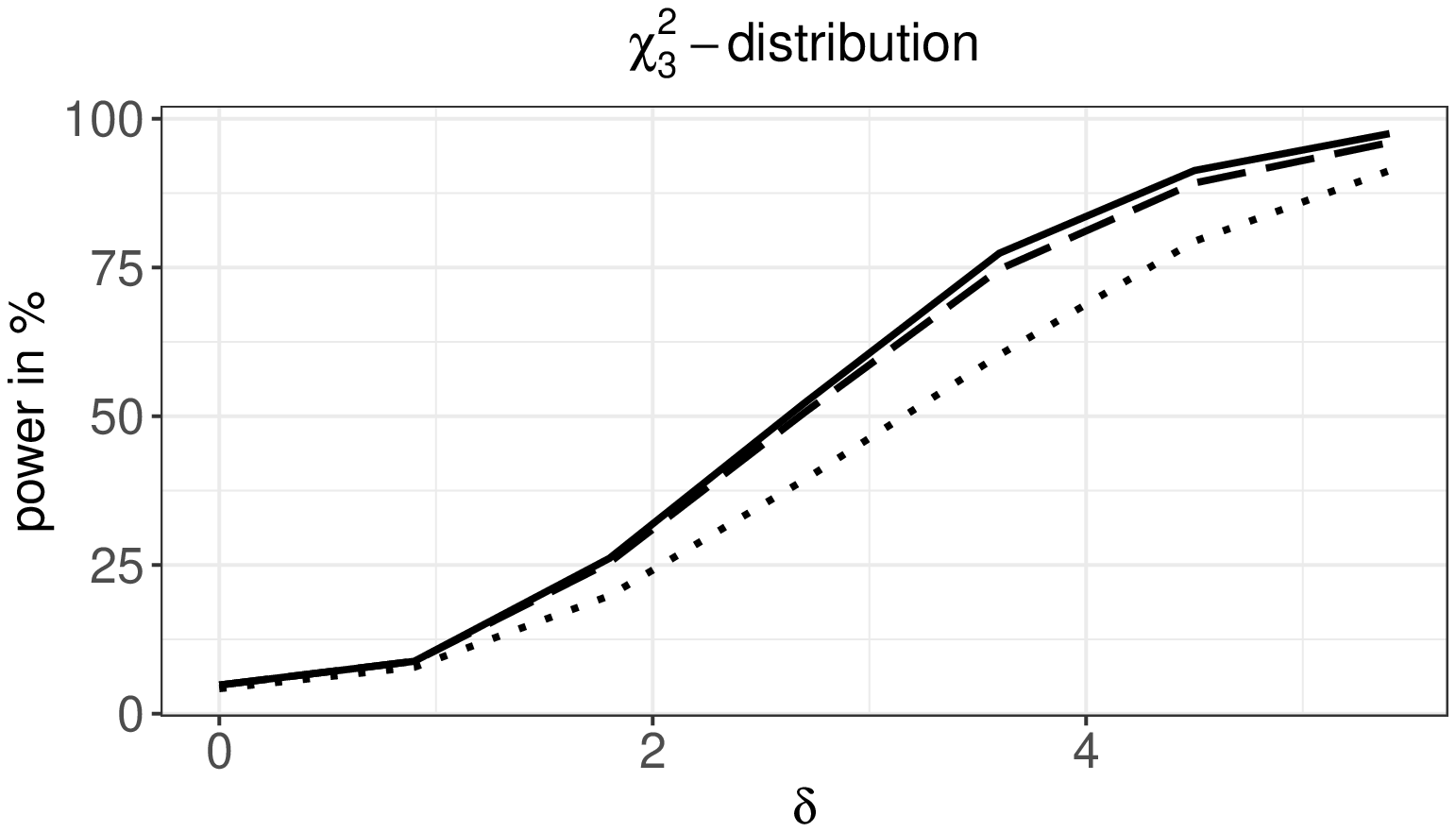}
	\caption{Power values for the three-quantile testing problem of our permutation tests based on interval-based (dashed), kernel density (dotted) and bootstrap (solid) covariance matrix estimation, resp., for $\mathbf{n}=(8,13,12,11)$, $\boldsymbol{\sigma}=\boldsymbol{\sigma_1}$ and shift alternatives $\boldsymbol{\mu}=(0,0,0,\delta)$ }\label{fig:3quan_power}
\end{figure}

\newpage
\section{Proofs}
\subsection{Proof of Proposition~{\ref{prop:uncond}}}
Although the statement was already proven by \cite{serfling2009}, we prefer to present the proof nevertheless for didactic reasons to prepare the proof for the permutation approach and the proof for the local alternatives. \\
Let $\mathbb{D}$ be the set consisting of all non-decreasing and right continuous functions $G:\R\to \R$. Clearly, $\mathbb{D}$ is a subset of the Skorohod space $D(\R)$ on $\R$, where we equip the latter with the sup-norm, as \cite{vaartWellner1996} did. For every $p\in (0,1)$ we define the corresponding inverse mapping $\Phi_p:\mathbb{D}\to\R$ \citep[compare to Section 3.9.4.2]{vaartWellner1996} by
\begin{align}\label{eqn:def_Phip}
\Phi_p(G) = G^{-1}(p) = \inf\{ t\in\R : G(t)\geq p\}.
\end{align}
If $G$ is differentiable at $q=G^{-1}(p)$ with positive derivative $g$ then $\Phi_p$ is Hadamard differentiable \cite[Lemma 3.9.20]{vaartWellner1996} at $G$ tangentially to the space $\mathbb{D}_q\subset D(\R)$, which consists of all bounded functions $\alpha\in D(\R)$ being continuous at $q$. The Hadamard derivative is given, in that case, by
\begin{align*}
\Phi'_{p,G}(\alpha) = -\frac{ \alpha(q)}{g(q)}.
\end{align*}
It is well known, see \cite{vaartWellner1996} or \cite{shorackWellner2009}, that
\begin{align}\label{eqn:process_conv_F_uncon}
n_i^{1/2}(\widehat F_{i} - F_i) \overset{\mathrm d}{\longrightarrow} B\circ F_i \text{ on }D(\R),
\end{align}
where $D(\R)$ denotes the Skorohod space on $\R$ equipped with the sup-norm and $B$ is a Brownian bridge on $[0,1]$. Applying the functional $\delta$-method \citep[Theorem 3.9.4]{vaartWellner1996} with the map $\Phi:\mathbb{D}\to \R^m$ given by $\Phi(G)=(\Phi_{p_1}(G),\ldots,\Phi_{p_m}(G))$ yields
\begin{align*}
n_i^{1/2}( \widehat F_i^{-1}(p_{r}) -  F_i^{-1}(p_{r}) )_{r = 1,\ldots,m} &\overset{\mathrm d}{\longrightarrow} \Bigl( -\frac{B(F_i(q_{ir}))}{f_i(q_{ir})} \Bigr)_{r = 1,\ldots,m} \\
&= \Bigl( -\frac{B(p_r)}{f_i(q_{ir})} \Bigr)_{r = 1,\ldots,m} \overset{\mathrm d}{=} \kappa_i^{1/2}\mathbf{Z}_i.
\end{align*}
It remains to prove that $\mathbf{\Sigma}^{(i)}$ is nonsingular. We define $\mathbf{A}_1=\kappa_i^{-1/2}\text{diag}(f_i(q_{i1})^{-1},\ldots,f_i(q_{im})^{-1})$ and $\mathbf{A}_2=\text{diag}(1-p_1,\ldots,1-p_m)$. Recall the well-known time transformation of a Brownian motion $W$ to obtain a Brownian bridge, i.e., we have $B(t) \overset{\mathrm d}{=} (1-t)W(t/(1-t))$. Consequently, $\mathbf{Z}_i$ has the same distribution as $\mathbf{A}_1\mathbf{A}_2 (W(t_1),\ldots,W(t_m))'$ with $t_i=p_i/(1-p_i)$, where $0<t_1<\ldots < t_m$. The covariance matrix $\mathbf{\Sigma}_W$ of $(W(t_1),\ldots,W(t_m))'$ is given by $\mathbf{\Sigma}_W=( \min(t_a,t_b))_{1\leq a,b\leq m}$ and its determinat equals
\begin{align*}
\text{det}(\mathbf{\Sigma}_W)=t_1(t_2-t_1)(t_3-t_2+t_1)\ldots(t_m-t_{m-1}+\ldots + (-1)^{m-1} t_1)>0.
\end{align*}
Consequently,  $\mathbf{\Sigma}_W$ and, thus,  $\mathbf{\Sigma}^{(i)}$ are nonsingular.

\subsection{Proof of Theorem \ref{theo:teststat_consist}}	
As explained before, Proposition \ref{prop:uncond} is also valid under alternatives and, consequently, it remains to prove that
$\mathbf{T}\mathbf{q}\neq \mathbf{0}_{km}$ always implies $( \mathbf{T} \mathbf{ q})' ( \mathbf{T} \mathbf{  \Sigma} \mathbf{T}' )^+ \mathbf{T} \mathbf{  q}>0$. Let $\mathbf{T}\mathbf{q}\neq \mathbf{0}_{km}$. Since  $\mathbf{  \Sigma}$ is nonsingular by Proposition \ref{prop:uncond}, the root $\mathbf{\Sigma}^{1/2}$ exists and is nonsingular as well. Moreover, there is some $\mathbf{\widetilde q}$ such that  $\mathbf{q}=\mathbf{\Sigma}^{1/2}\mathbf{\widetilde q}$. Recall the following well known properties of Moore--Penrose inverses: $(\mathbf{A}')^+ = (\mathbf{A}^+)'$, $(\mathbf{A}'\mathbf{A})^+=\mathbf{A}^+(\mathbf{A}')^+$ and $\mathbf{A}\mathbf{A}^+\mathbf{A} = \mathbf{A}$. For a detailed discussion of the Moore--Penrose inverse and more general inverses we refer to \cite{rao:mitra:1971}. By the properties mentioned before
\begin{align*}
&\mathbf{T}\mathbf{ \Sigma}^{1/2} \Bigl[ (\mathbf{T}\mathbf{ \Sigma}^{1/2})^+\mathbf{T}\mathbf{q} \Bigr] = \mathbf{T}\mathbf{ \Sigma}^{1/2}\mathbf{\widetilde q} = \mathbf{T}\mathbf{q} \neq \mathbf{0}_{km} \quad \\
\text{and, thus,}\quad &( \mathbf{T} \mathbf{ q})' ( \mathbf{T} \mathbf{  \Sigma} \mathbf{T}' )^+ \mathbf{T} \mathbf{  q} = ( \mathbf{T} \mathbf{ q})' ( \mathbf{  \Sigma}^{1/2} \mathbf{T}' )^+ ( \mathbf{T} \mathbf{  \Sigma}^{1/2} )^+ \mathbf{T} \mathbf{  q} \\
&\phantom{( \mathbf{T} \mathbf{ q})' ( \mathbf{T} \mathbf{  \Sigma} \mathbf{T}' )^+ \mathbf{T} \mathbf{  q}}= \Bigl[ (\mathbf{T}\mathbf{ \Sigma}^{1/2})^+\mathbf{T}\mathbf{q} \Bigr]'\Bigl[ (\mathbf{T}\mathbf{ \Sigma}^{1/2})^+\mathbf{T}\mathbf{q} \Bigr] >0.
\end{align*}

\subsection{Proof of Theorem \ref{theo:stat_perm}}	
The proof follows the same strategy as the one for Theorem \ref{theo:teststat_uncon}, compare to Section \ref{sec:asy_results}. First, we prove a multivariate central limit theorem for the permutation quantiles $\mathbf{\widehat q}^\pi$.
\begin{lemma}\label{lem:perm}
	Let $i\in\{1,\ldots,k\}$ and suppose that $n_i/n\to\kappa_i\in(0,1)$. Then under Assumption \ref{ass:densities_perm} we have
	\begin{align*}
	\sqrt{n}\Bigl( \widehat q_{ir}^{\:\pi} -  \widehat q_{ir} \Bigr)_{i=1,\ldots,k;r=1,\ldots,m} \overset{\mathrm d}{\longrightarrow} \mathbf{Z}^\pi,
	\end{align*}
	where $(C_{ir})_{i=1,\ldots,k;r=1,\ldots,m}$ abbreviates $(C_{11},C_{12},\ldots,C_{1m},C_{21},\ldots,C_{km})$ and $\mathbf{Z}^\pi$ is a zero-mean, multivariate normal distributed random variable with covariance matrix $\mathbf{\Sigma}^\pi$ given by its entries
	\begin{align*}
	&\mathbf{\Sigma}^\pi_{(ac), (bd) } = \E( \mathbf{Z}_{ac}^\pi \mathbf{Z}_{bd}^\pi) = \gamma^\pi(c,d)\frac{ p_a\wedge p_b - p_ap_b}{f(q_{a})f(q_{b})},\\
	&\text{where}\quad \gamma^\pi(c,d) =  \frac{1}{\kappa_c} \mathbf{1}\{ c = d\} - 1 \qquad(a,b\in\{1,\ldots,m\}; c,d\in\{1,\ldots,k\}).
	\end{align*}
\end{lemma}
The proof is given in Section \ref{sec:proof:lem:perm}.  Since the groups are, clearly, not independent within the permutation step, the limiting covariance matrix has not a block structure as $\mathbf{\Sigma}=\oplus_{i=1}^k\mathbf{\Sigma}^{(i)}$. But, due to $\mathbf{T}\mathbf{1} = \mathbf{0}_{km}$, we have
\begin{align*}
&\mathbf{T}\mathbf{\Sigma}^\pi \mathbf{T}'=\mathbf{T}\mathbf{\widetilde\Sigma}^\pi \mathbf{T}'\quad \text{ for }\mathbf{\widetilde\Sigma}^\pi = \bigoplus_{i=1}^k \mathbf{\widetilde\Sigma}^{(i),\pi},\\
&\text{where }
\mathbf{\widetilde\Sigma}^{(i),\pi}_{ab} = \kappa_i^{-1}\frac{1}{f(q_{a})f(q_{b})}( p_a\wedge p_b - p_ap_b )\qquad (a,b\in\{1,\ldots,m\}).
\end{align*}
Analogously as in the proof for Proposition \ref{prop:uncond}, we can deduce that $\mathbf{\widetilde\Sigma}^{(i),\pi}$ and, thus, $\mathbf{\widetilde\Sigma}^{\pi}$ are nonsingular matrices. Instead of estimating $\mathbf{\Sigma}^\pi$, we use the permutation counterpart $\mathbf{\widehat \Sigma}^\pi$ of $\mathbf{\widehat \Sigma}$, which estimates  $\mathbf{\widetilde\Sigma}^\pi$ consistently.
\begin{lemma}\label{lem:est_f_perm_consis}
	Let $\widehat\sigma_i(p_r)$ be the estimator defined in \eqref{eqn:def_hat_sigma_PB} and $\widehat\sigma_i^\pi(p_r)$ its permutation counterpart. Then under Assumption \ref{ass:densities_perm} we have the following conditional convergences given the data in probability:
	\begin{align}\label{eqn:est_f_perm}
	\widehat f_{ir}^\pi = \frac{\sqrt{p_r(1-p_r)}}{ \widehat\sigma_i^\pi(p_r)}\overset{p}{\rightarrow} f(q_r)\text{ and, thus, }\mathbf{\widehat \Sigma}^\pi\overset{p}{\rightarrow}\mathbf{\widetilde\Sigma}^\pi.
	\end{align}
\end{lemma}
\begin{lemma}\label{lem:est_f_kern_perm_consis}
	Let $\widehat f_{K,i}^\pi$ be permutation counterpart of the kernel density estimator from Section \ref{sec:kernel}. Then under Assumption \ref{ass:densities_perm} we have the following conditional convergences given the data in probability:
	\begin{align*}
	&\sup_{x\in\R} \Bigl | \widehat f_{K,i}^\pi(x) - f(x) \Bigr| \overset{p}{\rightarrow} 0 \text{ and, thus, }\mathbf{\widehat \Sigma}_K^\pi\overset{p}{\rightarrow}\mathbf{\widetilde\Sigma}^\pi.
	\end{align*}
\end{lemma}
Finally, combining all three lemmas, the continuous mapping theorem and Theorem 9.2.2 of \cite{rao:mitra:1971} proves Theorem \ref{theo:stat_perm}, compare to the argumentation in Section \ref{sec:asy_results}.

\subsubsection{Proof of Lemma \ref{lem:perm}}\label{sec:proof:lem:perm}

Basically, the proof follows the argumentation of the proof for Proposition \ref{prop:uncond}. Two aspects are more crucial than in the previous proof:
1. We do not know a concrete citation for the convergence of the empirical permutation process in case of more than two groups, i.e., the permutation version of \eqref{eqn:process_conv_F_uncon} has to be determined for $k\geq 3$. For the two-sample case ($k=2$) Theorems 3.7.1 and 3.7.2 of \cite{vaartWellner1996} can be applied. In the more general case $k\geq 2$, these results can be extended, that was already postulated by the authors \citep[Problem 3.7.2]{vaartWellner1996} but no specific details were given. 2. For the permutation approach we need uniform Hadamard differentiability instead of just Hadamard differentiability to apply the permutation $\delta$-method \citep[Theorem 3.9.5]{vaartWellner1996}. The extension of Lemma 3.9.20 of \cite{vaartWellner1996} to uniform differentiability is not straightforward but possible by a proper adjustment of the proof. 
\begin{lemma}\label{lem:emp_proc_F}
	Let $\mathbb{G}=(\mathbb{G}_1,\ldots,\mathbb{G}_k)$ be a zero-mean Gaussian process on $D(\R)^k$ with covariance structure
	\begin{align*}
	\E(\mathbb{G}_a(s) \mathbb{G}_b(t)) = \Bigl(\frac{1}{\kappa_a}\mathbf{1}\{a=b\} - 1\Bigr) \Bigl( F(s\wedge t) - F(t)F(s) \Bigr).
	\end{align*}
	for $a,b\in \{1,\ldots,k\}$ and $s,t\in\R$. Then
	\begin{align}\label{eqn:lem:emp_proc_F:result}
	n^{1/2} \Bigl( \widehat F^\pi_{ni} - \widehat F_{n} \Bigr)_{i=1,\ldots,k} \overset{\mathrm d}{\longrightarrow} \mathbb{G}\quad \text{on }D(\R)^k
	\end{align}
	almost surely given the observations.
\end{lemma}
\begin{lemma}[Uniform Hadamard differentiability]\label{lem:hada_diff}
	Let $G_n$ and $G$ be nondecreasing, real-valued functions. Moreover, let $G$ be continuously differentiable at $q = G^{-1}(p)$, $p\in(0,1)$, with positive derivative $g(q)>0$. Suppose that for some $M>0$
	\begin{align}
	&\sqrt{n}\sup_{ x \in \R} | G_n(x)-G(x)| \leq M \text{ and }\label{eqn:lem:hada:cond_Gn-G} \\
	&\sqrt{n}\sup_{|x|\leq K/\sqrt{n}} \Bigl |  G_n(q+x) -  G_n(q) - G(q+x) + G(q)\Bigr| \to 0 \label{eqn:lem:hada:cond_Gn-Gn-G+G}
	\end{align}
	for every $K>0$. Then
	\begin{align}\label{eqn:lem:hada_result}
	\sqrt{n} \Bigl( \Phi_p( G_n + n^{-1/2}h_n) - \Phi_p(G_n)  \Bigr) \to \Phi'_{p,G}(h) = - \frac{h(q)}{g(q)}
	\end{align}
	for every converging sequence $h_n$ such that $G_n + n^{-1/2}h_n\in \mathbb{D}$ and $h_n$ converges uniformly to $h\in \mathbb{D}_q$, where $\mathbb{D}_q$ consists all bounded functions being continuous at $q$.
\end{lemma}
The proofs for both lemmas are given subsequently. For our purposes, we make use of Lemma \ref{lem:hada_diff} with $G=F$, $q_r=G^{-1}(p_r)$ $(r=1,\ldots,m)$ and $G_n= \widehat F$ for fixed observations. Therefore, we need to ensure that \eqref{eqn:lem:hada:cond_Gn-G} and \eqref{eqn:lem:hada:cond_Gn-Gn-G+G} are fulfilled, where $M$ may depend on the underlying event $\omega\in\Omega$. By Lemma 1 of \cite{bahadur1966} we obtain for $a_n=n^{-1/2}(\log n)^{1/4}$ that almost surely
\begin{align}\label{eqn:sup_Fni(q+x)}
\sqrt{n}\sup_{|x|\leq a_n} \Bigl | \widehat F_{i}(q_{ir}+x) - \widehat F_{i}(q_{ir}) - F_i(q_{ir}+x) + F_i(q_{ir})\Bigr| \to 0
\end{align}
for every $i \in \{1,\ldots,k\}$ and $r\in\{1,\ldots,m\}$. We want to point out that the proof of \cite{bahadur1966} for his Lemma 1 is still valid for continuously differentiable distribution functions, a weaker assumption than the twice differentiability condition underlying his whole paper.  Note that due to its continuous differentiability $f_i$ is bounded in a neighborhood of $q_{ir}$ by $L_i$, say. Combining this, the mean value theorem, $\widehat F=\sum_{i=1}^k (n_i/n)\widehat F_{i}$ and \eqref{eqn:sup_Fni(q+x)} yields that almost surely
\begin{align}\label{eqn:serf}
&\sqrt{n}\sup_{|x|\leq K/\sqrt{n}} \Bigl | \widehat F(q_{ir}+x) - \widehat F(q_{ir}) - F(q_{ir}+x) + F(q_{ir})\Bigr| \nonumber \\
&\leq \sum_{i=1}^k \frac{n_i}{n}\, \sqrt{n}\sup_{|x|\leq a_n} \Bigl | \widehat F_{i}(q_{ir}+x) - \widehat F_{i}(q_{ir}) - F_i(q_{ir}+x) + F_i(q_{ir})\Bigr| \nonumber \\
&+ \sqrt{n} \sum_{i=1}^k \Bigl | \frac{n_i}{n} - \kappa_i \Bigr| \sup_{|x|\leq K/\sqrt{n}} \Bigl | F_i(q_{ir}+x) - F_i(q_{ir})\Bigr| \nonumber \\
&\leq o(1) +  \sqrt{n}\sum_{i=1}^k \Bigl | \frac{n_i}{n} - \kappa_i \Bigr| L_i \frac{K}{\sqrt{n}} \to 0.
\end{align}
Since we just want to show that the conditional convergence holds in probability given the observations, we can change the underlying probability space and even consider a triangular array $(X_{nij})_{i=1,\ldots,k;j=1,\ldots,n_i}$ of rowwise independent random variables $X_{nij}\sim F_i$. Hence, we can consider the special construction discussed in Section 3.1 of \cite{shorackWellner2009} for each group $i=1,\ldots,k$. In that case, \eqref{eqn:process_conv_F_uncon} even holds almost surely and not just in distribution. To be more specific, there are independent Brownian bridges $B_1,\ldots,B_k$ on $[0,1]$ and an appropriate triangular array $(X_{nij})_{i=1,\ldots,k;j=1,\ldots,n_i}$ such that almost surely
\begin{align*}
n^{1/2}(\widehat F_{i} - F_i)_{i=1,\ldots,k} \to \kappa_i^{-1/2}B_i\circ F_i.
\end{align*}
In particular, we have with probability one
\begin{align*}
&\limsup_{n\in\N}n^{1/2}\sup_{t\in\R}|\widehat F(t) - F(t)| \\
&\leq \sum_{i=1}^k \kappa_i^{-1/2}\sup_{t\in[0,1]}|B_i(t)| + \sum_{i=1}^k\limsup_{n\to\infty}n^{1/2} \Bigl | \frac{n_i}{n} - \kappa_i \Bigr|,
\end{align*}
where the latter sum is bounded by assumption and, thus, the complete right hand side is bounded for a fixed event $\omega\in\Omega$ with probability equal to one.
Last, we want to point out an effect caused by considering a triangular array, namely \eqref{eqn:serf} still holds in probability but not almost surely anymore. This difficulty can be solved by turning to subsequences. \\

From now on, we fix the observations. Due to the explanations above, we can assume without loss of generality that \eqref{eqn:lem:emp_proc_F:result}, 	 \eqref{eqn:lem:hada:cond_Gn-G} and \eqref{eqn:lem:hada:cond_Gn-Gn-G+G} hold with $G=F$, $G_n= \widehat F$ and $q_r=G^{-1}(p_r)$. Applying the (uniform) functional $\delta$-method \citep[Theorem 3.9.5]{vaartWellner1996} with the map $\Phi:\mathbb{D}^k\to R^{km}$ given by $\Phi(G_1,\ldots,G_k)=(\Phi_{p_1}(G_1),\ldots,\Phi_{p_m}(G_1),\ldots,\Phi_{p_m}(G_k))$ we obtain
\begin{align}\label{enu_quantil_conv_perm}
&n^{1/2}( (\widehat F_i^\pi)^{-1}(q_{ir}) -  \widehat F^{-1}(q_{ir}) )_{i=1,\ldots,k;r = 1,\ldots,m} \nonumber \\
&\overset{\mathrm d}{\longrightarrow} \Bigl( -\frac{\mathbb{G}_i(q_r)}{f(q_{ir})} \Bigr)_{i=1,\ldots,k;r = 1,\ldots,m}\overset{\mathrm d}{=} \mathbf{Z}^\pi.
\end{align}

\subsubsection{Proof of Lemma \ref{lem:est_f_perm_consis}}
Let the observations be fixed. As already discussed detailed in the proof of Lemma \ref{lem:perm}, we can assume without loss of generality that \eqref{eqn:lem:emp_proc_F:result}, 	\eqref{eqn:lem:hada:cond_Gn-G} and \eqref{eqn:lem:hada:cond_Gn-Gn-G+G} hold with $G=F$, $G_n= \widehat F$ and $q_r=G^{-1}(p_r)$ for every $r=1,\ldots,m$. Fix $k\in\{1,\dots,k\}$ and $r\in\{1,\ldots,m\}$. First, observe that
\begin{align*}
\big(2z_{\alpha_n^*(p)/2}+ 2n_i^{-1/2}\big)\widehat \sigma_i^{\pi,\text{PB}}(p_r) =  (\widehat F_{i}^\pi)^{-1}\bigl( p_r + k_{n,1} \bigr) - (\widehat F_{i}^\pi)^{-1}\bigl( p_r - k_{n,2} \bigr),
\end{align*}
where $n_i^{1/2}k_{n,j}\to \xi = z_{\alpha/2}(p_r(1-p_r))^{1/2}$. By the definition of $\Phi_p$, see \eqref{eqn:def_Phip}, we have
\begin{align*}
(\widehat F_{i}^\pi)^{-1}\Bigl( p_r + (-1)^{j+1}k_{n,j} \Bigr)  = \Phi_{p_r}( \widehat F_{i}^\pi - (-1)^{j+1}k_{n,j} ).
\end{align*}
From \eqref{eqn:lem:emp_proc_F:result} we can deduce that
\begin{align*}
n_i^{1/2}( \widehat F^\pi_{ni} - k_{n,1} - \widehat F, \widehat F^\pi_{ni} + k_{n,2} - \widehat F) \overset{\mathrm d}{\longrightarrow} (\kappa_i^{1/2}\mathbb{G}_i - \xi, \kappa_i^{1/2}\mathbb{G}_i + \xi) \quad \text{ on }D(\R)^{2}.
\end{align*}
Combining Lemma \ref{lem:hada_diff} and the (uniform) functional $\delta$-method \citep[Theorem 3.9.5]{vaartWellner1996} with the map $\Phi:\mathbb{D}^{2}\to \R^{2}$ given by $\Phi(G_1,G_{2})=(\Phi_{p_r}(G_1),\Phi_{p_r}(G_{2}))$ gives us
\begin{align*}
&n_i^{1/2} \Bigl( \Phi_{p_r}( \widehat F^\pi_{ni} + (-1)^{j}k_{n,j} ) - \Phi_{p_r}(\widehat F)  \Bigr)_{j=1,2}  \\
&\overset{\mathrm d}{\longrightarrow} -\frac{1}{f(q_{ir})}\Bigl( \kappa_i^{1/2}\mathbb{G}_i(q_r) + (-1)^{j} \xi \Bigr)_{j=1,2}.
\end{align*}
Altogether, we obtain in probability
\begin{align*}
\widehat \sigma_i^{\pi,\text{PB}}(p_r) \to \frac{1}{2z_{\alpha/2}} \frac{2\xi}{f(q_{ir})} = \frac{\sqrt{p_r(1-p_r)}}{f(q_{ir})}.
\end{align*}

\subsubsection{Proof of Lemma \ref{lem:est_f_kern_perm_consis}}
Let the observations be fixed and fix $i\in\{1,\ldots,k\}$. Subsequently, we use $\E(\cdot)$ as an abbreviation for the conditional expectation $\E(\cdot|\mathbf{X}_n)$ given the data $\mathbf{X}_n=(X_{ij})_{i=1,\ldots,k;j=1,\ldots,n_i}$. Similarly to the previous proofs, we can assume without loss of generality that \eqref{eqn:kernel_uniform} holds. Moreover, we can suppose that the conditional convergence in  \eqref{eqn:lem:emp_proc_F:result} holds. We adapt the proof idea of \cite{nadaraya:1965} for our purposes. In particular, the proof consists of two parts:
\begin{align*}
&\text{(i)}\quad V_{n,1}=\sup_{x\in\R}\Bigl | \widehat f_{K,i}^\pi(x) - \E(\widehat f_{K,i}^\pi(x)) \Bigr| \to 0 \text{ in probability}\\
&\text{(ii)}\quad V_{n,2}=\sup_{x\in\R}\Bigl | \E(\widehat f_{K,i}^\pi(x)) - f(x) \Bigr| \to 0.
\end{align*}
\underline{(i):} As abbreviation, define $K_{x,n,i}(u)=K( [x - u]/h_{ni})$. By assumption $K$ is of bounded variation and so is $K_{x,n,i}$, in particular, we have  $\int |\,\mathrm{ d }K_{x,n,i}|=\int |\,\mathrm{ d }K| =\mu_K < \infty$. Since $|K(x)|\leq |K(0)| + \mu_K$ holds, $K$ is also bounded in the classical sense. Now, observe that
\begin{align}\label{eqn:proof_kernel_perm_Ef}
\E(\widehat f_{K,i}^\pi(x)) = h_{ni}^{-1} \E \Bigl( K_{x,n,i}(X_{i1}^\pi) \Bigr) = h_{ni}^{-1} \int K_{x,n,i}(u)  \,\mathrm{ d }\widehat F_{i}(u).
\end{align}
Combining this, integration by parts \citep[Theorem A.1.2]{fleming:harrington:1991} and \eqref{eqn:lem:emp_proc_F:result} shows
\begin{align*}
V_{n,1} &= \sup_{x\in \R}\Bigl |h_{ni}^{-1} \int K_{x,n,i}(u) \,\mathrm{ d }( \widehat F_{i}^\pi - \widehat F_{i} )(u) \Bigr| \\
&= \sup_{x\in \R}\Bigl |h_{ni}^{-1} \int (\widehat F_{i}^\pi - \widehat F_{i} )(u-) \,\mathrm{ d }K_{x,n,i}(u)  \Bigr| \\
& \leq (h_{ni}^2n_i)^{-1/2}\mu_K n_i^{1/2}\sup_{x\in \R} \Bigr| \widehat F_{i}^\pi(x) - \widehat F_{i} (x) \Bigr| \to 0 \text{ in probability}
\end{align*}
because the assumptions on the bandwidth imply $h_{ni}^2n_i\to \infty$.

\underline{(ii):} From \eqref{eqn:proof_kernel_perm_Ef} it is easy to see that $\E(\widehat f_{K,i}^\pi(x))$ coincides with $f_{K,i}(x)$. Consequently, we obtain immediately from \eqref{eqn:kernel_uniform} that $V_{n,2}$ converges to $0$.

\subsubsection{Proof of Lemma \ref{lem:emp_proc_F}}
To verify the statement, we use empirical theory. For a detailed introduction into this field, we refer the reader to \cite{vaartWellner1996}.

Let $\epsilon_x$ be the Dirac measure centred at $x$, i.e., $\epsilon_x(A)=\mathbf{1}\{x\in A\}$.  For every group $i$, we introduce the group-specific empirical process ${\mathbb{P}_{i}}=n_i^{-1}\sum_{j=1}^{n_i}\epsilon_{X_{ij}}$ as well as its permutation counterpart ${\mathbb{P}_{i}^\pi}=n_i^{-1}\sum_{j=1}^{n_i}\epsilon_{X_{ij}^\pi}$ and the pooled process ${\mathbb{P}}= n^{-1}\sum_{i=1}^k\sum_{j=1}^{n_i}\epsilon_{X_{ij}}$. Moreover, let $P_{i}$ be the distribution of $X_{i1}$. We index the (empirical) measures $\mathbb P_{i}$, ${\mathbb{P}_{i}^\pi}$, $\mathbb{P}_{n}$ $P_{i}$ by the function class $\mathcal{F}=\{\mathbf{1}_{(-\infty,t]}:t\in\R\}$, which is a (universal) Donsker as well as a (universal) Glivenko-Cantelli class \citep[Examples 2.4.2 and 2.5.4]{vaartWellner1996}. To be more specific, we identify, e.g., ${\mathbb{P}_{i}}$ by $\{\int f \,\mathrm{ d }{\mathbb{P}_{i}}: f \in  \mathcal F\}$. In this way, we treat all of them as random elements of $l^\infty(\mathcal F)= \{Q\in \mathcal M_1(\R): \sup_{f\in\mathcal F}\int f \,\mathrm{ d }Q< \infty\}$, where $\mathcal M_1(\R)$ denotes the measure space of all probability measures on $\R$. Instead of Lemma \ref{lem:emp_proc_F} itself, we now prove the following empirical process version of it.
\begin{lemma}\label{lem:emp_process_perm}
	Let $\mathbb{G}^\pi_{P}$ be a zero-mean Gaussian process on $(l^\infty(\mathcal F))^k$ with covariance function $\mathbf{\Sigma}^\pi_P: (l^\infty(\mathcal F))^k\times (l^\infty(\mathcal F))^k\to \R^{k\times k}$, where for $f=(f_1,\ldots,f_k),g=(g_1,\ldots,g_k)\in (l^\infty(\mathcal F))^k$
	\begin{align}\label{eqn:gammaij}
	(\mathbf{\Sigma}^\pi_P(f,g))_{ij}= \Bigl( \frac{1}{\kappa_i}\mathbf{1}\{i=j\} - 1  \Bigr) P (f_i-Pf_i)(g_j-Pg_j).
	\end{align}
	Then given the observations we have almost surely:
	\begin{align}\label{eqn:lem_NAE_perm_main_state}
	n^{1/2}(\mathbb{P}_{1}^\pi-\mathbb{P}, \ldots,\mathbb{P}_{k}^\pi-\mathbb{P}) \overset{\mathrm d}{\longrightarrow} \mathbb{G}^\pi_{P}\quad\text{on }(l^\infty(\mathcal F))^k.
	\end{align}	
\end{lemma}
\begin{remark}\label{rem:perm}
	The statement of Lemma \ref{lem:emp_process_perm} is not restricted to the specific scenario considered here but is valid as long as $\mathcal F$ has a square $P$-integrable envelope function, i.e. $|f| \leq \widetilde F$ for all $f\in \mathcal F$ and $\int \widetilde F^2 \,\mathrm{ d }P < \infty$.
\end{remark}
\begin{proof}[Proof]
	There are different preservation results for VC, Glivenko-Cantelli and Donsker classes. Combining some of them, e.g., Theorem 3 of \cite{vaartWellner2000} and Problem 2.4.3 of \cite{vaartWellner1996}, ensures that
	\begin{align}\label{eqn:def_G}
	\mathcal{G}= \Bigl\{  \lambda_{1} f_{1} + \lambda_{2}f_{1}f_{2} \ : \lambda_{1},\lambda_2\in [-1,1],\,f_{1},f_2\in\mathcal F \Bigr\}
	\end{align}
	is still a (universal) Glivenko-Cantelli class. In particular, $\sup\{ |\mathbb{P}_{i}g - P_ig|: g\in\mathcal G\} \to 0$ almost surely for all $i=1,\ldots,k$ and, hence,
	\begin{align}\label{eqn:perm_Gliv_Can}
	\sup\{ |\mathbb{P}g - Pg|: g\in\mathcal G\} \to 0 \quad \text{almost surely}.
	\end{align}
	Let $\widetilde F$ be the envelope function mentioned in Remark \ref{rem:perm} (in the present setting, we have $\widetilde F \equiv 1$). It is a straight-forward exercise to show 
	\begin{align}\label{eqn:perm_maxg}
	\frac{1}{n_i}\max\{  \widetilde F(X_{ij})^2 : j=1,\ldots,n_i\} \to 0 \quad \text{almost surely}
	\end{align}
	for all $i=1,\ldots,k$. One proof strategy for \eqref{eqn:perm_maxg} is to follow the three steps: (i) dividing $\widetilde F$ into $\widetilde F_{1,M} = \widetilde F\mathbf{1}\{|\widetilde F|\leq M\}$ and $\widetilde F_{2,M} = \widetilde F \mathbf{1}\{|\widetilde F| > M\}$ for $M\in\N$. (ii) using the inequalities $(a+b)^2 \leq 4a^2+4b^2$ and  $\max_j\widetilde F_{2,M}(X_{ij})^2\leq \sum_j \widetilde F_{2,M}(X_{ij})^2$. (iii) letting first $n\to\infty$ and finally $M\to \infty$.
	
	By imitating the proof of Theorem 3.7.2 from \cite{vaartWellner1996}, we obtain (conditional) distributional convergence of $n^{1/2}(\mathbb{P}_i^\pi-\mathbb{P})$ to  $(1/\kappa_i - 1)^{1/2} \mathbb{\widetilde G}$ on $l^\infty(\mathcal F)$ almost surely given the observations, where $\mathbb{\widetilde G}$ is a $P$-Brownian bridge. From now on, we fix the observations such that this distributional convergences as well as \eqref{eqn:perm_Gliv_Can} and \eqref{eqn:perm_maxg} for all $i=1,\ldots,k$ hold. We can deduce from Lemma 1.3.8 of \cite{vaartWellner1996} that the sequence $(n^{1/2}(\mathbb{P}_i^\pi-\mathbb{P}))_{n\in \N}$ is asymptotically tight for all $i=1,\ldots,k$. Thus, the vector sequence $((n^{1/2}(\mathbb{P}_i^\pi-\mathbb{P}))_{1\leq i \leq k})_{n\in\N}$ is so as well \cite[Lemma 1.4.3]{vaartWellner1996}. Hence, it remains for \eqref{eqn:lem_NAE_perm_main_state} to verify the corresponding marginal convergence \citep[Theorem 1.5.4]{vaartWellner1996}. By a Cr\'{a}mer-Wold argument we can verify this by proving
	\begin{align}\label{eqn:lem_perm_C+wold}
	S_n^\pi=n^{1/2}\sum_{i=1}^k (\mathbb{P}_i^\pi g_i - \mathbb{P}g_i) \overset{\mathrm d}{\longrightarrow} G \sim N\Bigl( 0, \sum_{i,r=1}^k (\mathbf{\Sigma}^\pi_P(g_i,g_r))_{ir}  \Bigr)
	\end{align}
	for every $g_1,\ldots,g_k$ of the shape $g_i = \lambda_{i}f_{i}$ with $\lambda_{i}\in[-1,1]$ and $f_{i}\in\mathcal F$. Let $\Pi = \Pi_n$ be the underlying permutation of the index set $I=\{(i,j):i\in\{1,\ldots,k\},\,1\leq j \leq n_i\}$. Then
	\begin{align*}
	S_n^{\pi} = \sum_{(i,j)\in I} c_n( (i,j), \Pi(i,j))\quad\text{ with } c_n( (i,j),(r,s))= \frac{n^{1/2}}{n_i}( g_i(X_{rs})-\mathbb{P}g_i).
	\end{align*}
	Note that for every fixed $(i,j)\in I$ we have $\sum_{(r,s)\in I} c_n( (i,j), (r,s))=0$. By combining this with Theorems 2 and 3 of \cite{hoeffding1951}  it is sufficient for \eqref{eqn:lem_perm_C+wold} to show
	\begin{align}
	&\max\{ d_n((i,j),(r,s))^2:(i,j),(r,s)\in I\} \to 0\quad  \text{and} \label{eqn:lem_perm_max_d} \\
	& D_n =\frac{1}{n}\sum_{(i,j)\in I}\sum_{(r,s)\in I} d_n((i,j),(r,s))^2 \to \sum_{i,r=1}^k (\mathbf{\Sigma}^\pi_P(g_i,g_r))_{ir}, \label{eqn:lem_perm_sum_d}     
	\end{align}
	where
	\begin{align*}
	&d_n((i,j),(r,s)) = n^{1/2}  \frac{1}{n_i} \Bigl( g_i(X_{rs}) - \mathbb{P}g_i \Bigr) - n^{-1/2} \sum_{t=1}^k\Bigl( g_t(X_{rs}) - \mathbb{P}g_t \Bigr). \nonumber
	\end{align*}
	Note that $g_i^2 \leq \widetilde F^2$. Thus, \eqref{eqn:perm_maxg} implies
	\begin{align*}
	\frac{1}{n_i}\max\{  g_r(X_{ij})^2 : j=1,\ldots,n_i\} \to 0 
	\end{align*}
	for all $i,r=1,\ldots,k$. Combining this, \eqref{eqn:perm_Gliv_Can} and the general inequality $(a+b)^2 \leq 4a^2 + 4b^2$ we can deduce \eqref{eqn:lem_perm_max_d}.	Moreover, we obtain from elementary calculations that
	\begin{align*}
	D_n&= \Bigl[ \sum_{i=1}^k \frac{n}{n_i}   \mathbb P(g_i -\mathbb Pg_i)^2 \Bigr] - \sum_{i=1}^k\sum_{r=1}^k \mathbb P(g_i -\mathbb Pg_i)(g_r -\mathbb Pg_r) \\
	&= \sum_{i=1}^k\sum_{r=1}^k \gamma_n(i,r)\mathbb P(g_i -\mathbb Pg_i)(g_r -\mathbb Pg_r)\text{ with } \gamma_n(i,r)=\frac{n}{n_i} \mathbf{1}\{i=r\} - 1.
	\end{align*}
	Since $g_i,g_ig_r\in\mathcal G$ for all $i,r\in\{1,\ldots,k\}$ we can deduce \eqref{eqn:lem_perm_sum_d} from \eqref{eqn:perm_Gliv_Can}.
\end{proof}

\subsubsection{Proof of Lemma \ref{lem:hada_diff}}

Let $\mathbb{D}$ be the set consisting of all distribution functions $G:\R\to [0,1]$. For every $p\in (0,1)$ we define the corresponding inverse mapping $\Phi_p:\mathbb{D}\to\R$ \citep[compare to Section 3.9.4.2]{vaartWellner1996} by
\begin{align*}
\Phi_p(G) = G^{-1}(p) = \inf\{ t\in\R : G(t)\geq p\}.
\end{align*}
\cite{vaartWellner1996} already proved that $\Phi_p$ is Hadamard differentiable under certain regularity conditions. For our purposes, we need to extend their proof to uniform Hadamard differentiability. 

To shorten the proof, set $h_n^{(1)}= h_n$ and $h_n^{(2)} \equiv 0$ as well as $q_n^{(j)}=\Phi_p(G_n + n^{-1/2}h_n^{(j)})$ $(j=1,2)$. Since $h$ is bounded, we have $|h_n^{(j)}|\leq M_1$ for some $M_1>0$ and all sufficiently large $n\in\N$. Having subsequences in mind, we can suppose without loss of generality that for some $L_j \in\R\cup \{-\infty,\infty\}$
\begin{align*}
n^{1/2}( q_n^{(j)} - q) \to L_j .
\end{align*}
Let $\delta>0$ be arbitrary but $\delta\notin\{ |L_1 - L_2|, |L_1|,|L_2|\}$. By the definition of the inverse functional we have
\begin{align}\label{eqn:lem:hada_G+h<p<}
(G_n +  n^{-1/2}h_n^{(j)})( q_n^{(j)} - n^{-1/2}\delta) \leq p \leq (G_n +  n^{-1/2}h_n^{(j)})( q_n^{(j)}).
\end{align}
Combining this and \eqref{eqn:lem:hada:cond_Gn-G} yields
\begin{align}
&G ( q_n^{(j)} - n^{-1/2}\delta) \leq p + n^{-1/2}(M+M_1),\label{eqn:lem:hada_G<}\\
&G ( q_n^{(j)} ) \geq p - n^{-1/2}(M+M_1). \label{eqn:lem:hada_G>}
\end{align}
The remaining proof is divided into three steps. We will show: 1. $q_n^{(j)}\to q$, 2. $L_j\in\R$, 3. \eqref{eqn:lem:hada_result} holds.\\

1. Since $G$ is strictly increasing in every small enough neighborhood $(q-\eta,q+\eta)$, $\eta>0$, around $q$ we can deduce from \eqref{eqn:lem:hada_G<} and \eqref{eqn:lem:hada_G>} that for every $\eta>0$
\begin{align*}
q_n^{(j)} - n^{-1/2}\delta  \leq q + \eta \quad\text{and}\quad q_n^{(j)}   \geq q - \eta
\end{align*}
for all sufficiently large $n\in\N$. Letting $\eta$ tend to $0$ proves $q_n^{(j)}\to q$.\\

2. Since $\delta\neq |L_j|$ we have $q_n^{(j)} - q \pm n^{-1/2}\delta\neq 0$ for sufficiently large $n\in\N$. Rewriting \eqref{eqn:lem:hada_G<} and \eqref{eqn:lem:hada_G>} gives us
\begin{align*}
&n^{1/2}( q_n^{(j)} - n^{-1/2}\delta - q)\frac{ G ( q_n^{(j)} - n^{-1/2}\delta) - G(q) }{ q_n^{(j)} - n^{-1/2}\delta - q }  \leq  (M+M_1)\\
\text{and}\quad
&n^{1/2}( q_n^{(j)}  - q)\frac{ G ( q_n^{(j)} ) - G(q) }{ q_n^{(j)} - q }  \geq - (M+M_1).
\end{align*}
Consequently, combining these inequalities with the differentiability of $G$ we obtain that
\begin{align*}
\frac{M+M_1}{g(q)} + \delta \geq \lim_{n\to\infty} n^{1/2}| q_n^{(j)} - q| = |L_j|.
\end{align*}

3. First, observe that $\delta\neq |L_1 - L_2|$ implies
\begin{align}
& \Bigl| q_n^{(1)} - q_n^{(2)} \pm n^{-1/2}\delta\Bigr| \to | L_1 - L_2 \pm \delta| \neq 0,  \label{eqn:q1-q2_pm_delta}
\\
&\lim_{n\to\infty}\frac{n^{1/2}(q_n^{(2)}-q)}{n^{1/2}( q_n^{(1)} - n^{-1/2}\delta - q)} = \frac{L_2 }{ L_{1} - \delta} \neq 1,  \label{eqn:lem:hada:delta1} \\
&\lim_{n\to\infty}\frac{n^{1/2}(q_n^{(2)} - n^{-1/2}\delta - q)}{n^{1/2}( q_n^{(1)} - q)} = \frac{L_2 - \delta}{L_1} \neq 1. \label{eqn:lem:hada:delta2}
\end{align}
Due to the result of the second step, we can make use of \eqref{eqn:lem:hada:cond_Gn-Gn-G+G} with $K>\max\{|L_1|,|L_2|,\delta\}$ for $x = q_n^{(j)}$ as well as for $x = q_n^{(j)}\pm n^{-1/2}\delta$. Combining this with \eqref{eqn:q1-q2_pm_delta} as well as  the first and the second inequality from \eqref{eqn:lem:hada_G+h<p<}  for $j=1$ and for $j=2$, respectively, we can deduce that
\begin{align}\label{eqn:lem:had_0>...}
-h(q) &= \sqrt{n}(p-p) - h(q)\nonumber \\
&\geq \sqrt{n}(G_n +  n^{-1/2}h_n)( q_n^{(1)} - n^{-1/2}\delta) -  \sqrt{n}G_n ( q_n^{(2)})\nonumber \\
& = \sqrt{n}\Bigl( q_n^{(1)} - \frac{\delta}{\sqrt{n}} - q_n^{(2)} \Bigr)\frac{G(q_n^{(1)} - \delta/\sqrt{n}) - G(q_n^{(2)}) + o(n^{-1/2})}{ q_n^{(1)} - n^{-1/2}\delta - q_n^{(2)} } + o(1) \nonumber \\
& = \sqrt{n}\Bigl( q_n^{(1)} - \frac{\delta}{\sqrt{n}} - q_n^{(2)} \Bigr)\Bigl( \frac{G(q_n^{(1)} - \delta/\sqrt{n}) - G(q_n^{(2)}) }{ q_n^{(1)} - n^{-1/2}\delta - q_n^{(2)} } + o(1) \Bigr) + o(1).
\end{align}
Applying now the first inequality from \eqref{eqn:lem:hada_G+h<p<} for $j=2$ and the second one for $j=1$  we obtain analogously
\begin{align}\label{eqn:lem:had_0<...}
-h(q) \leq  \sqrt{n}\Bigl( q_n^{(1)} +  \frac{\delta}{\sqrt{n}} - q_n^{(2)} \Bigr)\Bigl[ \frac{G(q_n^{(1)} ) - G(q_n^{(2)} - n^{-1/2}\delta) }{ q_n^{(1)} + n^{-1/2}\delta - q_n^{(2)} } + o(1) \Bigr] + o(1).
\end{align}
Due to \eqref{eqn:lem:hada:delta1} and \eqref{eqn:lem:hada:delta2} we can apply the upcoming Lemma \ref{lem:analytic_derivative} to verify that the fractions in \eqref{eqn:lem:had_0>...} and \eqref{eqn:lem:had_0<...}, respectively, converge to $g(q)$. Altogether, 
\begin{align*}
-\frac{h(q)}{g(q)} - \delta   \leq \liminf_{n\to \infty} \sqrt{n}\Bigl( q_n^{(1)} - q_n^{(2)} \Bigr) \leq \limsup_{n\to \infty} \sqrt{n}\Bigl( q_n^{(1)} - q_n^{(2)} \Bigr)  \leq -\frac{h(q)}{g(q)} + \delta.
\end{align*}
Finally, letting $\delta$ tend to $0$ completes the proof.

\begin{lemma}\label{lem:analytic_derivative}
	Let $G$ be differentiable at $u$ with derivative $g$. Let $(\delta_{n,1})_{n\in\N}$, $(\delta_{n,2})_{n\in\N}$ be sequences in $\R$ converging to $0$ with $\limsup_{n\to\N}(\delta_{n,1}/\delta_{n,2})< 1$ or $\liminf_{n\to\N}(\delta_{n,1}/\delta_{n,2}) > 1$, where the convention $x/0=\infty$ for $x> 0$ is used. Then
	\begin{align*}
	\frac{G( u + \delta_{n,2}) - G( u + \delta_{n,1})}{ \delta_{n,2} - \delta_{n,1} } \to g(u).
	\end{align*}
\end{lemma}
\begin{proof}[Proof]
	By symmetry, we just need to consider $\limsup_{n\in\N}(\delta_{n,1}/\delta_{n,2})< 1$. Note that the statement follows obviously from the differentiability of $G$ if $\delta_{n,j}\equiv 0$ for $j=1$ or $j=2$. Having classical subsequence arguments in mind, we can assume that $\delta_{n,1},\delta_{n,2}\neq 0$, $\delta_{n,1}/\delta_{n,2} \to M\in[-\infty,1)$, where $M = -\infty$ is allowed. Observe that
	\begin{align}\label{eqn:lem:analy_der_proof}
	&\frac{G( u + \delta_{n,2}) - G( u + \delta_{n,1})}{ \delta_{n,2} - \delta_{n,1} } \nonumber \\
	&=
	\frac{G( u + \delta_{n,2}) - G( u )}{\delta_{n,2}} \frac{\delta_{n,2}}{ \delta_{n,2} - \delta_{n,1} } -
	\frac{G( u + \delta_{n,1}) - G( u )}{\delta_{n,1}} \frac{\delta_{n,1}}{ \delta_{n,2} - \delta_{n,1} }.
	\end{align}
	If $M= - \infty$ then $\delta_{n,2}/( \delta_{n,2} - \delta_{n,1} ) \to 0$ and $\delta_{n,1}/(\delta_{n,2} - \delta_{n,1}) \to 1$. Otherwise, i.e., if $M\in(-\infty,1)$, then $\delta_{n,2}/( \delta_{n,2} - \delta_{n,1} ) \to (1 - M )^{-1}$ and $\delta_{n,1}/(\delta_{n,2} - \delta_{n,1}) \to   (1 - M )^{-1} - 1$. Combining both cases with \eqref{eqn:lem:analy_der_proof} and the differentiability of $G$ proves the statement.
\end{proof}

\subsection{Proof of Theorem \ref{theo:local_alt}}
Here, we consider the triangular array $X_{nij}$ from Section \ref{sec:local_alt} fulling Assumption \ref{ass:local_alt}. The following two lemmas, which extend \eqref{eqn:process_conv_F_uncon} and Lemma \ref{lem:emp_proc_F}, are the key steps to derive the desired statement in Theorem \ref{theo:local_alt}. Their proofs can be found subsequently.
\begin{lemma}\label{lem:local_conv_q}
	We have
	\begin{align}
	n_i^{1/2}(\widehat F_{i} - F_{ni}) \overset{\mathrm d}{\longrightarrow} B\circ F_i \text{ on }D(\R),
	\end{align}
	where $B$ is a Brownian bridge on $[0,1]$.
\end{lemma}
\begin{lemma}\label{lem:local_emp_proc_F}
	Let $\mathbb{G}=(\mathbb{G}_1,\ldots,\mathbb{G}_k)$ be given as in \eqref{lem:emp_proc_F}.  Then
	\begin{align}\label{eqn:local_lem:emp_proc_F:result}
	n^{1/2} ( \widehat F^\pi_{i} - \widehat F )_{i=1,\ldots,k} \overset{\mathrm d}{\longrightarrow} \mathbb{G}\quad \text{on }D(\R)^k
	\end{align}
	given the observations in probability.
\end{lemma}
Since $f_{ni}$ converges uniformly to $f_i$ in a neighborhood of $q_{ir}$, a continuity point of $f_i$, we can deduce from the mean value theory that
\eqref{eqn:lem:hada:cond_Gn-Gn-G+G} holds for $q= q_{ir}$, $G_n = F_{ni}$ and $G=F_i$. Due to this, Assumption \ref{ass:local_alt}\eqref{enu:ass:al_alt_F} and Lemma \ref{lem:hada_diff} we can follow the proof argumentation for Proposition \ref{prop:uncond}, while applying this time the uniform functional $\delta$-method \citep[Theorem 3.9.5]{vaartWellner1996}, to obtain:
\begin{align*}
\sqrt{n} \mathbf{T} \mathbf{\widehat q} = \sqrt{n} \mathbf{T} (\mathbf{\widehat q} - \mathbf{q}_n) + \sqrt{n} \mathbf{T}  \mathbf{q}_n \overset{\mathrm d}{\longrightarrow} \mathbf{Y} + \boldsymbol{\theta} \sim N(\boldsymbol{\theta}, \boldsymbol{T}\boldsymbol{\Sigma}\mathbf{T}'),
\end{align*}
where $\mathbf{Y}$ and $\boldsymbol{\Sigma}$ are defined as in the paragraph below Proposition \ref{prop:uncond}. The extension of the covariance matrix estimators' consistency, i.e., Lemmas \ref{lem:kern_density_estimator}--\ref{lem:PB_estimator}, to the present local alternatives is straightforward and thus left to the reader. Finally, $S_n(\mathbf{T})$ converges in distribution to a non-central $\chi^2_{\text{rank}(\mathbf{T})}(\delta)$ with non-centrality parameter $\delta = \boldsymbol{\theta}'( \mathbf{T} \mathbf{\Sigma} \mathbf{T})^+ \boldsymbol{\theta}$; that proves the statement about the asymptotic test's power

As stated in Lemma \ref{lem:local_emp_proc_F}, considering the triangular array $X_{nij}$ instead of $X_{ij}$ does not affect the (conditional) convergence of the empirical distribution functions. Hence, it is not surprising that the same is true for the empirical quantiles. To prove the latter, we can follow the argumentation for Theorem \ref{theo:stat_perm}, the only detail which need more clarification is \eqref{eqn:sup_Fni(q+x)}. For this purpose, we want to remind that Assumption \ref{ass:local_alt}\eqref{enu:ass:local_alt_fi_perm} and the mean value theorem implies for every $K>0$
\begin{align*}
\sqrt{n}\sup_{|x|\leq K/\sqrt{n}} \Bigl |  F_{ni}(q+x) -  F_{ni}(q) - F_i(q+x) + F_i(q)\Bigr| \to 0.
\end{align*}
Combining this with the arguments of \cite{bahadur1966} for his Lemma 1 we can deduce \eqref{eqn:sup_Fni(q+x)}. Consequently, we obtain (conditional) convergence \eqref{enu_quantil_conv_perm} of the permutation quantiles given the observations in probability. We want to remind the reader that we can always turn to subsequences to get almost sure convergence instead of convergence in probability. While almost sure convergence is nice to have for the proofs, convergence in probability is usually enough for statistical purposes, as it is in the present situation. Due to Lemma \ref{lem:local_emp_proc_F}, all arguments in the proofs for Lemmas \ref{lem:est_f_perm_consis} and \ref{lem:est_f_kern_perm_consis} are still valid for the underlying local alternatives. In particular, the permutation covariance matrix estimators converge, given the data in probability, to the correct limit. Consequently, \eqref{eqn:theo:stat:perm} holds also for the present local alternatives, given the observations in probability, completing the proof. 

\subsubsection{Proof of Lemma \ref{lem:local_conv_q}}

We again use empirical theory, as already done for the proof of Lemma \ref{lem:emp_proc_F}. Since we discuss here the triangular arrays, we add an index to all introduced empirical measures: ${\mathbb{P}_{ni}}=n_i^{-1}\sum_{j=1}^{n_i}\epsilon_{X_{nij}}$, ${\mathbb{P}_{ni}^\pi}=n_i^{-1}\sum_{j=1}^{n_i}\epsilon_{X_{nij}^\pi}$ and ${\mathbb{P}_{n}}= n^{-1}\sum_{i=1}^k\sum_{j=1}^{n_i}\epsilon_{X_{nij}}$. Moreover, we denote by $P_{ni}$ the distribution of $X_{ni1}$ and by $P_i$ the distribution corresponding to $F_i$. Again, we index all these (empirical) measures by the function class $\mathcal F$.

In the classical sequence situation $X_{nij}=X_{ij}$, we can deduce from $\mathcal F$ being a Donsker class that
\begin{align}\label{eqn:emp_proc_conv_Pni-Pi}
\sqrt{n_i} (\mathbb{P}_{ni} - P_i) \overset{\mathrm d}{\longrightarrow} \mathbb{Z}_i \quad \text{on }l^\infty(\mathcal F),
\end{align}
where $\mathbb{Z}_i$ is a $P_i$-Brownian bridge. Note that $\widehat F_{ni}(t) = \int\mathbf{1}_{(-\infty,t]}\,\mathrm{ d }\mathbb{P}_{ni}$ and, thus,  \eqref{eqn:emp_proc_conv_Pni-Pi} implies distributional convergence of the empirical distribution function. In their Section 2.8.3, \cite{vaartWellner1996} discussed conditions, under which the aforementioned empirical process convergence hold even for triangular arrays. To explicitly state these conditions here, it would require to introduce too much notation. That is why we just explain how the conditions can be justified. From Assumption \ref{ass:local_alt}\label{enu:ass:local_alt_F} and the continuity of $F_i$ we can deduce that $F_{ni}$ converges uniformly to $F_i$ and, thus, (2.8.5) of \cite{vaartWellner1996} holds. The underlying function class $\mathcal F$ has the constant envelope function $G\equiv 1$, i.e. $|f(x)|\leq 1= G(x)$ for all $f\in\mathcal F$. This implies (2.8.6) of \cite{vaartWellner1996}. Consequently, we can apply Theorem 2.8.10 of \cite{vaartWellner1996}; note that the condition therein about the bracketing number follows directly from their Examples 2.5.4 and 2.5.7. Finally,
\begin{align}\label{eqn:emp_proc_conv_Pni-Pni}
\sqrt{n_i} (\mathbb{P}_{ni} - P_{ni}) \overset{\mathrm d}{\longrightarrow} \mathbb{Z}_i \quad \text{on }l^\infty(\mathcal F),
\end{align}
which, in particular, proves  Lemma \ref{lem:local_conv_q}.

\subsubsection{Proof of Lemma \ref{lem:local_emp_proc_F}}
We adapt the notation from the previous proof. Instead of Lemma \ref{lem:local_emp_proc_F}, we prove the empirical process version of it.
\begin{lemma}\label{lem:emp_proc_perm_local_alt}
	Let $\mathbb{G}^\pi_{P}$ be the zero-mean Gaussian process introduced in Lemma \ref{lem:emp_process_perm}. Then given the observations we have in probability:
	\begin{align*}
	n^{1/2}(\mathbb{P}_{n1}^\pi-\mathbb{P}_n, \ldots,\mathbb{P}_{nk}^\pi-\mathbb{P}_n) \overset{\mathrm d}{\longrightarrow} \mathbb{G}^\pi_{P}\quad\text{on }(l^\infty(\mathcal F))^k.
	\end{align*}	
\end{lemma}
\begin{proof}[Proof]
	By \eqref{eqn:emp_proc_conv_Pni-Pni}, Assumption \ref{ass:local_alt}\eqref{enu:ass:al_alt_F} and the continuity of $F_i$, we obtain
	\begin{align*}
	\sup\{ |\mathbb{P}_{ni}f - P_if|: f\in\mathcal F\} = \sup\{ |\widehat F_{ni}(t) - F_i(t)|: t\in\R\} \overset{p}{\rightarrow} 0.
	\end{align*}
	Since $|f|\leq 1$ for all $f\in \mathcal F$ it is easy to see that the aforementioned convergence is still true for $\mathcal F$ replaced by $\mathcal G$ from \eqref{eqn:def_G}. Consequently, we can deduce that in probability
	\begin{align}\label{eqn:perm_Gliv_Can_triangular_arrays}
	\sup\{ |\mathbb{P}_{n} g - Pg|: g\in\mathcal G\} \to 0 .
	\end{align}
	Turning to subsequences, we can assume that \eqref{eqn:perm_Gliv_Can_triangular_arrays} even holds with probability one. Hence, the marginal convergence, given the data, follows as in the proof of Lemma \ref{lem:emp_process_perm}. Consequently, it remains to prove the asymptotic tightness of $(n^{1/2}(\mathbb{P}_i^\pi-\mathbb{P}_n))_{n\in\N}$ given the data, or equivalently uniform equicontinuity \citep[Theorem 1.5.7]{vaartWellner1996}. In the situation of the previous proof, \cite{vaartWellner1996} verified the uniform equicontinuity by combining several inequalities and the unconditional multiplier Theorem 2.9.2, see their proof of Theorem 3.7.1. Note that Theorem 2.9.2 is, in its current version, just valid for the usual setting $X_{nij}=X_{ij}$ and not for general triangular arrays as needed here. But we just need the uniform equicontinuity result from the proof of Theorem 2.9.2, for which again different inequalities were combined. All inequalities from the proofs of Theorem 2.9.2 and 3.7.1, namely Proposition A.1.9 (Hoeffding inequality), Lemmas 3.6.6, 2.9.2 and 2.3.6 (we ordered these inequalities in the order they are needed for the proof) can be directly applied in our more general situation. Finally, the desired equicontinuity can be deduced from the equicontinuity of the processes $(\sqrt{n_i}(\mathbb{P}_{ni} - P_{ni}))_{n\in\N}$, where the latter is an immediate consequence of the process convergence \eqref{eqn:emp_proc_conv_Pni-Pni}.
\end{proof}

\section*{Acknowledgement}
The authors thank the COHORT investigators \citep{richter:ETAL:2011} for providing us their data, which were collected in 5 different studies \citep{study_pelotas,study_adair,study_southafrica,study_INCAP,study_INDIA}. Here, we are especially grateful to Linda Richter for helping us with the communication between all sites. The work of Marc Ditzhaus and Markus Pauly was funded by the \textit{Deutsche Forschungsgemeinschaft} (grant no.  PA-2409 5-1).

\bibliographystyle{plainnat}
\bibliography{sample}

\end{document}